\documentclass[11pt]{article}

\usepackage{graphicx}
\usepackage{epstopdf}
\usepackage{amsmath,amssymb}
\usepackage{subfig}
\usepackage{multirow}
\usepackage{alltt}
\usepackage{color}
\usepackage{longtable}
\usepackage{array, caption}
\usepackage[backref=page,colorlinks]{hyperref}
\hypersetup{colorlinks = true, citecolor = blue, urlcolor = blue}

\newcommand{\E}{{\mathbb E}}
\newcommand{\PP}{{\mathbb P}}
\newcommand{\R}{{\mathbb R}}

\newcommand{\bes}{\begin{equation*}}
\newcommand{\ees}{\end{equation*}}
\usepackage{amsmath,amssymb,amsfonts, amsthm, amsbsy, mathtools, epsfig}
\usepackage[ruled,vlined]{algorithm2e}
\usepackage{todonotes}

\newcommand{\Prob}{{\mathbb P}}

\newcommand{\xiX}{\xi^X}
\newcommand{\KL}{\text{KL}}
\newcommand{\xiY}{\xi^Y}

\newtheorem{thm}{Theorem}[section]

\newtheorem{assumption}{Assumption}
\newtheorem{remark}[thm]{Remark}

\newtheorem{theorem}{Theorem}[section]
\newtheorem{lemma}[theorem]{Lemma}

\begin{document}
\title{Replica Exchange for Non-Convex Optimization}
\author{Jing Dong\footnote{Columbia University, Email: jing.dong@gsb.columbia.edu}~ and Xin T. Tong\footnote{Department of Mathematics, National University of Singapore, Email: mattxin@nus.edu.sg}}
\date{}
\maketitle

\begin{abstract}
Gradient descent (GD) is known to converge quickly for convex objective functions, but it can be trapped at local minima.
On the other hand, Langevin dynamics (LD) can explore the state space and find global minima, but in order to give accurate estimates, LD needs to run with a small discretization step size and weak stochastic force, which in general slow down its convergence. This paper shows that these two algorithms 
can ``collaborate" through a simple exchange mechanism, in which they swap their current positions if LD yields a lower objective function. 
This idea can be seen as the singular limit of the replica-exchange technique from the sampling literature. We show that this new algorithm converges to the global minimum linearly with high probability, assuming the objective function is strongly convex in a neighborhood  of the unique global minimum. By replacing gradients with stochastic gradients, and adding a proper threshold to the exchange mechanism, our algorithm can also be used in online settings. We also study non-swapping variants of the algorithm, which achieve similar performance.
We further verify our theoretical results through some numerical experiments, and observe superior performance of the proposed algorithm over running GD or LD alone.\end{abstract}


\section{Introduction}
Division of labor is the secret of any efficient enterprises. 
By collaborating with individuals with different skillsets, 
we can focus on tasks within our own expertise and produce better outcomes than working independently. 
This paper asks whether the same principle can be applied when designing an algorithm. 

Given a general smooth non-convex objective function $F$, we consider the unconstrained optimization problem
\begin{equation}\label{eq:main}
\min_{x\in \R^d}{F(x)}.
\end{equation}
It is well-known that deterministic optimization algorithms, such as gradient descent (GD), can converge to a local minimum quickly \cite{Nes13}. 
However, this local minimum may not be the global minimum, and GD will be trapped there afterwards.  On the other hand, sampling-based algorithms, such as Langevin dynamics (LD) can escape local minima by their stochasticity, but the additional stochastic noise contaminates the optimization results and slows down the convergence when the iterate is near the global minimum. 
More generally, deterministic algorithms are mostly designed to  find local minima quickly, but they can be terrible at exploration. Sampling-based algorithms are better suited for exploring the state space, but they are inefficient when pinpointing the local minima. This paper investigates how they can ``collaborate" to get the ``best of both worlds".

%
%

The collaboration mechanism we introduce stems from the idea of replica exchange in the sampling literature. Its implementation is very simple: we run a copy of GD, denoted by $X_n$; and a copy of discretized LD, denoted by $Y_n$, in parallel. At each iteration, if $F(X_n)>F(Y_n)$, we swap their positions. At the final iteration, we output $X_N$. The proposed algorithm, denoted by GDxLD and formalized in Algorithm \ref{alg:GDxLD} below, enjoys the ``expertise" of both GD and LD. In particular, we establish that if $F$ is convex in a neighborhood of the unique global minimum, then, for any $\epsilon>0$, there exists $N(\epsilon)=O(\epsilon^{-1})$, such that for $N\geq N(\epsilon)$, $|F(X_N)-F(x^*)|<\epsilon$ with high probability, where $x^*$ is the global minimum. If $F$ is strongly convex in the same neighborhood, we can further obtain linear convergence; that is, $N(\epsilon)$ can be reduced to $O(\log\tfrac1\epsilon)$. 

As we will demonstrate with more details in Section \ref{sec:continuous}, GDxLD can be seen as the singular limit of a popular sampling algorithm, known as replica exchange or parallel tempering. The exchange mechanism is in place to make the sampling algorithm a reversible MCMC. However, for the purpose of optimization, exchanging the locations of GD and LD is not the only option. In fact, both GD and LD can obtain the location of $Y_n$ if $F(X_n)>F(Y_n)$. This leads to a slightly different version of the algorithm, which we will denote as nGDxLD where ``n" stands for ``non-swapping". Since the LD part will not be swapped to the location of the GD, it is a bona-fide Langevin diffusion algorithm. In terms of performance, we establish the same complexity estimates for both GDxLD and nGDxLD. Our numerical experiments also verify that GDxLD and nGDxLD have similar performance. To simply the notation, we write (n)GDxLD when we refer to both versions of the algorithm.

Notably, the complexity bounds we establish here are the same as the complexity bounds for standard GD when $F$ is globally convex (or strongly convex) \cite{Nes13}, but we only need $F$ to be convex ( or strongly convex) near $x^*$, which is significantly weaker. It is not difficult to see intuitively why (n)GDxLD works well in such non-convex settings. The LD explores the state space and visits the neighborhood of the global minimum. Since this neighborhood is of a constant size, it can be found by the LD in constant time. Moreover, this neighborhood gives lower values of $F$ than anywhere else, so the GD will be swapped there, if it is not there already. Finally, GD can pinpoint the global minimum as it now starts in the right neighborhood. Figure \ref{fig:concept} provides a visualization of the mechanism. 

\begin{figure}
\centering{
\includegraphics[width=7cm]{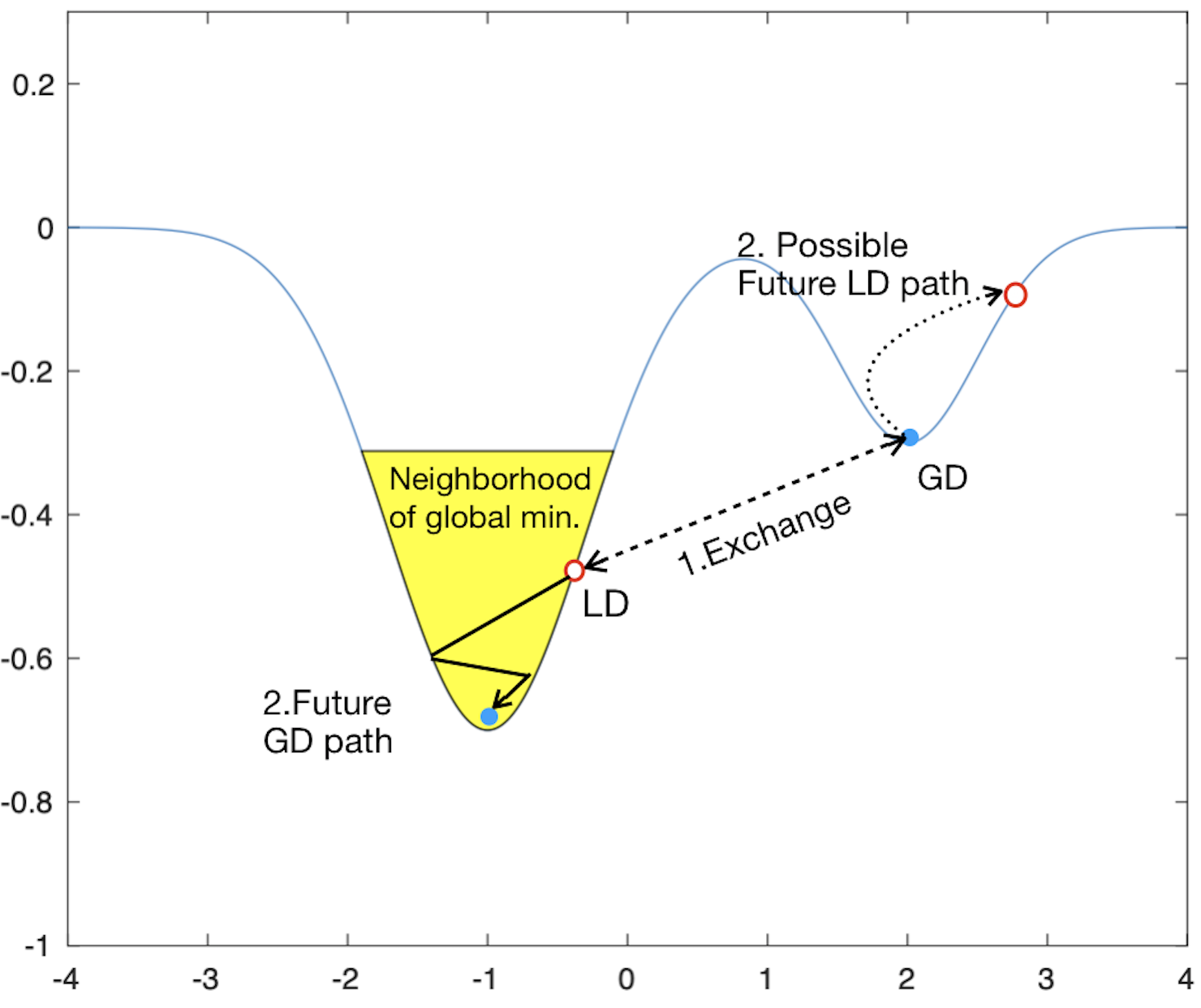}
\includegraphics[width=7.2cm]{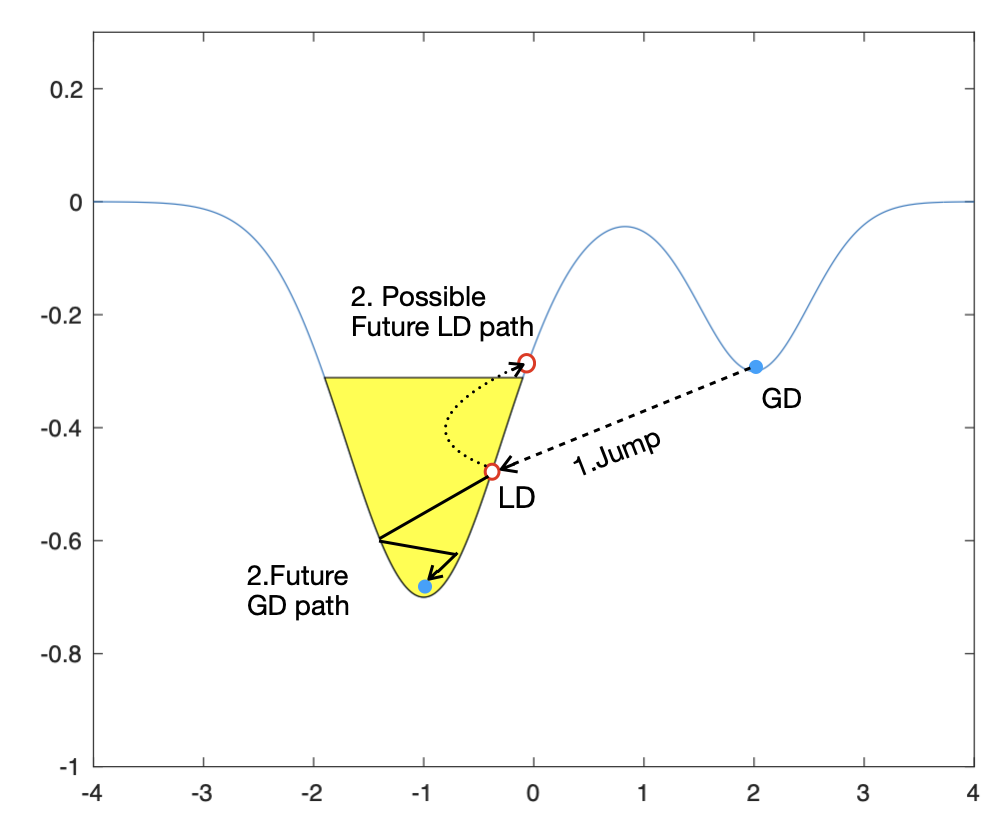}
}
\caption{(Left) GDxLD in action: when LD enters a neighborhood of the global minimum (yellow region), exchange happens and helps GD to escape the local minimum that traps it. After exchange, GD can converge to the global minimum, and LD keeps exploring. (Right) nGDxLD is similar to GDxLD; the only difference is that LD will stay at its original location, instead of swapping with  GD. }
\label{fig:concept}
\end{figure}

For many modern data-driven applications, $F$ is an empirical loss function, so its precise evaluation and its gradient evaluation can be computationally expensive. For these scenarios, we also consider an online modification of our algorithm. The natural modification of GD is stochastic gradient decent (SGD), and the modification of LD is stochastic gradient Langevin dynamics (SGLD). 
This algorithm, denoted by SGDxSGLD and formalized in Algorithm \ref{alg:SGDxSGLD} below, achieves a similar complexity bound as GDxLD if $F$ is strongly convex in the neighborhood of the unique global minimum $x^*$. For the theory to apply, we also need the noise of the stochastic gradient to be sub-Gaussian and is of order $O(\sqrt{\epsilon})$. This assumption can often be met by using a mini-batch of size $O(\epsilon^{-1})$, which in principle should be factored in for complexity. In this case, the overall complexity  of SGDxSGLD is $O(\epsilon^{-1}\log \tfrac1\epsilon)$. Similar to the offline scenario, the non-swapping variation, nSGDxSGLD, can be obtained by keeping the SGLD part not swapped to the location of the SGD. 

\subsection{Related work}
Non-convex optimization problems arise in numerous advanced machine learning, statistical learning, and structural estimation settings \cite{Geyer:1995, Bottou:2018}.
How to design efficient algorithms with convergence guarantees has been an active area of research due to their practical importance. In what follows, we discuss some existing results related to our work. As this is a fast growing and expanding field, we focus mostly on algorithms related to GD, LD, or their stochastic gradient versions. 

One main approach to study nonlinear optimization is to ask when an algorithm can find local minima efficiently. 
The motivation behind finding local minima, 
rather than global minima, is that in some machine learning problems, such as matrix completion and wide neural networks, local minima already have good statistical properties or prediction power \cite{ge2017no,park2017non,bhojanapalli2016global,du2017spurious,ge2017learning,mei2017solving}.
Moreover, the capability to find local minima or second-order stationary points (SOSP) is nontrivial, since GD can be trapped at saddle points. 
When full Hessian information is available, this can be done through algorithms such as cubic-regularization or trust region \cite{nesterov2006cubic, curtis2014trust}. 
If only gradient is available, one general idea is to add stochastic noise so that the algorithms can escape saddle points \cite{jin2017escape, allen2017natasha,jin2017accelerated,yu2017third,daneshmand2018escaping,jin2019nonconvex}.
But the ``size" of the noise and the step size need to be tuned based on the accuracy requirement. 
This, on the other hand, reduces the learning rate and the speed of escaping saddle points. 
For example, to find an $\epsilon$-accurate SOSP in the offline setting, the perturbed gradient descent method requires $O(\epsilon^{-2})$ iterations  \cite{jin2017escape}, and in the online setting, it requires $O(\epsilon^{-4})$ iterations \cite{jin2019nonconvex}. 
These  convergence rates are slower than the ones of our proposed algorithms. 

For problems in which local minima are not good enough, we often need to use sampling-based algorithms to find global minima. 
This often involves simulating an ergodic stochastic process for which the invariant measure is proportional to $\exp(-\frac{1}{\gamma} F(x))$, where $\gamma$ is referred to as the ``temperature", which often controls the strength of stochasticity.
Because the process is ergodic, it can explore the state space. However, for the invariant measure to concentrate around the global minimum, 
$\gamma$ needs to be small. Then for these sampling based-algorithms to find accurate approximation of global minima, they need to use weaker stochastic noise or smaller step sizes, which in general slow down the convergence. 
For LD in the offline setting and SGLD in the online setting, the complexity is first studied in \cite{raginsky2017non}, later improved by \cite{xu2017global}, and generalized by \cite{CDT19} to settings with decreasing step sizes. In \cite{xu2017global}, it is shown that LD can find an $\epsilon$-accurate global minimum  in $O(\epsilon^{-1})$ iterations, and SGLD can do so in $O(\epsilon^{-5})$ iterations. These algorithms have higher complexity than the ones we proposed in this paper.   Note that in \cite{xu2017global}, it keeps the temperature at a constant order, in which case the step size needs to scale with $\epsilon$. As we will discuss in more details in Section \ref{sec:main}, the algorithm we propose keeps both the temperature and the step size at constant values.

 We also comment that sampling-based algorithms may lead to better dependence on the dimension of the problem for some MCMC problems as discussed in \cite{ma2019sampling}. However, if the goal is to find global minima of a Lipschitz-smooth non-convex optimization problem, the complexity in general has exponential dependence on the dimension. This has to do with the spectral gap of the sampling process and has been extensively discussed in \cite{ma2019sampling, raginsky2017non, xu2017global}. Due to this, our developments in this paper focus on settings where the dimension of the problem is fixed, and we characterize how the complexity scales with the precision level $\epsilon$.

Aside from optimization, LD and related Markov Chain Monte Carlo (MCMC) algorithms are also one of the main workhorses for Bayesian statistics. Our work is closely related to the growing literature on convergence rate analysis for LD.
Asymptotic convergence of discretized LD with decreasing temperature (diffusion coefficient) and step sizes is extensively studied in the literature (see, e.g., \cite{Gidas:1985, Pelletier:1998}).
Nonasymptotic performance bounds for discrete-time Langevin updates in the case of convex objective functions are studied in \cite{Durmus:2017,Dalalyan:2017,Cheng:2018,Dalalyan:2019}. 
In MCMC, the goal is to sample from the stationary distribution of LD, and thus the performance is often measured by the Wasserstein distance or the Kullback-Leibler divergence between the finite-time distribution and the target invariant measure \cite{Ma:2015, Bubeck:2018}. In this paper, we use the performance metric  $\Prob(|F(X_N)-F(x^*)|>\epsilon)$, which is more suitable for the goal of optimizing a non-convex function. This also leads to a very different framework of analysis compared to the existing literature.

\subsection{Main contribution}
The main message of this paper is that, by combining a sampling algorithm with an optimization algorithm, we can create a better tool for some non-convex optimization problems. From the practical point of view, the new algorithms have the following advantages:

\begin{itemize}
\item When the exact gradient $\nabla F$ is accessible, we propose the GDxLD algorithm and its non-swapping variant, nGDxLD. Their implementation does not require the step size $h$ or temperature  $\gamma$ to change with the precision level $\epsilon$. Such independence is important in practice, since tuning hyper-parameters is in general difficult. As we will demonstrate in Section \ref{sec:num_sen}, our algorithm is quite robust to a wide range of hyper-parameters.
\item When the dimension is fixed, for a given precision level $\epsilon$,  (n)GDxLD beats existing algorithms in computational cost. In particular, we show (n)GDxLD can reach approximate global minimum with a high probability in $O(\log \frac{1}{\epsilon})$ iterations. 
Comparing to the iteration estimate of LD in \cite{xu2017global}, which is $O(\epsilon^{-1}\log \frac{1}{\epsilon})$, our algorithm is much more efficient,
but we require the extra assumption that $F$ has a unique global minimum $x^*$, and $F$ is strongly convex near $x^*$.
\item  When only stochastic approximation of $\nabla F$ is available, we propose the SGDxSGLD algorithm and its non-swapping variant, nSGDxSGLD. Like (n)GDxLD, their implementation does not require the temperature or the step size to change with the precision level. 
(n)SGDxSGLD is also more efficient when comparing with other online optimization methods using stochastic gradients. In particular, we show (n)SGDxSGLD can reach approximate global minimum with high probability with an $O(\epsilon^{-1}\log \frac{1}{\epsilon})$ complexity. This is better than the complexity estimate of VR-SGLD in \cite{xu2017global}, which is $O(\epsilon^{-4}(\log \frac{1}{\epsilon})^4)$. 
The additional assumptions we require is that the function evaluation noise is sub-Gaussian, and $F$ is strongly convex near the unique global minimum $x^*$.
\end{itemize}

In term of algorithm design, a novel aspect we introduce is the exchange mechanism. The idea comes from replica-exchange sampling, which has been designed to overcome some of the difficulties associated with rare transitions to escape local minima \cite{Earl:2005, Woodard:2009}.  \cite{Dupuis:2012} uses the large deviation theory to define a rate of convergence for the empirical measure of replica-exchange Langevin diffusion.
It shows that the large deviation rate function is increasing with the exchange intensity, which leads to the development of infinite swapping dynamics. We comment that infinite swapping is not feasible in practice as the actual swapping rate depends on the discretization step-size. The algorithm that we propose attempts a swap at every iteration, which essentially maximizes the practical swapping intensity.
\cite{Chen:2019} extend the idea to solve non-convex optimization problems, but they only discuss the exchange between two LD processes. Our work further extends the idea by combining GD with LD and provides finite-time performance bounds that are tailored to the optimization setting.  

In online settings, the function and gradient evaluations are contaminated with noises. Designing the exchange mechanism in this scenario is more challenging, since we need to avoid and mitigate the possible erroneous exchanges. We demonstrate such a design is feasible by choosing a proper mini-batch size  as long as the function evaluation noise is sub-Gaussian. The analysis of SGDxSGLD is hence much more challenging and nonstandard, when comparing with that of GDxLD. 


While it is a natural idea to let two algorithms specializing in different tasks to collaborate, coming up with a good collaboration mechanism is not straightforward. For example, one may propose to let the LD explore the state-space first, after it finds the region of the global minimum, one would turn-off the temperature, i.e., setting $\gamma=0$, and run the GD. However, in practice, it is hard to come up with a systematic way to check whether the process is in the neighborhood of the global minimum or not. One may also propose to turn down the temperature of the LD gradually. This is the idea of simulated annealing \cite{kirkpatrick1983optimization, mangoubi2018convex}. The challenge there is to come up with a good ``cooling schedule". To ensure convergence to global minima theoretically, the temperature needs to be turned down very slowly, which could jeopardize the efficiency of the algorithm \cite{Gelfand:1991}. For example, \cite{granville1994simulated} shows that for the simulated annealing to converge, the temperature at the $n$-th iteration needs to be of order $1/\log n$.

Readers who are familiar with optimization algorithms might naturally think of doing replica exchange between other deterministic algorithms and other sampling-based algorithms. Standard deterministic algorithms include GD with line search, Newton's method, and heavy-ball methods such as Nestrov acceleration. Sampling-based algorithms include random search, perturbed gradient descent, and particle swarm optimization. We investigate GDxLD instead of other exchange combinations, not because GDxLD is superior in convergence speed, but because of the following two reasons:
\begin{enumerate}
\item GDxLD can be seen as a natural singular limit of replica-exchange LD, which is a mathematical subject with both elegant theoretical properties and  useful  sampling implications.  We will explain the connection between GDxLD and replica-exchange LD in Section \ref{sec:continuous}.
\item GDxLD is very simple to implement. It can be easily adapted to the online setting. We will explain how to do so in Section \ref{sec:online}. 
\end{enumerate}
On the other hand, it would be interesting to see whether exchange between other algorithms can provide faster rate of convergence in theory or in practice. 
We leave this as a future research direction, and think the analysis framework and proving techniques we developed here can be extended to more general settings.  In particular, due to the division of labor, LD is only used to explore the state-space. Therefore, instead of establishing its convergence as in \cite{raginsky2017non, xu2017global}, we only need to prove that it is suitably ergodic, such that there is a positive probability of visiting the neighborhood of the global minimum. GD is only used for exploitation and the complexity only depends on its behavior in the neighborhood of the global minimum. 

Lastly, one of the key element in the complexity of (n)GDxLD is how long it takes LD to find the neighborhood of the global minimum.
High dimensional and multi-modal structures can be difficult to sample using just LD. In the literature, sampling techniques such as dimension reduction,  Laplace approximation, Gibbs-type-partition can be applied to improve the sampling efficiency  for high dimensional problems \cite{hairer2014spectral, cui2014likelihood, morzfeld2019localization, tong2020mala}. As for multi-modality, it is well known that LD may have slow convergence when sampling from mixture of log-concave densities if each component density is close to singular \cite{MS14}. The replica exchange method, discussed in Section \ref{sec:continuous} below, considers  running multiple LDs with different temperatures and exchanging among them \cite{swendsen1986replica}. It has been shown to perform well for multi-modal densities \cite{Woodard:2009,dong2020spectral}. Other than replica exchange, simulated tempering, which considers dynamically changing the temperature of the LD can also handle such multimodality in general \cite{marinari1992simulated, Woodard:2009, lee2018beyond}. Replacing LD with these more advanced sampler may facilitate more efficient exploration when facing more challenging energy landscapes. 

\section{Main results} \label{sec:main}
In this section, we present the main algorithms: GDxLD and nGDxLD (Algorithm \ref{alg:GDxLD}), SGDxSGLD  and nSGDxSGLD (Algorithm \ref{alg:SGDxSGLD}). We also provide the corresponding complexity analysis.
We start by developing (n)GDxLD in the offline setting (Section \ref{sec:offline}) and discuss the connection between GDxLD and replica-exchange LD (Section \ref{sec:continuous}).
The rigorous complexity estimate is given by Theorem \ref{thm:strongconvex} in Section \ref{sec:off}.
We also study how the complexity depends on the exploration speed of LD in Section \ref{sec:LSI} (see Theorem \ref{thm:LSI}).
We then discuss how to adapt the algorithm to the online setting and develop (n)SGDxSGLD in Section \ref{sec:online}. 
Section \ref{sec:onlinebound} provides the corresponding complexity analysis -- Theorems \ref{thm:SGDstrongconvex} and \ref{thm:LSIonline}. 
To highlight the main idea and make the discussion concise, we defer the proof of  Theorems \ref{thm:strongconvex} and \ref{thm:LSI} to Appendix \ref{sec:offlineproof},  and the proof of Theorems \ref{thm:SGDstrongconvex} and \ref{thm:LSIonline} to Appendix \ref{sec:onlineproof}. 

\subsection{Replica exchange in offline setting} \label{sec:offline}
 We begin with a simple offline optimization scenario, where the objective function $F$ and its gradient $\nabla F$ are directly accessible. 
GD is one of the simplest deterministic iterative optimization algorithms. It involves iterating the formula
\begin{equation}
\label{eqn:GD}
X_{n+1}=X_n-\nabla F(X_n) h. 
\end{equation}
The hyper-parameter $h$ is known as the step size, which can often be set as a small constant. If $F$ is convex and $x^*$ is a global minimum, the GD iterates can converge to a global minimum ``linearly" fast; that is, $F(X_n)-F(x^*)\leq \epsilon$ if $n=O(\frac{1}{\epsilon})$. If $F$ is strongly convex with convexity parameter $m$, then the convergence will be ``exponentially" fast; that is, $F(X_n)-F(x^*)\leq \epsilon$ if $n=O((\log\tfrac1\epsilon)/m)$. 

In the general non-convex optimization setting,  $F$ can still be strongly convex near a global minimum $x^*$. In this case, if GD starts near $x^*$, the iterates can  converge to $x^*$ very fast. However, this is hard to implement in practice, since it requires a lot of prior knowledge of $F$. If we start GD at an arbitrary point, the iterates are likely to be trapped at a local minimum. To resolve this issue, one method is to add stochastic noise to GD and generate iterates according to 
\begin{equation}
\label{eqn:LD}
Y_{n+1}=Y_n-\nabla F(Y_n) h+\sqrt{2\gamma h} Z_n,
\end{equation}
where $Z_n$'s are i.i.d. samples from  $\mathcal{N}(0,I_d)$. 
For a small enough step size or learning rate $h$, the iterates \eqref{eqn:LD} can be viewed as a temporal discretization of the LD
\begin{equation}
\label{eqn:cLD}
dY_t=-\nabla F(Y_t) dt+\sqrt{2\gamma } dB_t,
\end{equation}
where $B_t$ is a $d$-dimensional standard Brownian motion. The diffusion process \eqref{eqn:cLD} is often called the overdamped Langevin dynamic or unadjusted Langevin dynamic \cite{Durmus:2017}. Here for notational simplicity, we refer to algorithm \eqref{eqn:LD} as LD as well. $\gamma$ is referred to as the temperature of the LD. 

It is known that under certain regularity conditions, \eqref{eqn:cLD} has an invariant measure $\pi_\gamma(x)\propto \exp(-\frac1{\gamma} F(x))$. Note that the stationary measure is concentrated around the global minimum for small $\gamma$. Therefore, it is reasonable to hypothesize that by iterating according to \eqref{eqn:LD} enough times, the iterates will be close to $x^*$. Adding the stochastic noise $\sqrt{2\gamma h}Z_n$ in \eqref{eqn:LD}  partially resolves the non-convexity issue, since the iterates can now escape local minima. 
%
On the other hand, in order for \eqref{eqn:LD} to be a good approximation of \eqref{eqn:cLD}, and consequently to be a good sampler of $\pi_\gamma$, it is necessary to use a small step size $h$ and hence a lot of iterations.  
In particular, \cite{xu2017global} shows in Corollary 3.3 that in order for 
$ \E[F(Y_n)]-F(x^*)=O(\epsilon)$, $h$ needs to scale linearly as $\epsilon$, and the computational complexity is $O(\epsilon^{-1}\log \tfrac1\epsilon)$. This is slower than GD if $F$ is strongly convex. 
 

In summary, when optimizing a non-convex $F$, one has a dilemma in choosing the types of algorithms. Using stochastic algorithms like LD will eventually find a global minimum, but they are inefficient for accurate estimation. Deterministic algorithms like GD is more efficient if initialized properly, but there is the danger that they can get trapped in local minima when initialized arbitrarily.

 To resolve this dilemma, idealistically, we can use a stochastic algorithm first to explore the state space. Once we detect  that the iterate is close to $x^*$, we switch to a deterministic algorithm.  However, it is in general difficult to write down a detection criterion a-priori. Our idea is to  devise a simple on-the-fly criterion. It involves running a copy of GD,  $X_n$ in \eqref{eqn:GD}, and a copy of LD, $Y_n$ in  \eqref{eqn:LD}, simultaneously. Since a smaller $F$-value implies the iterate is closer to $x^*$ in general, we apply GD in the next iteration for the one with a smaller $F$-value, and LD to the one with a larger $F$-value. In other words, we want to exchange the locations of $X_n$ and $Y_n$ if $F(X_n)>F(Y_n)+t_0$, where $t_0\geq 0$ is a properly chosen threshold. 
Finally, to implement the non-swapping variation, that is nGDxLD, it suffices to let LD stay at its original location when the exchange takes place.  A more detailed description of the algorithm is given in Algorithm \ref{alg:GDxLD}. 

\begin{algorithm}[ht]
\SetAlgoLined
 \textbf{Input: }{Temperature $\gamma$, step size $h$, number of steps $N$, 
 exchange threshold $t_0\in[0,r_0/8)$, and initial $X_0,Y_0$.}\\
 \For{ $n=0$ to $N-1$}{
$X'_{n+1}=X_n-\nabla F(X_n) h$\;
$Y'_{n+1}=Y_n-\nabla F(Y_n)h+\sqrt{2\gamma h}Z_n$, where $Z_n\sim N(0,I_d)$.\;
  \eIf{$F(Y'_{n+1})< F(X'_{n+1})-t_0$}{ $(X_{n+1}, Y_{n+1})=\begin{cases}(Y'_{n+1}, X'_{n+1}),\quad \text{if GDxLD is applied}\\ 
  (Y'_{n+1}, Y'_{n+1}),\quad \text{if nGDxLD is applied}\end{cases}$\;
   }{
$(X_{n+1}, Y_{n+1})=(X'_{n+1}, Y'_{n+1})$.
  }
 }
  \textbf{Output: }{$X_N$ as an optimizer for $F$.}
 \caption{GDxLD and nGDxLD: offline optimization }
 \label{alg:GDxLD}
\end{algorithm}
%

%
%

\subsection{GDxLD as a singular limit of replica-exchange Langevin diffusion} \label{sec:continuous}
In this section, we review the idea of replica-exchange LD, and show its connection with GDxLD. Consider an LD
\[
dX_t=-\nabla F(X_t)dt+\sqrt{2\nu} dW_t,
\]
where $W_t$ is a Brownian motion independent of $B_t$. 
Under certain regularity conditions on $F$, $\{X_t\}_{t\geq 0}$ has a unique stationary measure $\pi_\nu$ satisfying
\[
\pi_{\nu}(x)\propto\exp\left(-\tfrac1\nu F(x)\right).
\]
Note that the stationary measure is concentrated around the unique global minimum $x^*$, and the smaller the value of $\nu$, the higher the concentration is. In particular, if $\nu\to 0$, then $\pi_\nu$ is a Dirac measure at $x^*$.  However, from the algorithmic point of view, the smaller the value of $\nu$, the slower the stochastic process converges to its stationary distribution. In practice, we can only sample the LD for a finite amount of time, which gives rise to the tradeoff between the concentration around the global minimum 
and the convergence rate to stationarity.
One idea to overcome this tradeoff is to run two Langevin diffusions with two different temperatures, high and low, in parallel, and ``combining" the two
in an appropriate way so that we can enjoy the benefit of both high and low temperatures. This idea is known as replica-exchange LD.

Consider
\[\begin{split}
&dX_t^a=-\nabla F(X_t^a)dt+\sqrt{2\nu} dW_t\\
&dY_t^a=-\nabla F(Y_t^a)dt+\sqrt{2\gamma} dB_t
\end{split}\] 
The way we would connect them is to allow exchange between $X_t$ and $Y_t$ at random times.
In particular, we swap the positions of $X_t$ and $Y_t$ according to a state-dependent rate
\[
s(x,y;a):=a\exp\left(0\wedge \left\{ (\tfrac1\nu-\tfrac1\gamma)(F(x)-F(y)\right\}\right).
\]
We refer to $a$ as the exchange intensity.
The joint process $(X_t^a, Y^a_t)$ has a unique stationary measure 
\[
\pi_{\nu,\gamma}(x,y)\propto \exp\left(-\tfrac1\nu F(x)-\tfrac1\gamma F(y)\right).
\]
Based on $\pi_{\nu,\gamma}$, one would want to set $\nu$ small so that $X_t$ can exploit local minima, 
and set $\gamma$ large so that $Y_t$ can explore the state space to find better minima.
We exchange the positions of the two with high probability if $Y_t$ finds a better local minimum neighborhood to exploit
than the one where $X_t$ is currently at.

For optimization purposes, we can send $\nu$ to zero. In this case
\[
dX_t^a=-\nabla F(X_t^a)dt,
~~~
s(x,y;a)=a1_{\left\{F(x)<F(y)\right\}},
\]
and
\[
\pi_{0,\gamma}(x,y)\propto\delta_{x^*}(x) \exp\left(-\tfrac1\gamma F(y)\right).
\]

We would also like to send $a$ to infinity to allow exchange as soon as $Y_t^a$ find a better region to explore.
However, the processes $(X^a, Y^a)$ do not converge in natural senses when  $a\rightarrow\infty$, 
because the number of swap-caused discontinuities, which are of size $O(1)$, will grow without bound in any time interval of positive length.
\cite{Dupuis:2012} uses a temperature swapping idea to accommodate this.
Consider a temperature exchange process
\[\begin{split}
&d\tilde X_t^a=-\nabla F(\tilde X_t^a)dt+\sqrt{2\gamma1_{\{\tilde Z_t^a=1\}}}dW_t\\
&d\tilde Y_t^a=-\nabla F(\tilde Y_t^a)dt+\sqrt{2\gamma1_{\{\tilde Z_t^a=0\}}} dB_t,
\end{split}\]
where $\{\tilde Z_t^a\}_{t\geq 0}$ is a Markov jump process that switches from $0$ to $1$ at rate
$a1_{\{F(x)>F(y)\}}$ and $1$ to $0$ at rate $a1_{\{F(x)<F(y)\}}$. Now, sending $a$ to infinity, we have
\[\begin{split}
&d\tilde X_t^{\infty}=-\nabla F(\tilde X_t^{\infty})dt+\sqrt{2\gamma1_{\{F(\tilde X_t^{\infty})>F(\tilde Y_t^{\infty})\}}}dW_t\\
&d\tilde Y_t^{\infty}=-\nabla F(\tilde Y_t^{\infty})dt+\sqrt{2\gamma1_{\{F(\tilde X_t^{\infty})<F(\tilde Y_t^{\infty})\}}}dB_t.
\end{split}\]
In actual implementations, we will not be able to exactly sample the continuous-time processes. 
We thus implement a discretized version of it as described in Algorithm \ref{alg:GDxLD}.\\

\subsection{Complexity bound for GDxLD and nGDxLD} 
\label{sec:off}

We next present the performance bound for GDxLD and nGDxLD. 
We start by introducing some assumptions on $F$.

First we need the function to be smooth in the following sense:
\begin{assumption}
\label{aspt:Lip}
$F$ is Lipschitz continuous with Lipschitz constant $L\in(0,\infty)$, i.e.,
\[
\|\nabla F(x)-\nabla F(y)\|\leq L\|x-y\|.
\]
\end{assumption}
Second, we need some conditions under which the iterates or their function values will not diverge to infinity. The following assumption ensures the gradient will push the iterates back once they get too large:
\begin{assumption}
\label{aspt:coercive}
The utility function is coercive. In particular, there exist constants $\lambda_c, M_c\in (0,\infty)$, such that  
\[
\|\nabla F(x)\|^2\geq \lambda_cF(x) -M_c, 
\]
and $F(x)\to \infty$ when $\|x\|\to \infty$. 
\end{assumption}
Note that another more commonly used definition of coerciveness (or dissipation) \cite{raginsky2017non, xu2017global} is
\begin{equation}\label{eq:coer}
-\langle \nabla F(x), x\rangle \leq -\lambda_0 \|x\|^2+M_0. 
\end{equation}
The condition \eqref{eq:coer} is stronger than Assumption \ref{aspt:coercive}. In general, \eqref{eq:coer} can be enforced by adding proper regularizations. These are explained by Lemma \ref{lem:coercive}. Its proof can be found in Appendix A. 
\begin{lemma}
\label{lem:coercive}
Under Assumption \ref{aspt:Lip}, 
\begin{enumerate}
\item For any $\lambda>0$, \eqref{eq:coer} holds for $F_\lambda(x)=F(x)+\lambda \|x\|^2$.
\item Suppose  \eqref{eq:coer} holds, then Assumption \ref{aspt:coercive} holds. 
\end{enumerate} 
\end{lemma}

The last condition we need is that $F$ is twice differentiable and strongly convex near the global minimum $x^*$. This allows GD to find $x^*$ efficiently when it is close to $x^*$. 
\begin{assumption}
\label{aspt:convstrong}
$x^*$ is the unique global minimum for $F$. There exists $r_0>0$, such that 
the sub-level set $B_0=\{x: F(x)\leq F(x^*)+r_0\}$ is star-convex with $x^*$ being the center, i.e., a line segment connecting $x^*$ and any $x\in B_0$ is in $B_0$.
There exists $m,M,r>0$ such that the Hessian $\nabla^2F$ exists for $B_0\subset\{x:\|x-x^*\|<r\}$, and for any vector $v$
\[
m \|v\|^2\leq v^T\nabla^2 F(x)v\leq M\|v\|^2 ,\quad \forall  x:\|x-x^*\|<r.
\]
Note that when Assumption \ref{aspt:Lip} holds, $m\leq M\leq L$.
\end{assumption}
Assumption \ref{aspt:convstrong} is a mild one in practice, since checking  the Hessian matrix is positive definite is often the most direct way to determine whether a point is a local minimum. Figuratively speaking, Assumption \ref{aspt:convstrong}
 essentially requires that $F$ has only one  global minimum  $x^*$ at the bottom of a ``valley" of $F$.
It is important to emphasize that we do not assume knowing where this ``valley'' is, otherwise it suffices to run GD within this ``valley". Moreover, when the LD process enters a ``valley", there is no mechanism required to detect whether this ``valley" contains the global minimum. In fact, designing such detection mechanism can be quite complicated, since it is similar to finding the optimal solution  in a multi-arm bandit problem.
Instead, for GDxLD, we only need to implement the simple exchange mechanism.

Requiring $B_0$ to be star-convex rules out the scenario where there are multiple local minima taking values only $\epsilon$-apart from the global optimal value where $\epsilon$ can be arbitrarily small. In this scenario, GDxLD will have a hard time to identify the true global local minima. Indeed,  GD iterates may stuck at one of the local minima, and the exchange will take place only when  LD iterate is $o(\epsilon)$ away from $x^*$, which will happen rarely due to the noisy movement of LD. On the other hand, finding solutions whose functional values are only $\epsilon$ away from the global optimal is often considered good enough in practice.   


When $\nabla^2 F$ is not available, or $F$ is only convex near $x^*$, we can also use the following more general version of Assumption \ref{aspt:convstrong}. Admittedly the complexity estimate under it will be be worse than the one under Assumption \ref{aspt:convstrong}.
%
%

\begin{assumption}
\label{aspt:conv}
$x^*$ is the unique global minimum for $F$. There exists $r_0>0$, such that 
the sub-level set $B_0=\{x: F(x)\leq F(x^*)+r_0\}$ is star-convex with $x^*$ being the center, i.e., a line segment connecting $x^*$ and any $x\in B_0$ is in $B_0$. $F$ is convex in $B_0$. In addition, there exist $0<r_l<r_u<\infty$, such that $\{\|x-x^*\|\leq r_l\}\subset B_0\subset\{\|x-x^*\|\leq r_u\}$. 
\end{assumption}

%


With all assumptions stated, the complexity of GDxLD and nGDxLD is given in the following theorem. 

\begin{theorem}
\label{thm:strongconvex}
Consider the iterates following Algorithm \ref{alg:GDxLD} ((n)GDxLD). 
Under Assumptions \ref{aspt:Lip}, \ref{aspt:coercive}, \ref{aspt:convstrong}, and fixed $\gamma=O(1)$, $0<h\leq \tfrac1{2L}$, $0\leq t_0\leq r_0/8$,
for any $\epsilon>0$ and $\delta>0$, there exists $N(\epsilon,\delta)=O(-\log \epsilon)+O(-\log \delta)$ such that for any $n\geq N(\epsilon,\delta)$, 
\begin{equation}
\label{eqn:PAC}
\Prob(F(X_n)-F(x^*)\leq \epsilon)\geq1- \delta.
\end{equation}
Alternatively, under Assumptions \ref{aspt:Lip}, \ref{aspt:coercive}, \ref{aspt:conv}, and fixed $\gamma=O(1)$, $0<h\leq \min\{\tfrac1{2L},\tfrac{r_u^2}{r_0}\}$, \eqref{eqn:PAC} holds with $N(\epsilon,\delta)=O(\epsilon^{-1})+O(\log(1/\delta))$.
\end{theorem} 
In particular, if we hold $\delta$ fixed, to achieve an $\epsilon$ accuracy, the complexity is $O(\log(1/\epsilon))$ when $F$ is strongly convex in $B_0$ and 
$O(\epsilon^{-1})$ when $F$ is convex in $B_0$. 
These complexity estimates are of the same order as GD in the convex setting. However,  $F$ does not need to be convex globally for our results to hold.

The fast convergence rate comes partly from the fact that GDxLD does not require the hyper-parameters $h$ and $\gamma$ to change with the error tolerance $\epsilon$. This is quite unique when compared with other ``single-copy" stochastic optimization algorithms. This feature is of great advantage in both practical implementations and theoretical analysis. 

 Finally, as pointed out in the introduction, our analysis and subsequent results focus on a fixed dimension setting. 
 The big $O$ terms in the definition of $N(\epsilon,\delta)$ ``hide" constants that are independent of $\epsilon$ and $\delta$. 
Such constants can scale exponentially with $d$ and $\frac{1}{\gamma h}$. In particular, the $O(\log(1/\delta))$ term is associated with how long it takes the LD to visit $B_0$, which is further determined by how fast LD converges to stationarity. With a fixed value of $\gamma$, the convergence speed in general scales inversely with the exponential of $d$ (see the proof of Theorem \ref{thm:strongconvex} in Appendix \ref{sec:offlineproof} for more details).
Curse of dimensionality is a common issue for sampling-based optimization algorithms.
In \cite{raginsky2017non,xu2017global}, this is mentioned as a problem of the spectral gap. A more detailed discussion can also be found in \cite{ma2019sampling}.

\subsection{Exploration speed of LD} \label{sec:LSI}
Intuitively, the efficiency of (n)GDxLD depends largely on how often LD visits the neighborhood of $x^*$, $B_0$. 
A standard way to address this question is to study the convergence rate of LD to stationarity. 
Let 
\[\pi_{\gamma}\propto \exp(-F(x)/\gamma),\] 
which denotes the ``target" invariant distribution, i.e., it is the invariant distribution of the Langevin diffusion defined in \eqref{eqn:cLD}. We also write $U_\gamma$ as the normalizing constant of $\pi_\gamma$.

In the literature, the convergence of a stochastic processes is usually obtained using either a small set argument or a functional inequality argument. 
Theorem \ref{thm:strongconvex} uses the small set argument. In particular, in Lemma \ref{lem:smallset}, we show that each ``typical" LD iterate can reach $B_0$ with a probability lower bounded by $\alpha>0$. While such $\alpha$ is usually not difficult to obtain, it is often pessimistically small and does not reveal how the convergence depends on $F$. In contrast, the functional inequality approach usually provides more details about how the convergence depends on $F$. 
In particular, using functional inequalities like log Sobolev inequality (LSI) or Poincar\'{e} inequality, one can show that the convergence speed of LD to $\pi_\gamma$ depends on the pre-constant of the functional inequality. This pre-constant usually only depends on the target density $\pi_\gamma$, i.e., $F$, and can be seen as an intrinsic measure of how difficult it is to sample $\pi_\gamma$. In addition, the convergence speed derived from this approach is in general better. For example, Proposition 2.8 in \cite{hairer2014spectral} indicates that the convergence speed estimate from the small set argument is in general smaller than the one derived from Poincar\'e inequality. Therefore, it is of theoretical importance to show how the computational complexity estimates of (n)GDxLD depend on the pre-constant of the functional inequality. In this paper, we use LSI as it controls the convergence speed of LD in Kullback-Leibler (KL) divergence. Note that in general, KL divergence between two distributions has a linear dependence on the dimension, whereas the $\chi^2$ divergence, which is associated with Poincar\'{e} inequality, often scales exponentially in the dimension.

Let $\mu_t$ denote the distribution of $Y_t$, i.e., the Langevin diffusion \eqref{eqn:cLD} at time $t$. Let $\hat \mu_n$ denote the distribution of $Y_n$, i.e., the discrete LD update \eqref{eqn:LD} at step $n$. For two measures, $\mu$ and $\pi$, with $\mu\ll \pi$, the KL divergence of $\mu$ with respect to $\pi$ is defined as
\[\KL(\mu\|\pi)=\int \log\left(\frac{\mu(x)}{\pi(x)}\right)\mu(x)dx.\]
We impose the following assumption on $\pi_{\gamma}$.
\begin{assumption}
\label{aspt:LSI}
$\pi_\gamma$ satisfies a log Sobolev inequality with pre-constant $\beta>0$:
\[
\int \pi_\gamma(x)f(x)^2\log f(x)^2\leq \frac{2}{\beta} \E_{\pi_\gamma}\|\nabla f\|^2.
\]
\end{assumption}

It is well known that with LSI, the Langevin diffusion \eqref{eqn:cLD} satisfies
\[
\KL(\mu_t |\pi_\gamma)\leq \exp(-2\gamma\beta t) \KL(\mu_0|\pi).
\]
In addition, the LD \eqref{eqn:LD} satisfies a similar convergence rate \cite{vempala2019rapid}:
\[
\KL(\hat{\mu}_n |\pi)\leq \exp(-\gamma\beta  nh) \KL(\mu_0|\pi)+\frac{8 hdL^2}{\beta\gamma}.
\]
Adapting these results, we can establish the following convergence results for (n)GDxLD.

\begin{theorem}
\label{thm:LSI}
Consider the iterates following Algorithm \ref{alg:GDxLD} ((n)GDxLD). Suppose Assumptions \ref{aspt:Lip}, \ref{aspt:coercive}, \ref{aspt:convstrong}, and \ref{aspt:LSI} hold. In addition, for $GDxLD$, $t_0>0$. Fix
\[
h<\left(\frac{\pi_{\gamma}(B_0)}{2R_V+6\sqrt{d}L/\sqrt{\beta \gamma}}\right)^2,
\]
where $R_V=8(M_c/4+4\gamma Ld)/\lambda_c$.
For any $\epsilon>0$ and $\delta>0$, set
\[
K=\frac{\log (\delta)}{ \log (1-\pi_\gamma(B_0)/2)} \mbox{ and } n_0=\frac{1}{\beta\gamma}\frac{1}{h}\log\left(\frac{1}{h}\right).
\]
There exists 
\[N(\beta,\epsilon,\delta)=O(Kn_0)+O(\log(1/\epsilon)) =O(\log(1/\delta)/\beta)+O(\log(1/\epsilon))\] 
such that 
for $n\geq N(\beta, \epsilon,\delta)$
\[
\Prob(|X_n-x^*|\leq \epsilon)>1-\delta.
\]
\end{theorem}

The big O terms in the definition of $N(\beta,\epsilon, \delta)$ in Theorem \ref{thm:LSI} ``hide" constants that have polynomial dependence on $d$, $1/\pi_{\gamma}(B_0)$, and $\log U_\gamma$ (see the proof of Theorem \ref{thm:LSI} in Appendix \ref{app:thm3} for more details). The complexity bound in Theorem \ref{thm:LSI} can be viewed as a refinement of the bound established in Theorem \ref{thm:strongconvex}, where the hidden constant scales exponentially with $d$ and $\frac1{\gamma h}$. This is due to the fact that the LSI constant $\beta$ characterizes how difficult it is to sample $\pi_\gamma$. 
To interpret this refined complexity bound, we note that $n_0$ represents the time it takes the LD to sample the approximate stationary distribution. If we check whether the LD is in $B_0$ every $n_0$ steps, $K$ is the number of trials we need to run to ensure a larger than $1-\delta$ probability of visiting $B_0$. Then, the time it takes the LD to visit $B_0$ with high probability is upper bounded by $Kn_0$. Note that $K$ scales with $1/\pi_{\gamma}(B_0)$. In order to have a smaller $K$, we want to have a larger $\pi_{\gamma}(B_0)$. This can be achieved by using a smaller $\gamma$, which gives rise to a $\pi_\gamma$ that is more concentrated on $B_0$. On the other hand, to achieve a smaller $n_0$, we want to have a larger $\gamma$ so the LD can escape local minima and converge to stationarity faster.
Thus, there is a non-trivial tradeoff to be made here to optimize the constant term ``hidden" in $O(\log(1/\delta)/\beta)$ in the quantification of $N(\beta,\epsilon, \delta)$. However, as will be demonstrated in our numerical experiments in Section \ref{sec:num_sen}, (n)GDxLD achieves good and robust performance for a reasonable range of $\gamma$ (not too small or too big). This is because in (n)GDxLD, $\gamma$ does not need to depend on $\epsilon$.
 
We also note that the analysis of nGDxLD is conceptually simpler than GDxLD, since the LD part is a bona-fide unadjusted Langevin algorithm, so the result of \cite{vempala2019rapid} can be implemented rather straightforwardly. In contrast, the LD part in GDxLD can be swapped with the GD part. Thus, the analysis of GDxLD is more challenging. In particular, we need to impose the assumption that $t_0>0$ in Algorithm \ref{alg:GDxLD}. 
This technical assumption limits the number of swaps between GD and LD.

\subsection{Online optimization with stochastic gradient}
\label{sec:online}
In a lot of data science applications, we define a loss function for a given parameter $x$ and data sample $s$  as $f(x,s)$, and the loss function we wish to minimize is the average of $f(x,s)$ over a distribution $S$. Let
\[
F(x)=\E_S[f(x,S)].
\]
Since the distribution of $S$ can be complicated or unknown,  the precise evaluation of $F$ and the gradient $\nabla F$ may be computationally too expensive or practically infeasible. However, we often have access to a large number of samples of $S$ in applications. So given an iterate $X_n$, we can draw two independent batches of independent and identically distributed (iid) samples, 
$s_{n,1},\ldots, s_{n,\Theta}$ and $s'_{n,1},\ldots, s'_{n,\Theta}$, and use
\[
\hat F_n(X_n)=\frac1\Theta \sum_{i=1}^\Theta f(X_n,s_{n,i}),\quad \nabla \hat F_n(X_n)=\frac1\Theta \sum_{i=1}^\Theta \nabla_x f(X_n,s'_{n,i})
\]
to approximate $F$ and $\nabla F$. Here we require $\{s_{n,i}, 1\leq i\leq \Theta\}$ and $\{s'_{n,i}, 1\leq i\leq \Theta\}$ to be two independent batches, so that the corresponding approximation errors are uncorrelated.

When we replace $\nabla F$ with $\nabla \hat F_n$ in GD and LD, the resulting algorithms are called SGD and SGLD. They are useful when the data samples are accessible only online:  to run the algorithms, we only need to get access to and operate on a batch of data. This is very important when computation or storage capacities are smaller than the size of the data.

To implement the replica exchange idea in the online setting, it is natural to replace GD and LD with their online versions. In addition, we need to pay special attention to the exchange criterion. Since we only have access to $\hat F_n$, not $F$,  $\hat F_n(X'_{n+1})>\hat F_n(Y_{n+1})$ may not guarantee that  $F(X'_{n+1})>F(Y_{n+1})$. Incorrect exchanges may lead to bad performance, and thus we need to be cautious to avoid that. One way to avoid/reduce incorrect exchanges is to introduce a threshold $t_0>0$ when comparing $\hat F_n$'s. In particular, if $t_0$ is chosen to be larger than the ``typical" size of approximation errors of $\hat F_n$, then, $\hat F_n(X'_{n+1})>\hat F_n(Y_{n+1})+t_0$ indicates that $\hat F(X'_{n+1})$ is very ``likely" to be larger than $\hat F(Y_{n+1})$.
Lastly, since the approximation error of $\hat F_n(x)$ in theory can be very large when $x$ is large, we avoid exchanges if the iterates are very large, i.e., when $\min\{\|X_n\|, \|Y_n\|\}>\hat M_v$ for some large $\hat M_v\in(0,\infty)$.

Putting these ideas together, the SGDxSGLD algorithm is given in Algorithm \ref{alg:SGDxSGLD}:\\
\begin{algorithm}[ht]
\SetAlgoLined
 \textbf{Input: }{Temperature $\gamma$, step size $h$, number of steps $N$, initial $X_0,Y_0$, estimation error parameter $\Theta$ (when using batch means, $\Theta$ is the batch size, it controls the accuracy of $\hat F_n$ and $\nabla \hat F_n$), threshold $t_0\in(0,r_0/8]$, and exchange boundary $\hat M_v$.}\\
 \For{ $n=0$ to $N-1$}{
$X'_{n+1}=X_n-h\nabla \hat F_n(X_n)$\;
$Y'_{n+1}=Y_n-h\nabla \hat F_n(Y_n)+\sqrt{2\gamma h}Z_n$, where $Z_n\sim N(0,I_d)$\;
  \eIf{$\hat F_n(Y'_{n+1})<\hat F_n(X'_{n+1})-t_0$, $\|X_{n+1}^{\prime}\|\leq \hat M_V, \mbox{ and } \|Y_{n+1}^{\prime}\|\leq \hat M_V$}{ $(X_{n+1},Y_{n+1})=\begin{cases}(Y'_{n+1}, X'_{n+1}),\quad \text{if SGDxSGLD is applied}\\ 
  (Y'_{n+1}, Y'_{n+1}),\quad \text{if nSGDxSGLD is applied}\end{cases}$\;
   }{
$(X_{n+1}, Y_{n+1})=(X'_{n+1}, Y'_{n+1})$.
  }
 }
 \textbf{Output: }{$X_N$ as an optimizer for $F$.} 
 \caption{SGDxSGLD and nSGDxSGLD: online optimization }
 \label{alg:SGDxSGLD}
\end{algorithm}

\subsection{Complexity bound for SGDxSGLD}
\label{sec:onlinebound}
To implement SGDxSGLD, we require three new hyper-parameters, $\Theta$, $t_0$ and $\hat M_v$. We discuss how they can be chosen next.

First of all, the batch-size $\Theta$ controls the accuracy of the stochastic approximations of $F$ and $\nabla F$.  In particular, we define
\[
\zeta_n(x):=\hat F_n(x)-F(x)\mbox{ and } \xi_n(x)=\nabla \hat F_n(x)-\nabla F(x),\]
where $\zeta_n(x)$'s and $\xi_n(x)$'s are independent random noise with $\E[\zeta_n(x)]=\E[\xi_n(x)]=0$. By controlling the number of samples we generate at each iteration, we can control the accuracy of the estimation, as the variances of the estimation errors are of order $1/\Theta$. We will see in Theorem \ref{thm:SGDstrongconvex} and the discussions following it that $1/\Theta$ should be of the same order as the error tolerance $\epsilon$.  For the simplicity of exposition, we introduce a new parameter $\theta=O(\epsilon)$ to describe the ``scale" of $\zeta_n$ and $\xi_n$. In addition, we assume that the errors have sub-Gaussian tails:

\begin{assumption}
\label{aspt:noise}
There exists a constant $\theta>0$, such that for any $0<b<\frac{1}{2\theta}$, 
\[
\E\left[\exp\left(a^T \xi_n(x)+ b\|\xi_n(x)\|^2\right)\right]\leq \frac{1}{(1-2b\theta)^{d/2}}\exp\left(\frac{\|a\|^2\theta}{2(1-2b\theta)}\right),
\]
and for $\|x\|,\|y\|\leq \hat M_V$, we have for any $z>0$,
\[
\Prob(\zeta_n(x)-\zeta_n(y)\geq z)\leq \exp\left(-\frac{z^2}{\theta}\right).
\]
\end{assumption}
 Note that Assumption \ref{aspt:noise} implies
\[
\E_n[\|\xi_n(x)\|^2]\leq d\theta.
\]
We also remark that Assumption \ref{aspt:noise} holds if $\xi_n(x)\sim \mathcal{N}(0, \theta I_d)$ and $\zeta_n(x)-\zeta_n(y)\sim \mathcal{N}(0,\frac12\theta)$.
In practice, Assumption \ref{aspt:noise} can be verified using Hoeffding inequality or other concentration inequalities if the stochastic gradients are bounded.

Second, the threshold $t_0$ is related to the shape of the ``valley" around $x^*$.  To keep the exposition concise, we set $t_0\leq r_0/8$ where $r_0$ is defined in Assumption \ref{aspt:conv}. Heuristically, it can be chosen as a generic small constant such as $10^{-2}$.

Lastly, $\hat{M}_V$ is introduced to facilitate theoretical verification of Assumption \ref{aspt:noise}. In other words, if Assumption \ref{aspt:noise} holds for $\hat{M}_V=\infty$, then we can set $\hat{M}_V=\infty$. More generally, under Assumptions \ref{aspt:Lip} and \ref{aspt:coercive}, we set
\[       
\hat C_V=\frac{M_c}{4}+(8\gamma Ld+4L\theta d), ~ \hat R_V=8\lambda_c^{-1}\hat C_V 
\mbox{ and } \hat M_V=\sup\{x: F(x)\leq \hat R_V\}.
\]                       
In practice, one can set $\hat{M}_V$ as a generic large number.

%

%


%
%
%

We are now ready to present the complexity of (n)SGDxSGLD:

\begin{theorem}
\label{thm:SGDstrongconvex}
Consider the iterates following Algorithm \ref{alg:SGDxSGLD} ((n)SGDxSGLD).
Under Assumptions \ref{aspt:Lip}, \ref{aspt:coercive}, \ref{aspt:convstrong}, \ref{aspt:noise}, and fixed $\gamma=O(1)$, $0<h\leq \min\{\tfrac1{2L},\tfrac{r_u^2}{r_0}\}$, and $0<t_0\leq r_0/8$,
for any $\epsilon>0$ and $\delta>0$, there exists $N(\epsilon,\delta)=O(-\log(\delta)-\log(\epsilon))$, such that for any fixed $N>N(\epsilon,\delta)$, there exists 
$\theta(N,\epsilon,\delta)=O\left(\min\{N^{-1},\epsilon\delta\}\right)$, and for $\theta \leq \theta(N,\epsilon,\delta)$,
we have
\[
\Prob(F(X_N)-F(x^*)\leq \epsilon)\geq1- \delta.
\]
\end{theorem} 

In particular, if we hold $\delta$ fixed, then to achieve an $\epsilon$ accuracy, we need to set the number of iterations $N=O(-\log \epsilon)$ and the batch size $\Theta=O(\theta^{-1})=O(\epsilon^{-1})$.
In this case, the total complexity (including data sample) is $O(N\Theta)=O(\epsilon^{-1}\log \tfrac1\epsilon )$.

To see where the $O(\epsilon^{-1})$ batch size comes from, we can look at a simple example where $F(x)=\frac12 x^2$.
As this function is strongly convex, we can focus on the SGD part. 
The iterates in this case takes the form
\[
X_{n+1}=X_n-hX_n+h\xi_{n}(X_n).
\]
For simplicity, we assume $\xi_n(X_n)$'s are iid $\mathcal{N}(0,\theta I_d)$. Then, $X_n$ is a linear auto-regress sequence with the invariant measure $\mathcal{N}(0, \frac{h\theta}{2-h} I_d)$. Now, for $\E[F(X_N)]=\frac{h\theta d}{2(2-h)}=O(\epsilon)$ when $h$ is a constant, $\theta=O(\epsilon)$.

Similar to the offline case, we can also characterize how the computational cost depends on the LSI pre-constant. Let 
\[
\hat B_0=\left\{x: F(x)\leq F(x^*)+\frac{r_0}{4}\right\}.
\]
\begin{theorem}
\label{thm:LSIonline}
Consider the iterates following Algorithm \ref{alg:SGDxSGLD} ((n)SGDxSGLD).
Suppose Assumptions \ref{aspt:Lip}, \ref{aspt:coercive}, \ref{aspt:convstrong}, \ref{aspt:LSI}, and \ref{aspt:noise} hold. 
Fix.
\[
h < \left(\frac{\pi_{\gamma}(\hat B_0)}{4\hat R_V+8\sqrt{ d} L/\sqrt{\beta \gamma}}\right)^2.
\]
For any $\epsilon>0$ and $\delta>0$, let
\[
K=\frac{\log(\delta)}{\log(1-\pi_{\gamma}(\hat B_0)/2)} \mbox{ and } n_0=\frac{4}{\beta\gamma} h^{-1}\log(1/h)+1.
\]
There exists
\[N(\beta,\epsilon,\delta)=O(Kn_0) + O(\log(1/\epsilon))=O(\log(1/\delta) /\beta+\log (1/\epsilon)),\]
such that for $N>N(\beta,\epsilon,\delta)$, there exists $\theta(N, \epsilon, \delta)=O(\min\{N^{-1}, \epsilon\delta\})$, and for $\theta<\theta(N,\epsilon,\delta)$, 
\[
\Prob(F(X_N)-F(x^*)<\epsilon)>1-\delta.
\]
\end{theorem}

Similar to the offline versions, the big O term in the definition of $N(\beta,\epsilon)$ in Theorem \ref{thm:SGDstrongconvex} ``hides" constants that scale exponentially with $d$ and  $\frac1{\gamma h}$, while
the big O terms in the definition of $N(\beta,\epsilon, \delta)$ in Theorem \ref{thm:LSIonline} ``hide" constants that have polynomial dependence on $d$ (see the proofs of Theorems \ref{thm:SGDstrongconvex} and \ref{thm:LSIonline} in Appendices \ref{proof:thm4} and \ref{app:thm5} for more details).
Lastly, it is worth noting that the step size $h$ and temperature $\gamma$ in (n)SGDxSGLD is independent of the error tolerance $\epsilon$. This is one of the reason why it can beat existing sampling-based algorithms on the convergence speed. 

\section{Numerical experiments}
\label{sec:num}
In this section, we provide some  numerical experiments to illustrate the performance of (n)GDxLD and (n)SGDxSGLD. 
Our main focus is to demonstrate that by doing exchange between the two algorithms, (S)GD and (SG)LD, 
the performance of the combined algorithm can be substantially better than running isolated copies of the individual algorithms.  We also demonstrate the robustness of our algorithm to different choices of the temperature $\gamma$ and step size $h$.

\subsection{Two-dimensional Problems}
We start by looking at two-dimensional examples, which are easier to visualize.

\subsubsection{Offline setting} 
\label{sec:offnum}
First, we consider how to find the mode of a two-dimensional Gaussian-mixture density. The loss function is given by 
\begin{equation}
\label{eqn:offF}
F(x)=-\sum_{i=1}^M \frac{w_i }{\sqrt{\text{det}(2\pi \Sigma_i)}} \exp\left(-\tfrac{1}{2} (x-m_i)^T\Sigma_i^{-1}(x-m_i)\right)+L(x). 
\end{equation}
For simplicity, we choose $M=5\times 5=25$, each $m_i$ is a point in the meshgrid $\{0,1,2,3,4\}^2$, and $\Sigma_i=$diag$(0.1)$ so the ``valleys" are distinctive. The weights are generated randomly. As the Gaussian-mixture density and its gradient are almost zero for $x$ far away from $m_i$'s, we add a quadratic regularization term
\begin{equation}\label{eq:reg}
L(x)=\sum_{i=1}^{2}\left\{(x(i)+1)^21_{x(i)\leq -1} + (x(i)-5)^21_{x(i)\geq 5}\right\},
\end{equation}
where $x(i)$ is the $i$-th element of $x$. 

Figure \ref{fig:ObjF} shows the heat map and the 3-d plot of one possible realization of $F$. We can see that it is highly non-convex with 25 local minima.
In this particular realization of $F$, the global minimum is at $(3,2)$ and $F(3,2)=-0.168$.

\begin{figure}[tbp]
	\centering
	\subfloat[Heat map of $F$]{\includegraphics[width=0.48\textwidth]{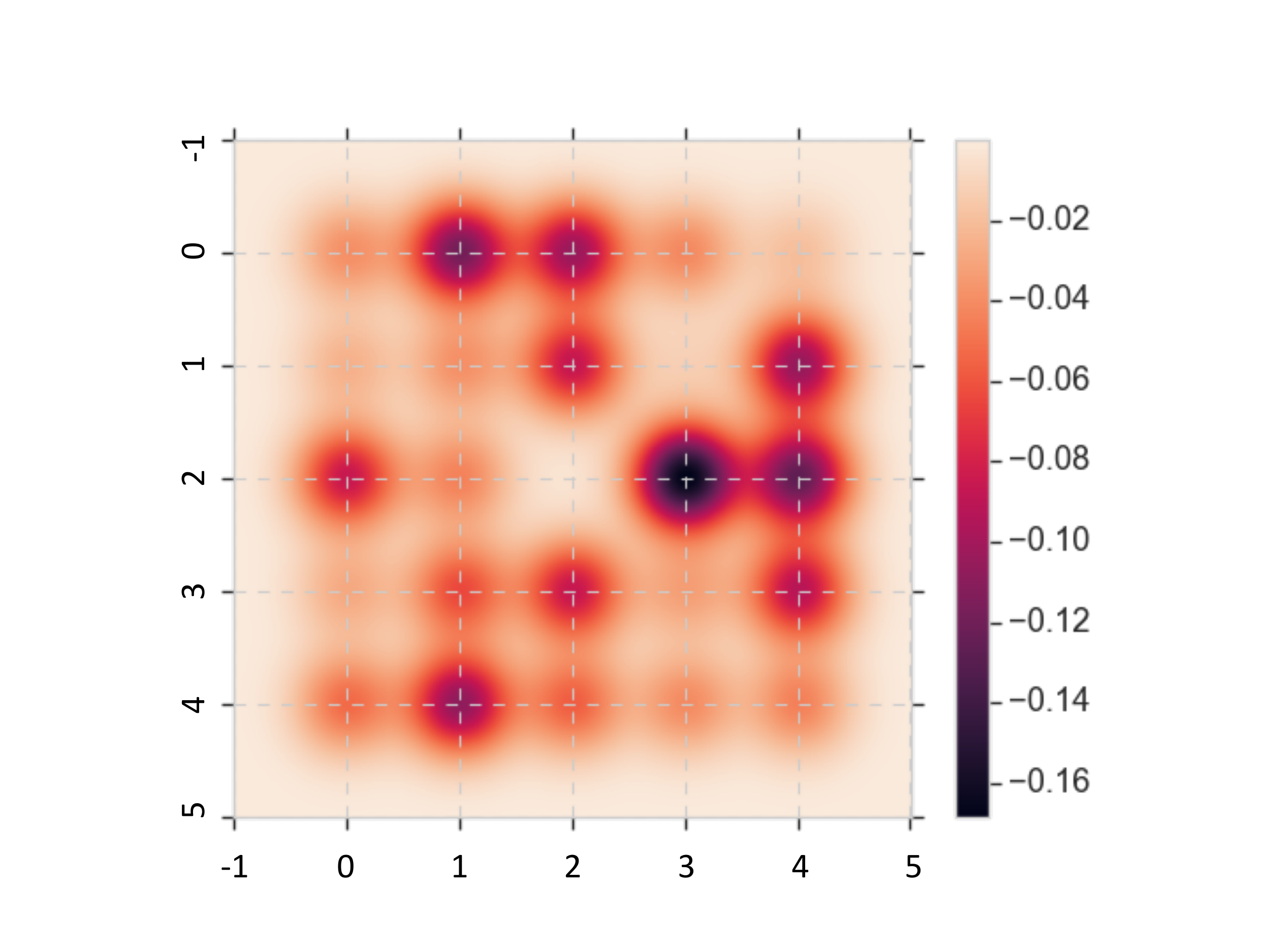}}
	\quad
	\subfloat[3-d plot of $F$]{\includegraphics[width=0.48\textwidth]{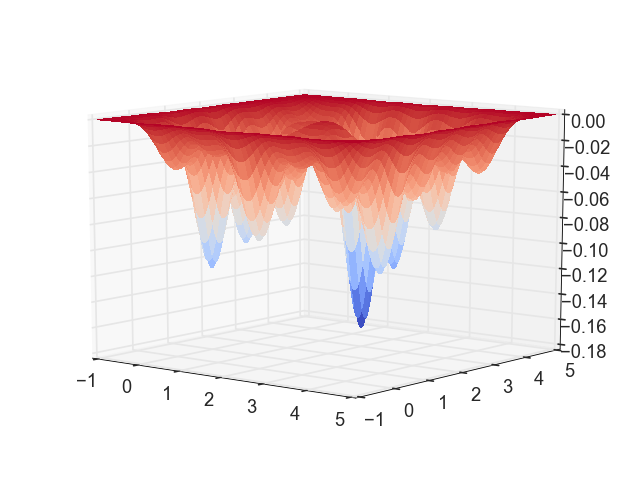}}
	\\	
	\caption{The objective function $F$ in \eqref{eqn:offF}}
	\label{fig:ObjF}
\end{figure}


We implement GDxLD for the objective function plotted in Figure \ref{fig:ObjF} with $h=0.1$, $\gamma=1$, $X_0=(0,0)$, and $Y_0=(1,1)$.
We plot $F(X_n)$ and $X_n$ at different iterations $n$ in Figure \ref{fig:GDxLD}.
We do not plot $Y_n$, the sample path of the LD, since it is used for exploration, not optimization.
We observe that the convergence happens really fast despite the underlying non-convexity. In particular, we find the global minimum with less than 300 iterations.
We run $100$ independent copies of GDxLD and are able to find the global minimum within 1000 iterations in all cases.
In Figure \ref{fig:nGDxLD}, we plot $F(X_n)$ and $X_n$ at different iterations for a typical nGDxLD implementation.
We again observe that the convergence happens really fast, i.e., with less than 400 iterations.

\begin{figure}[tbp]
	\centering
	\subfloat[$F(X_n)$]{\includegraphics[width=0.48\textwidth]{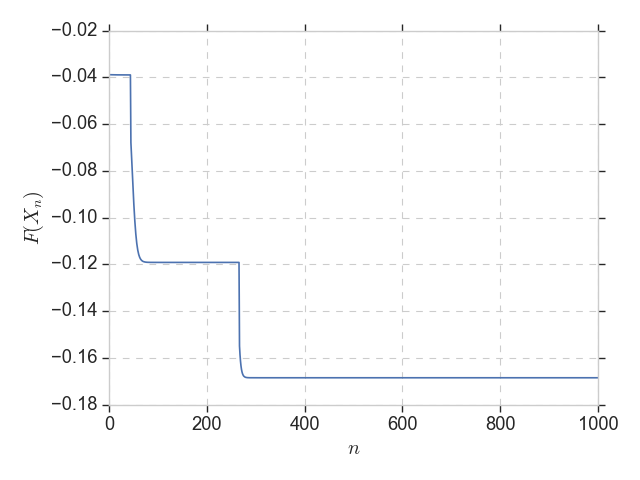}}
	\subfloat[$X_n$]{\includegraphics[width=0.48\textwidth]{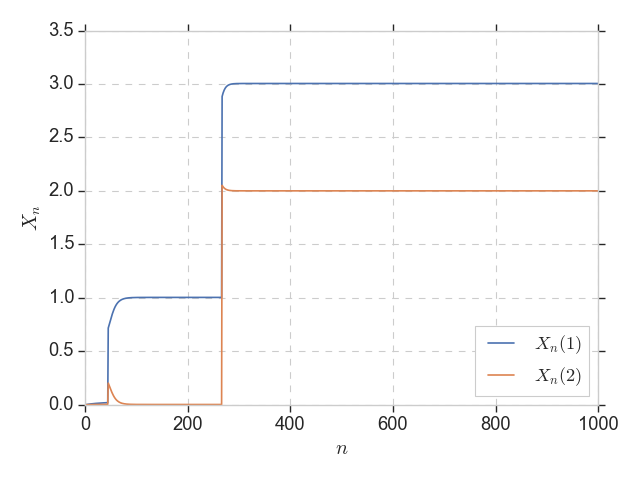}}
	\\	
	\caption{Convergence of the iterates under GDxLD}
	\label{fig:GDxLD}
\end{figure}

\begin{figure}[tbp]
	\centering
	\subfloat[$F(X_n)$]{\includegraphics[width=0.48\textwidth]{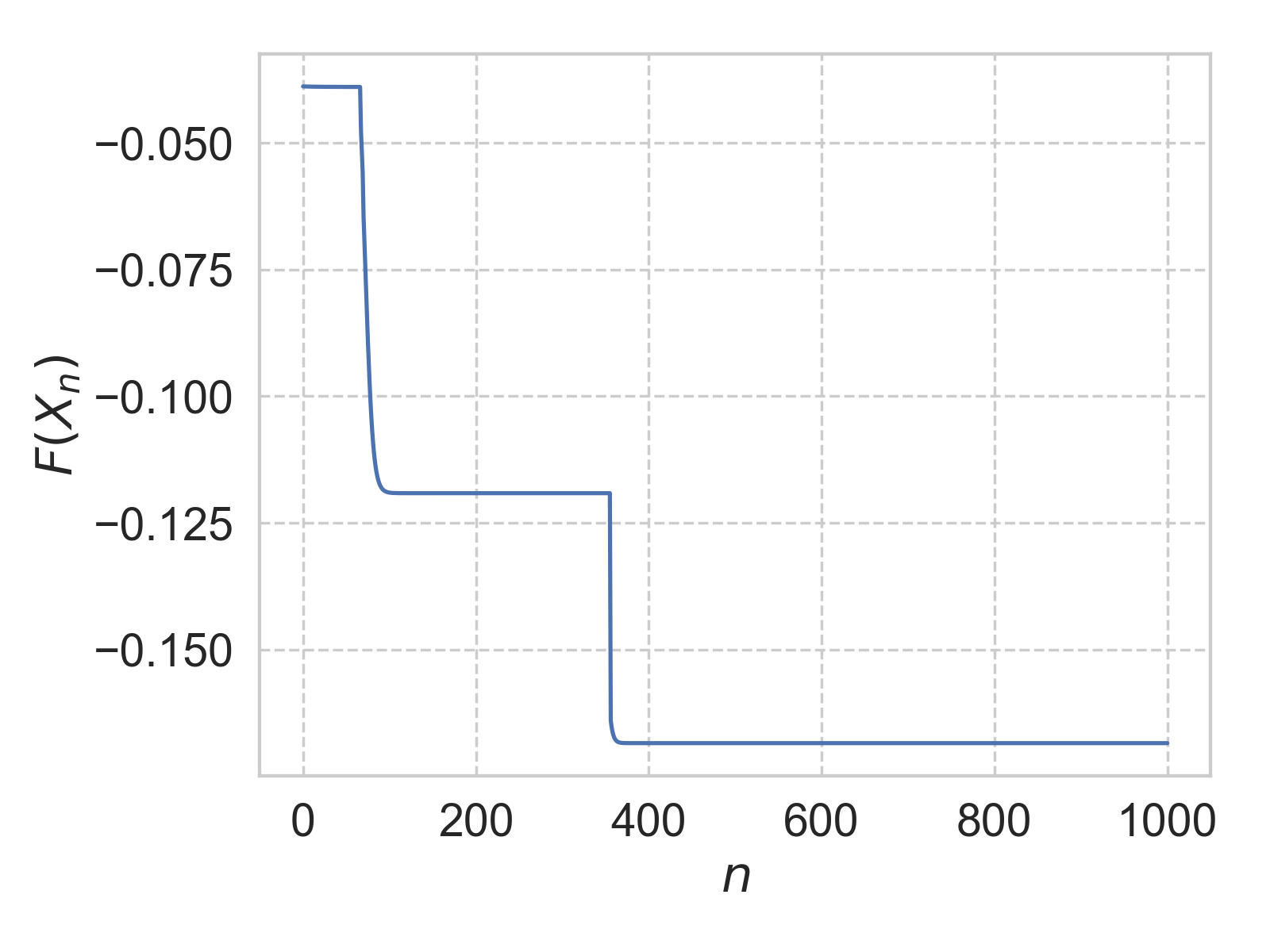}}
	\subfloat[$X_n$]{\includegraphics[width=0.48\textwidth]{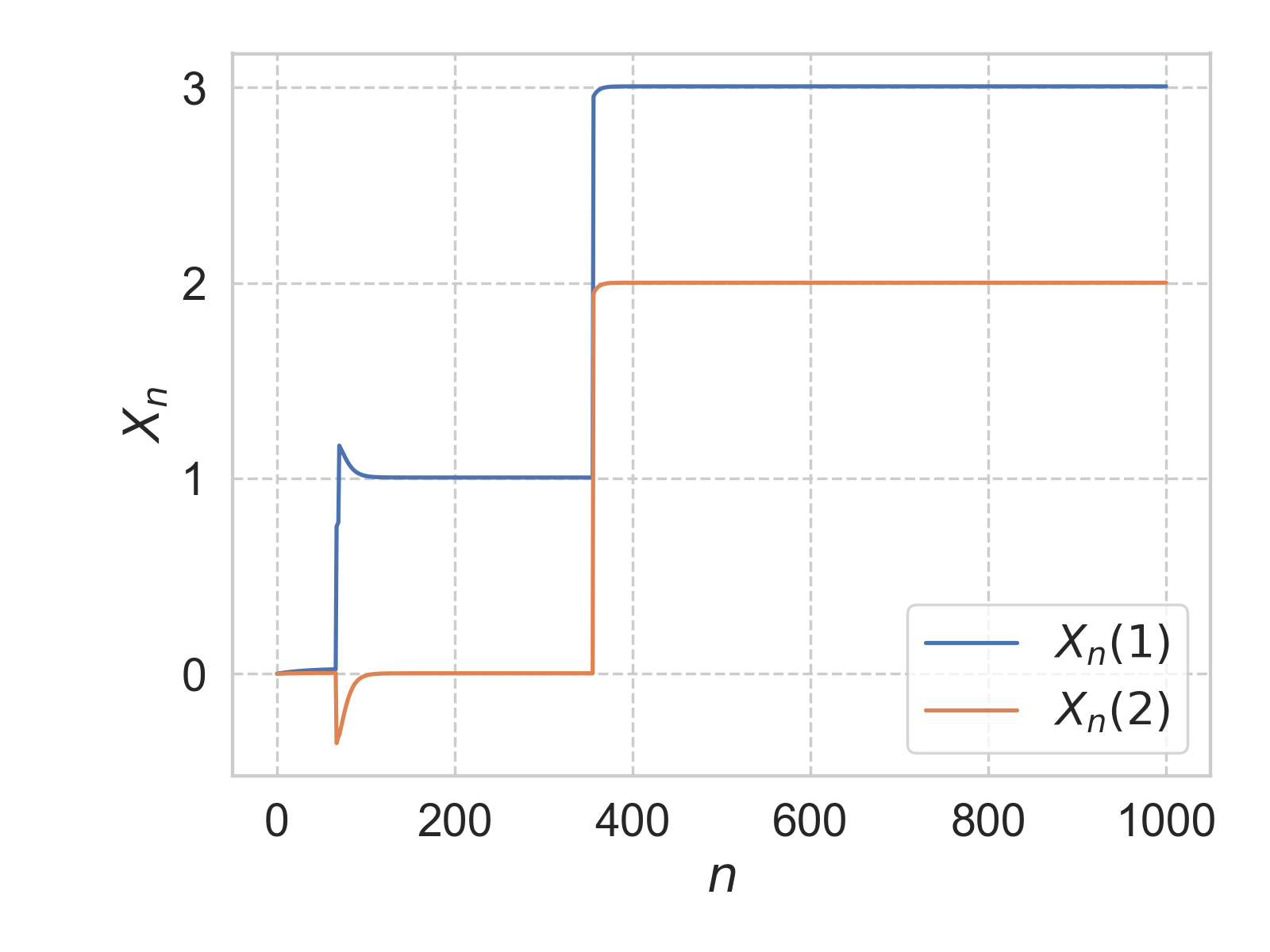}}
	\\	
	\caption{Convergence of the iterates under nGDxLD}
	\label{fig:nGDxLD}
\end{figure}


For comparison, we also implement GD and LD with the same $F$. For GD, the iteration takes the form
\[X_{n+1}=X_n- h\nabla F(X_n)\]
with $h=0.1$. In this case, $X_n$ gets stuck at different local minima depending on where we start, i.e., the value of $X_0$. For example, Figure \ref{fig:GD} plots the trajectories of $X_n$ under GD with $X_0=(0,0)$, which is the same as the $X_0$ we used in GDxLD.  As for LD,  Figure \ref{fig:LD} plots $X_n$ following
\[X_{n+1}=X_n - h\nabla F(X_n) + \sqrt{2\gamma h} Z_n.\]
We set $h=0.1$ and test two different values of $\gamma$:
$\gamma=1$, which is the $\gamma$ used in GDxLD, and $\gamma=0.01$.
When $\gamma=1$ (Figure \ref{fig:LD} (a)), we do not see convergence to the global minimum at all. The process 
is doing random exploration in the state-space. When $\gamma=0.01$  (Figure \ref{fig:LD} (b)), 
we do observe convergence of $X_n$ to the neighborhood of the global minimum.
However, compared with GDxLD, the convergence is much slower under LD, since the exploration is slowed down by the small $\gamma$. 
In particular, we find the approximate global minimum with around $1.2\times 10^5$ iterations.

\begin{figure}[htp]
\centering
\includegraphics[height=5.5cm]{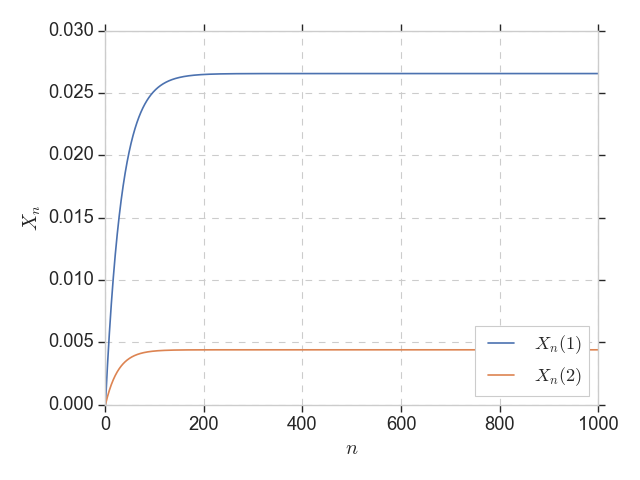}
\caption{Convergence of the iterates under GD}
\label{fig:GD}
\end{figure}

\begin{figure}[tbp]
	\centering
	\subfloat[$\gamma=1$]{\includegraphics[width=0.48\textwidth]{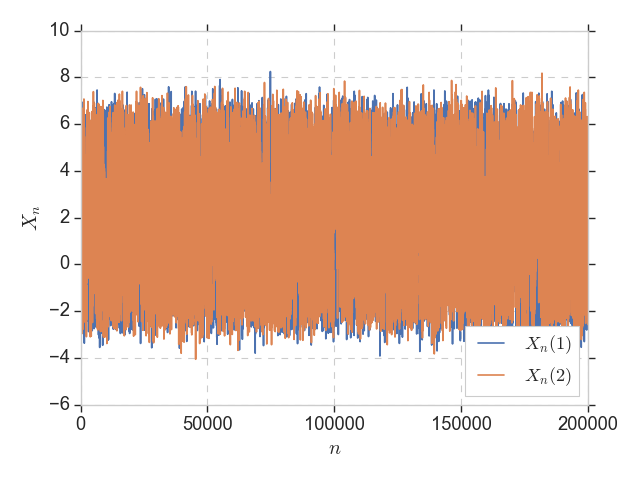}}
	\quad
	\subfloat[$\gamma=0.01$]{\includegraphics[width=0.48\textwidth]{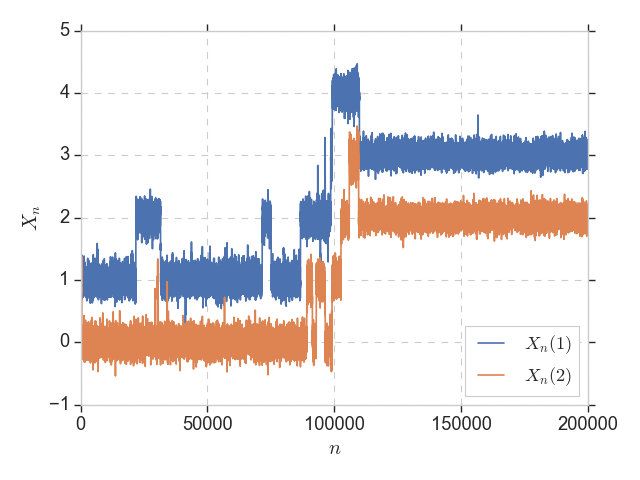}}
	\\	
	\caption{Convergence of the iterates under LD}
	\label{fig:LD}
\end{figure}



\subsubsection{Sensitivity analysis of the hyper-parameters} \label{sec:num_sen}
One attractive feature of GDxLD is that we do not require the temperature $\gamma$ and the step size $h$ to change with the precision level $\epsilon$.
In the theoretical analysis, we fix them as constants. From a practical point of view, we want $\gamma$ to be large enough so that it is easy for LD to escape the local minima. On the other hand, we do not want $\gamma$ to be too large as we want it to focus on exploring the ``relevant" region so that there is a good chance of visiting the neighborhood of the global minimum.
As for $h$, we want it to be small enough so that the GD converges once it is in the right neighborhood of the global minimum.
On the other hand, we do not want $h$ to be too small, as the convergence rate of GD, when it is in the right neighborhood, increases with $h$.

In this section, we conduct some sensitivity analysis for different values of $\gamma$ and $h$. We use the same objective function as the one plotted in Figure \ref{fig:ObjF}.

In Figure \ref{fig:sen_temp}, we keep $h=0.1$ fixed and vary the value of $\gamma$ from $0.1$ to $100$. The left plot shows $\E[F(X_n)]$ estimated based on $100$ independent replications of GDxLD. The right plot shows $\PP(\|X_n-x^*\|\leq 10^{-3})$, which is again estimated based on 100 independent replications. We observe that as long as $\gamma$ is not too small or too large, i.e., $0.5\leq \gamma \leq 10$, GDxLD achieves very fast convergence. For $0.5 \leq \gamma \leq 10$, the convergence speed is slightly different for different values of $\gamma$, with $\gamma=2.5, 5$ among the fastest. 

\begin{figure}[htp]
	\centering
	\subfloat{\includegraphics[width=0.48\textwidth]{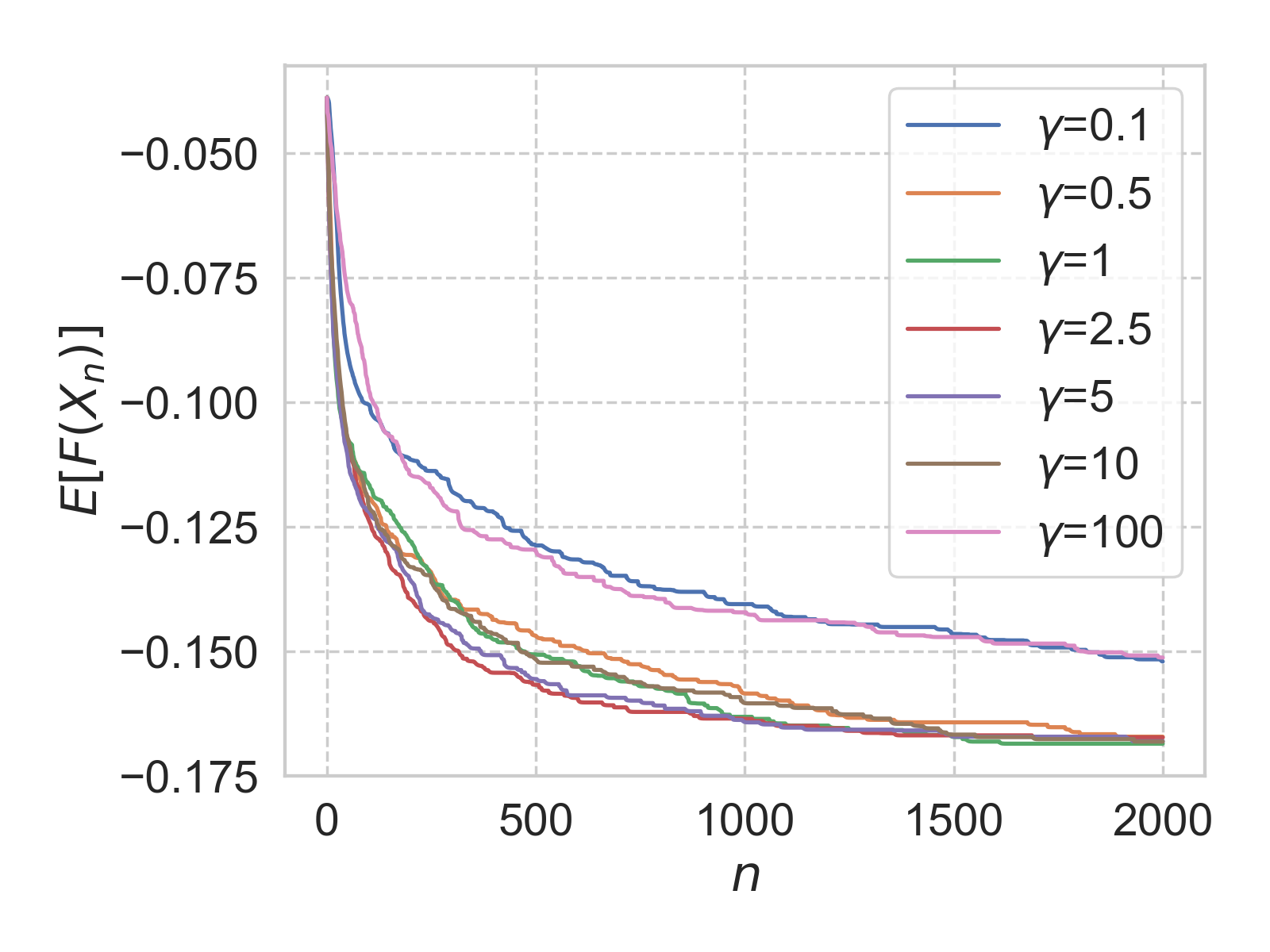}}
	\quad
	\subfloat{\includegraphics[width=0.48\textwidth]{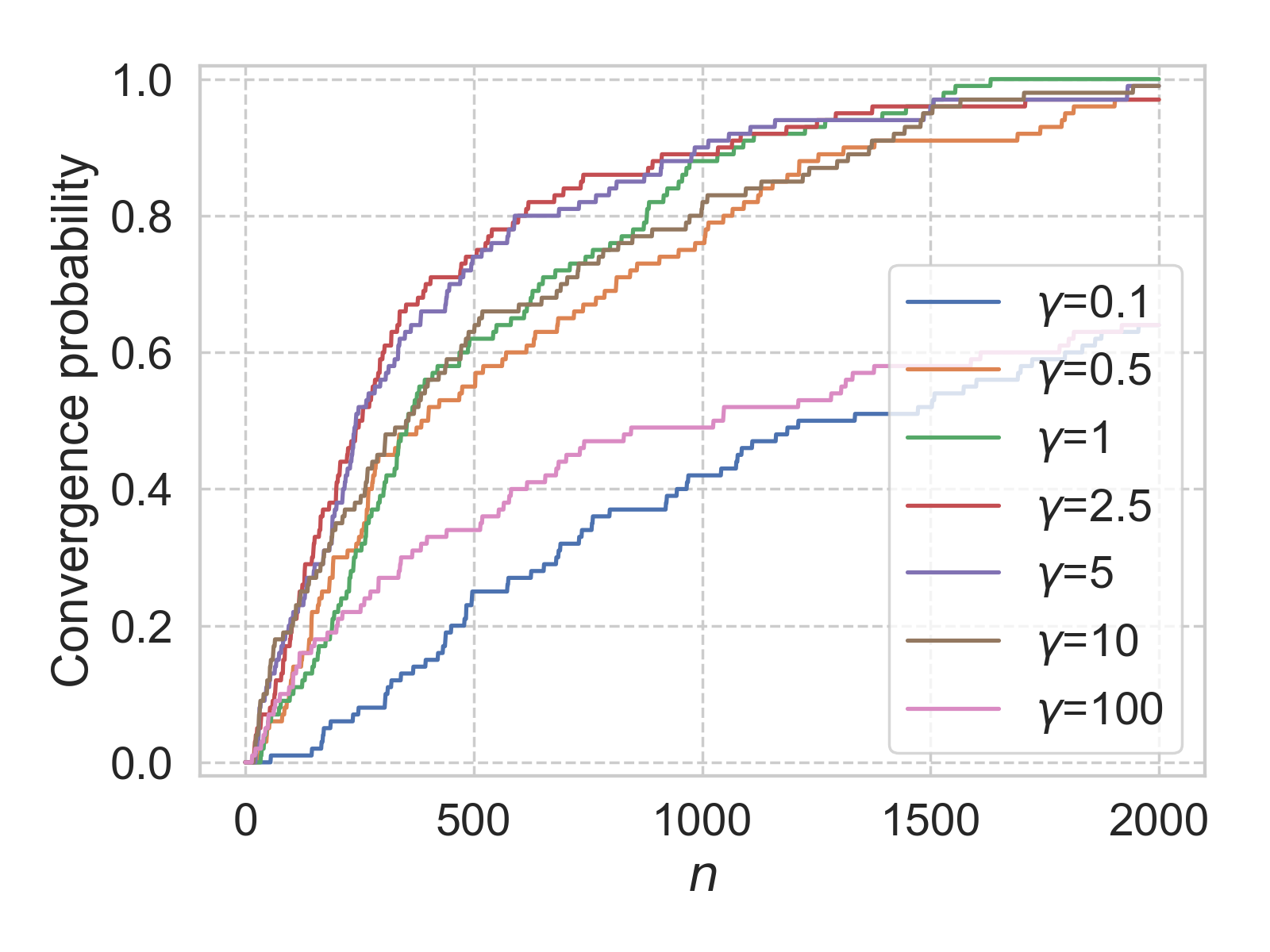}}
	\\	
	\caption{Convergence of the iterates under GDxLD with different values of $\gamma$. $h=0.1$. Averages are estimated based on $100$ independent replications.}
	\label{fig:sen_temp}
\end{figure}

In Figure \ref{fig:sen_step}, we keep $\gamma=1$ fixed and vary the value of $h$ from $0.1$ to $1.5$. The left plot shows $\E[F(X_n)]$ and the right plot shows $\PP(\|X_n-x^*\|\leq 10^{-3})$. We observe that as long as $h$ is not too big or too small, i.e., $0.05\leq \gamma \leq 1$, GDxLD achieves very fast convergence. Taking a closer look at the sample path of GDxLD when $h=1.5$, we note that $X_n$ is swapped to the region around the global minimum fairly quickly but it keeps oscillating around the global minimum due to the large step size. For $0.05\leq h\leq 1$, the convergence speed is slightly different for different values of $h$ with $h=1$ being the fastest and $h=0.05$ being the slowest.

\begin{figure}[htp]
	\centering
	\subfloat{\includegraphics[width=0.48\textwidth]{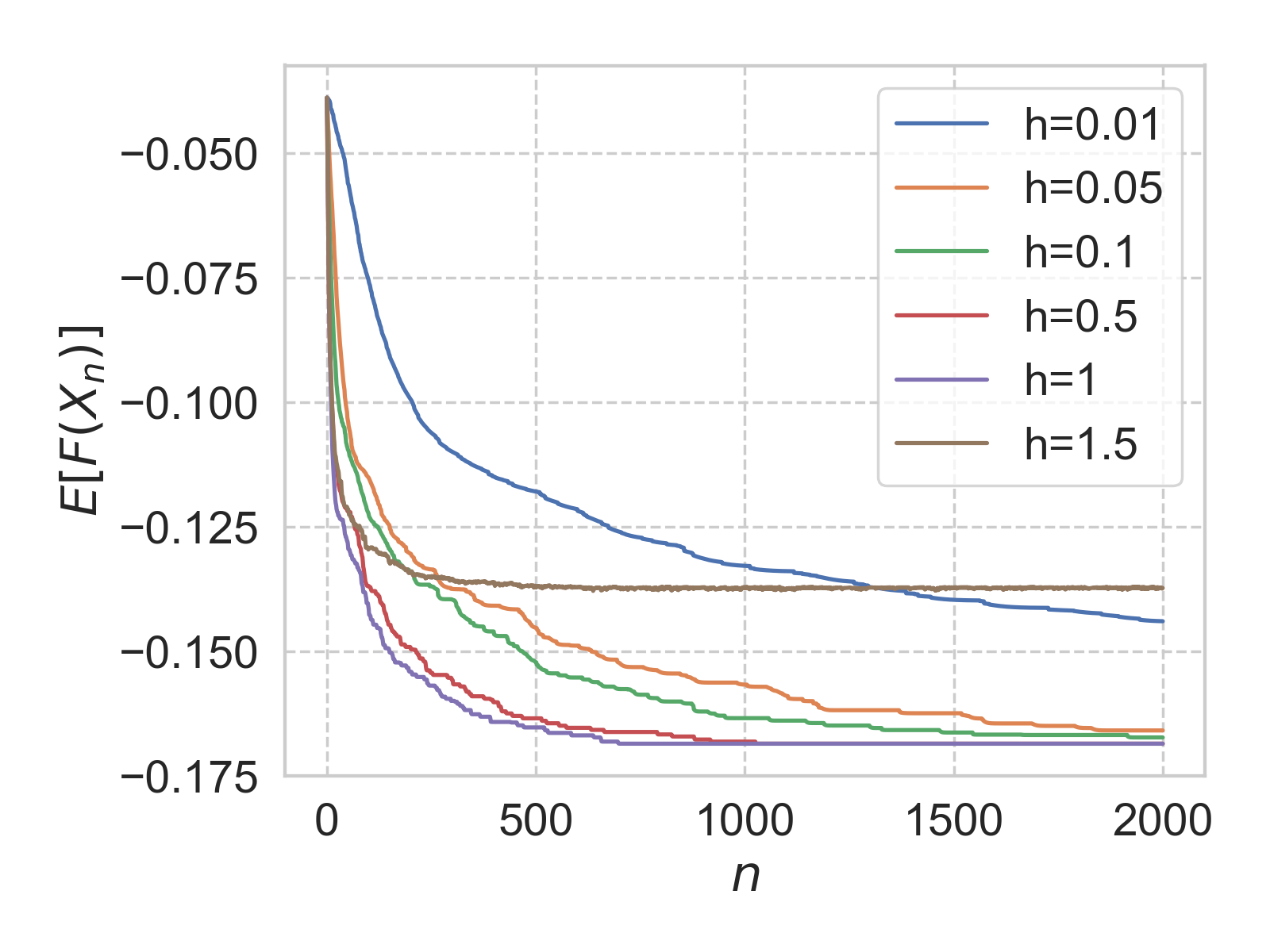}}
	\quad
	\subfloat{\includegraphics[width=0.48\textwidth]{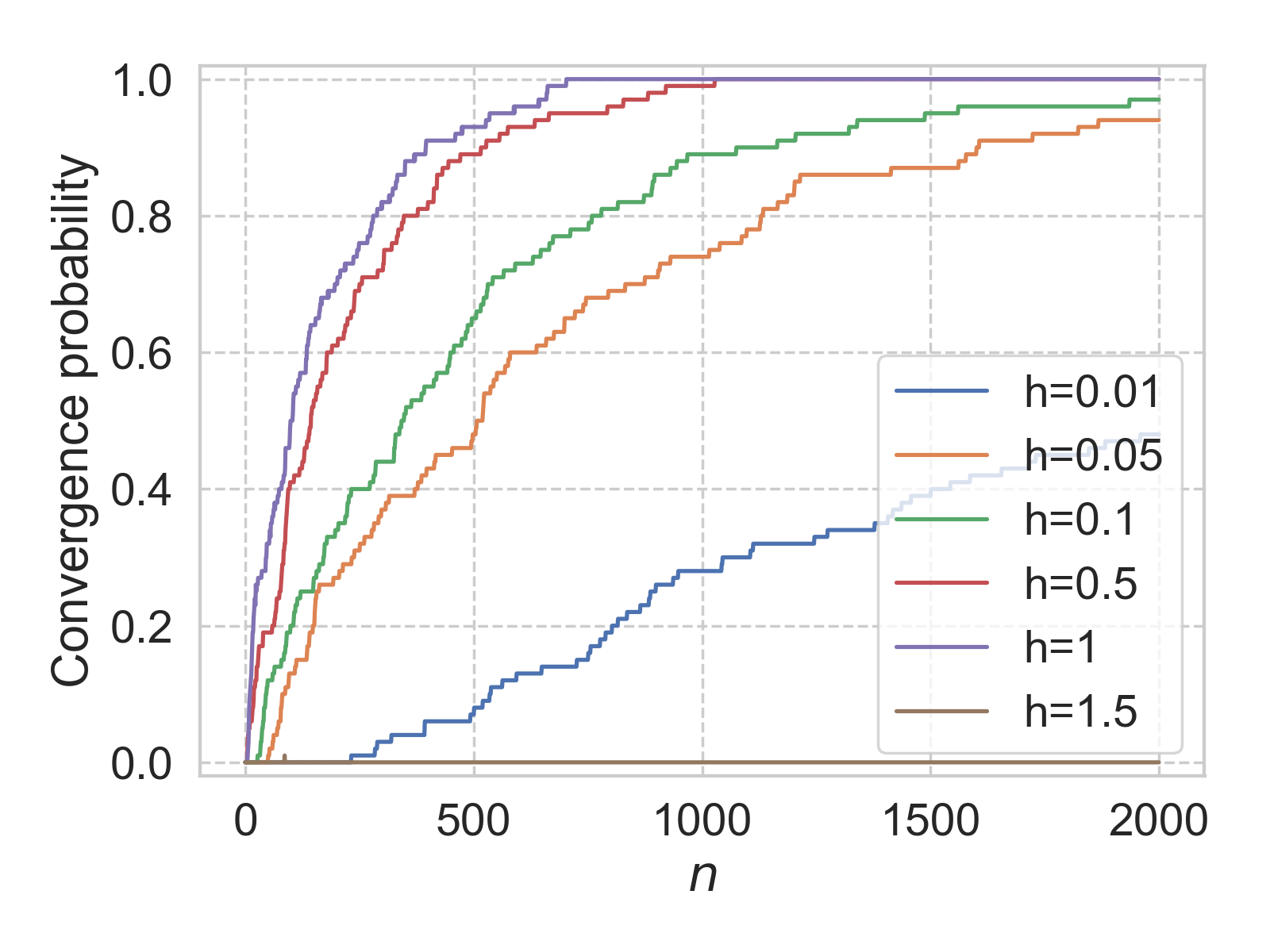}}
	\\	
	\caption{Convergence of the iterates under GDxLD with different values of $h$. $\gamma=1$. Averages are estimated based on $100$ independent replications.}
	\label{fig:sen_step}
\end{figure}

Above all, our numerical experiments suggest that while different hyper-parameters may lead to different performances of GDxLD, the differences are fairly small as long as $\gamma$ and $h$ are within a reasonable range. This suggests the robustness of GDxLD to the hyper-parameters.

In Figure \ref{fig:sen_nGDxLD}, we conduct the same analysis for nGDxLD. In particular, we plot $\E[F(X_n)]$ at different iterations for different values of the temperature $\gamma$ and step size $h$ in nGDxLD. The performances are very similar as those in GDxLD. In all our subsequent numerical experiments, we implement both GDxLD and nGDxLD. Since their performances are very similar, we only show the results for GDxLD in the figures.

\begin{figure}[tbp]
	\centering
	\subfloat[$h=0.1$]{\includegraphics[width=0.48\textwidth]{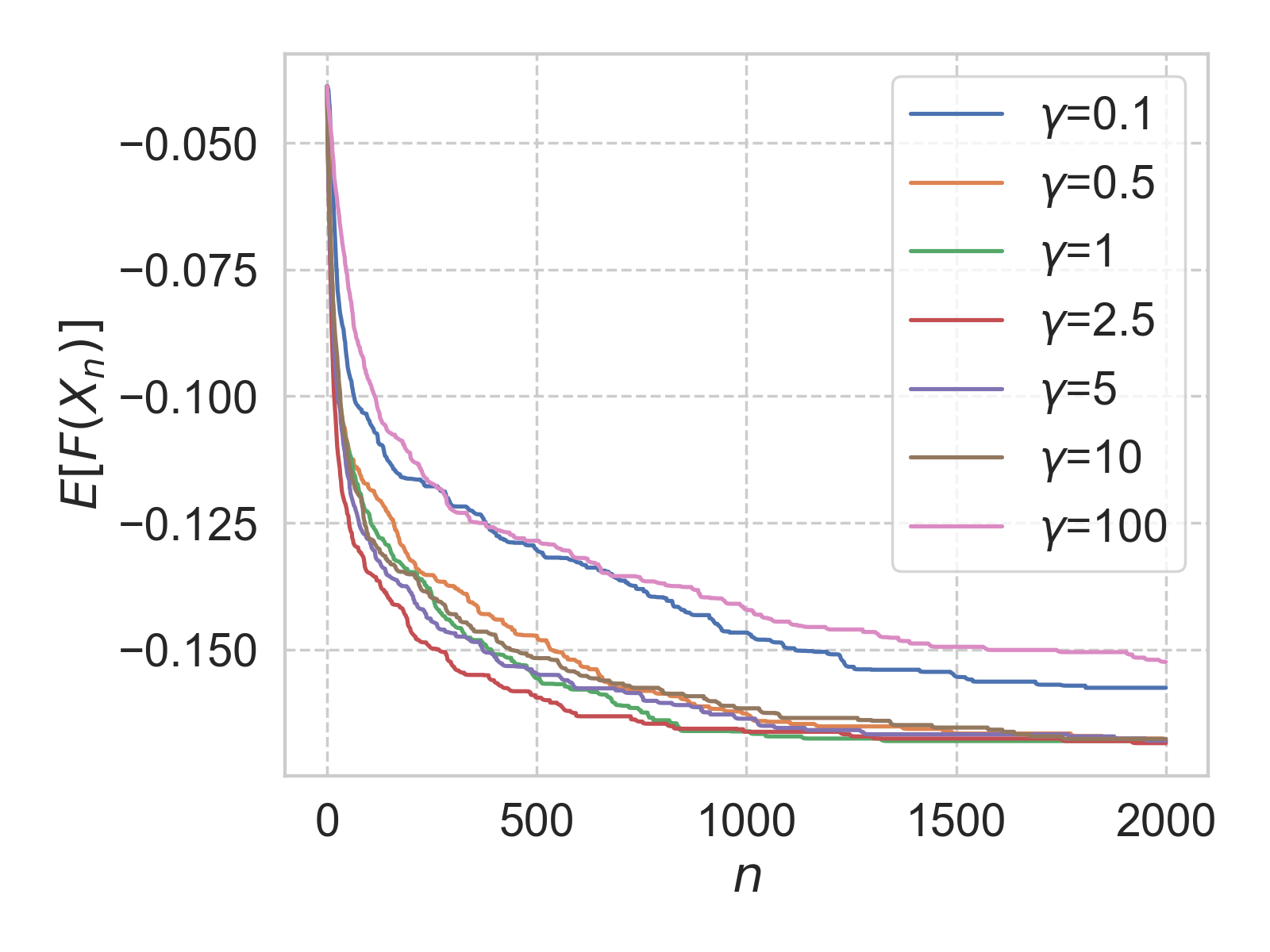}}
	\quad
	\subfloat[$\gamma=1$]{\includegraphics[width=0.48\textwidth]{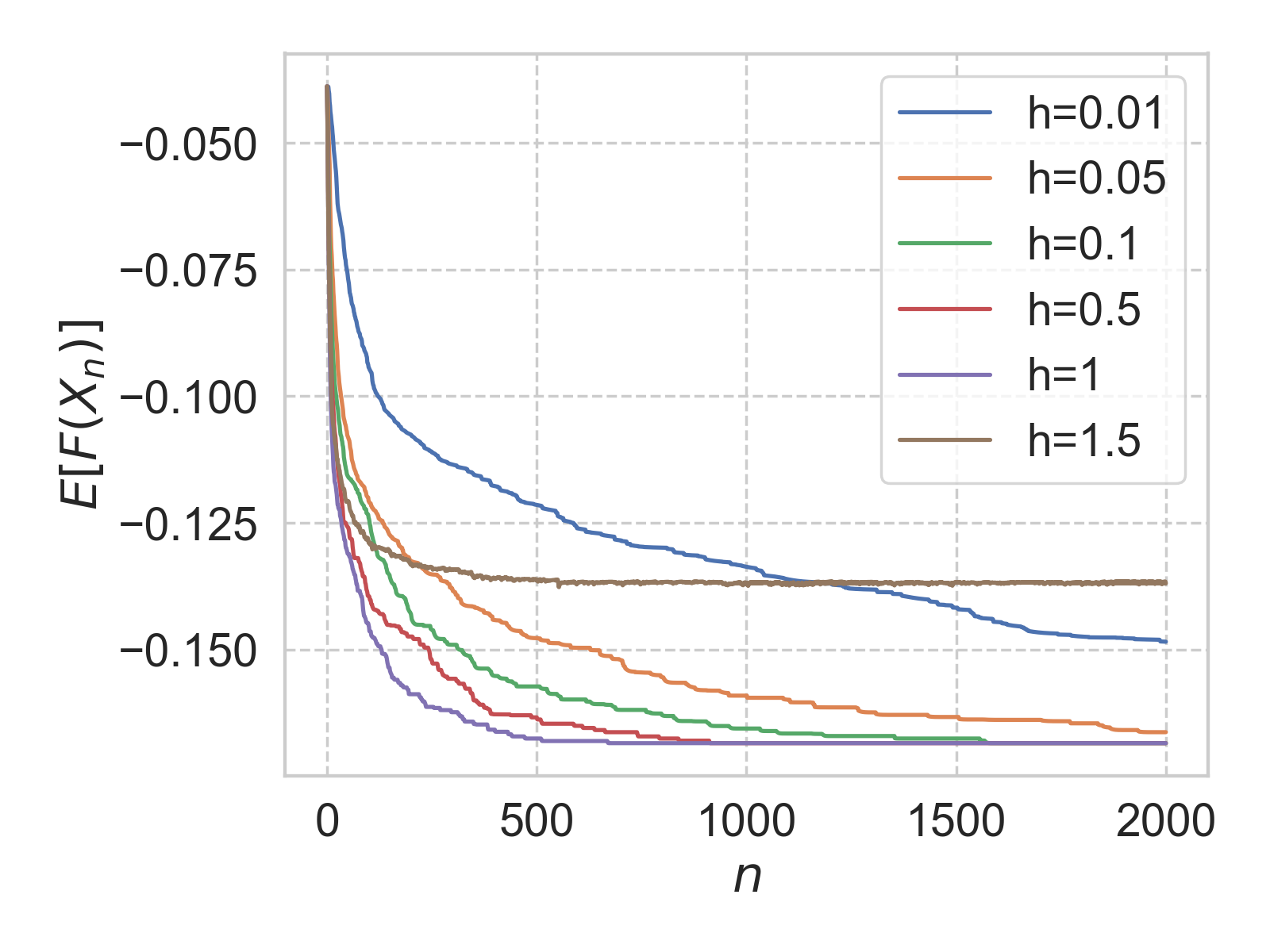}}
	\\	
	\caption{Convergence of the iterates under nGDxLD with different values of $\gamma$ and $h$. Averages are estimated based on $100$ independent replications.}
	\label{fig:sen_nGDxLD}
\end{figure}

\subsubsection{Online setting}
\label{sec:onlinenum}
In this section, we consider an online version of the test problem from Section \ref{sec:offnum}.
In particular, we consider the setting of kernel density estimation (KDE)
\[
\hat{p}_N(x)=\frac{1}{N} \sum_{i=1}^N\kappa_\sigma(x,s_i).
\]
$\kappa_\sigma$ is known as a kernel function with tuning parameter $\sigma$. 
It measures the similarity between $x$ and the sample data $s_i$'s. 
There are many choices of kernel functions, and here we use the Gaussian kernel 
\[
\kappa_\sigma(x,s_i)=\frac{1}{(2\pi\sigma)^{\frac{d}{2}}}\exp(-\tfrac{1}{2\sigma}(x-s_i)^2).
\]
Then,  $\hat{p}_{N}(x)$ can be seen as a sample average version of 
\[
p(x)=\E_{s} \kappa_\sigma(x,s).
\]
Notably, $p$ is the density function of $X=S+\sqrt{\sigma}Z$, where $S$ follows the distribution of the sample data and $Z\sim \mathcal{N}(0,1)$. 
In the following example, we assume $S$ follows a mixture of Gaussian distribution with density
\begin{equation}
\label{eqn:mix}
f(x)=\sum_{i=1}^M \frac{w_i }{\sqrt{\text{det}(2\pi \Sigma_i)}} \exp\left(-\tfrac{1}{2} (x-m_i)^T\Sigma_i^{-1}(x-m_i)\right). 
\end{equation} 
As in Section \ref{sec:offnum}, we set $M=5\times 5=25$, each $m_i$ is a point in the meshgrid $\{0,1,2,3,4\}^2$, $\Sigma_i=$diag$(0.1)$, and the weights are randomly generated. Our goal is to find the mode of $p$. 
In this case, we write
\[
F(x)=-p(x)+L(x)=-\E_{s} \kappa_\sigma(x,s)+L(x),
\]
where $L$ is the quadratic regularization function defined in \eqref{eq:reg}.
Then, we can run SGDxSGLD with the mini-batch average approximations of $F$ and $\nabla F$:
\[
\hat F_n(X_n)=\frac1\Theta \sum_{i=1}^\Theta \kappa_\sigma (X_n,s_{n,i})+L(x), \mbox{ and } \nabla \hat F_n(X_n)=\frac1\Theta \sum_{i=1}^\Theta \nabla_x \kappa_\sigma (X_n,s'_{n,i})+\nabla L(x),
\]
where the data-specific gradient takes the form 
\[
\nabla_x \kappa_\sigma(x,s_i)=\frac{1}{(2\pi\sigma)^{\frac{d}{2}}\sigma}\exp(-\tfrac{1}{2\sigma}|x-s_i|^2)(x-s_i).
\]
In Figure \ref{fig:ObjK}, we plot the heat map and 3-plot of one possible realization of $\hat F_n$ with $\sigma=0.1^2$ and $n=10^4$. Note that in this particular realization, the global minimum is achieved at $(3,2)$.

\begin{figure}[tbp]
	\centering
	\subfloat[Heat map of $\hat F_n$]{\includegraphics[width=0.48\textwidth]{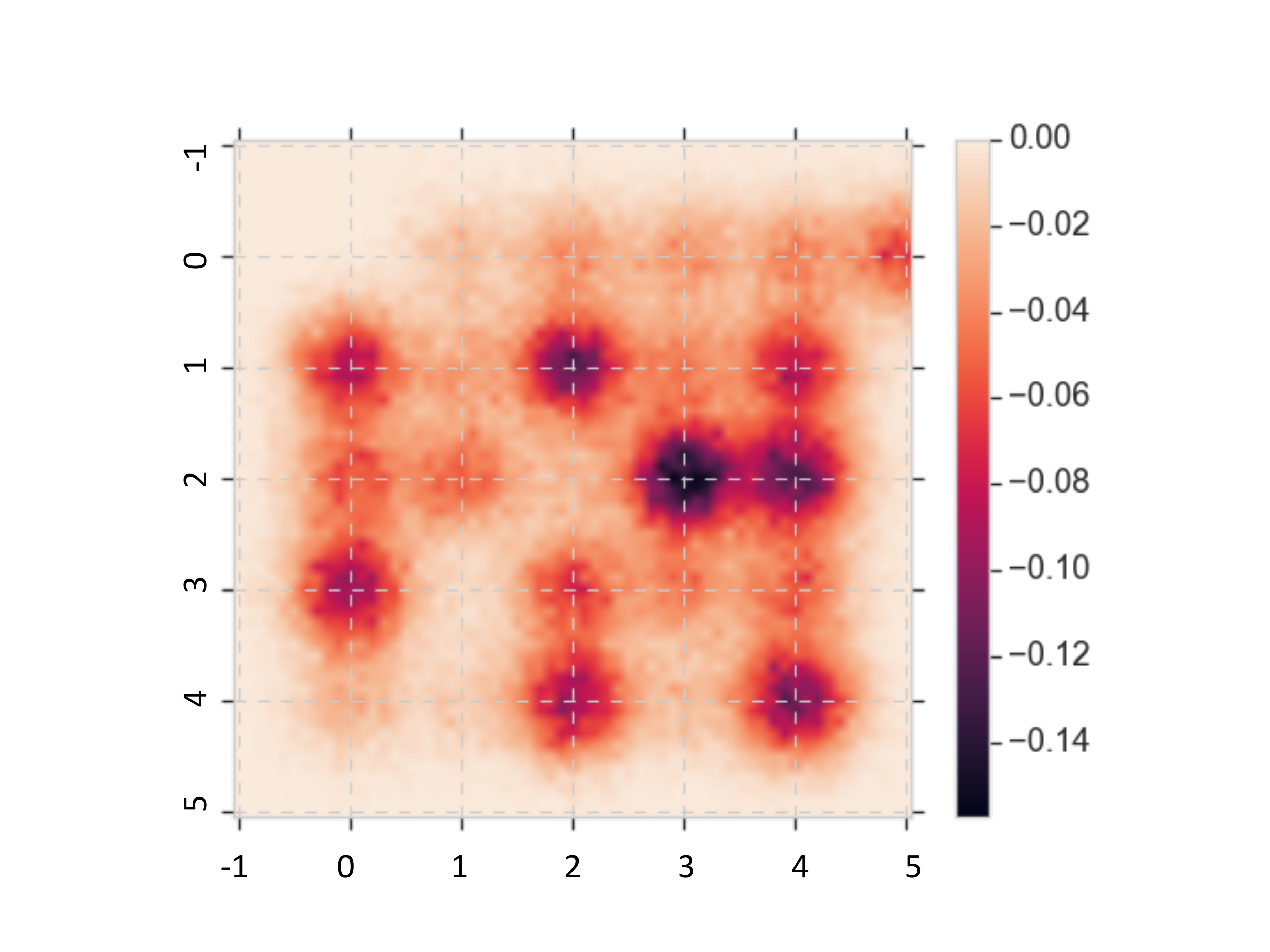}}
	\subfloat[3-d plot of $\hat F_n$]{\includegraphics[width=0.48\textwidth]{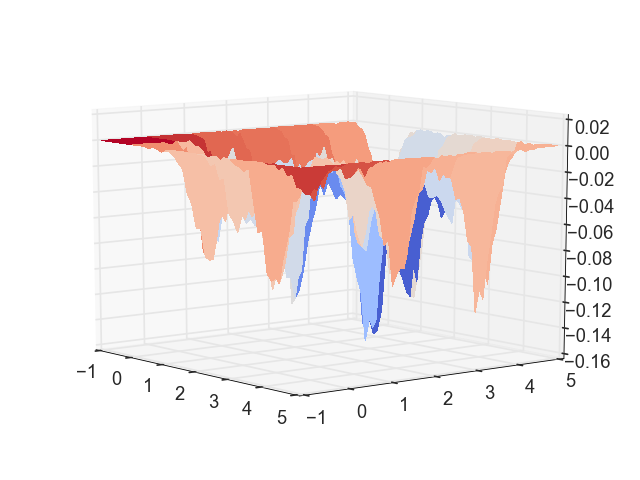}}
	\\	
	\caption{The estimated objective function}
	\label{fig:ObjK}
\end{figure}


In Figure \ref{fig:SGDxSGLD}, we plot $X_n$ for different values of $n$ under SGDxSGLD with the objective
function plotted in Figure \ref{fig:ObjK}.
We set $h=0.1$, $\gamma=1$, $\Theta=10^3$, $t_0=0.05$, $\hat M_v=5$, $X_0=(0,0)$, and $Y_0=(1,1)$.
We observe that SGDxSGLD converges to the approximate global minimum very fast, within $10^3$ iterations.

\begin{figure}[htp]
\centering
\includegraphics[height=5.5cm]{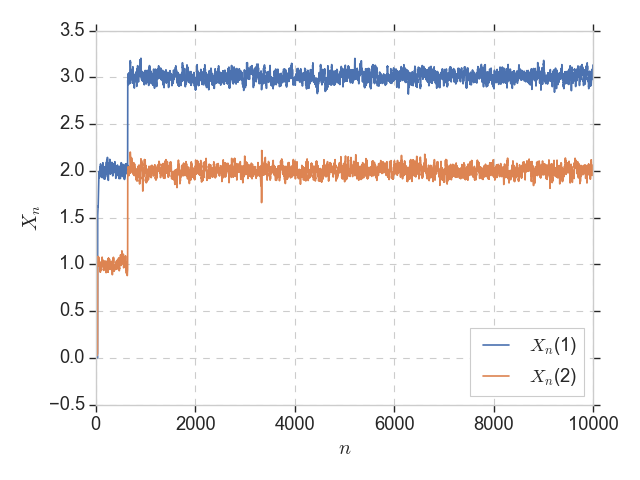}
\caption{Convergence of the iterates under SGDxSGLD}
\label{fig:SGDxSGLD}
\end{figure}

For comparison, in Figure \ref{fig:SGcomp}, we plot the sample path of $X_n$ under SGD and SGLD with the objective
function plotted in Figure \ref{fig:ObjK}.
For SGD, the iteration takes the form
\[
X_{n+1}=X_n - h\nabla \hat F(X_n).
\]
For SGLD, the iteration takes the form
\[
X_{n+1}=X_n - h\nabla \hat F(X_n) + \sqrt{2\gamma h} Z_n.
\]
We set $h=0.1$, $\gamma=0.01$, and $\Theta=10^3$. Note that $\gamma=0.01$ is tuned to ensure convergence.
We observe that SGD still gets stuck in local minima. For example, in Figure \ref{fig:SGcomp} (a), when $X_0=(0,0)$,
$X_n$ gets stuck at $(0,1)$.
SGLD is able to attain the global minimum, but at a much slower rate than SGDxSGLD.
In particular, SGLD takes more than $2 \times 10^4$ iterations to converge to the approximate global minimum in Figure \ref{fig:SGcomp} (b).

\begin{figure}[tbp]
	\centering
	\subfloat[SGD]{\includegraphics[width=0.48\textwidth]{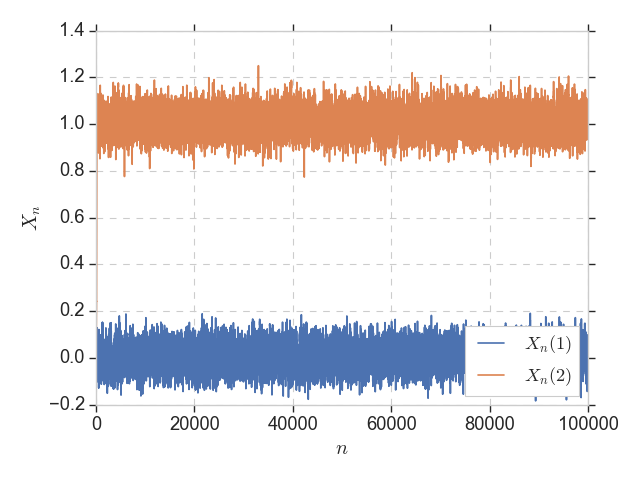}}
	\quad
	\subfloat[SGLD]{\includegraphics[width=0.48\textwidth]{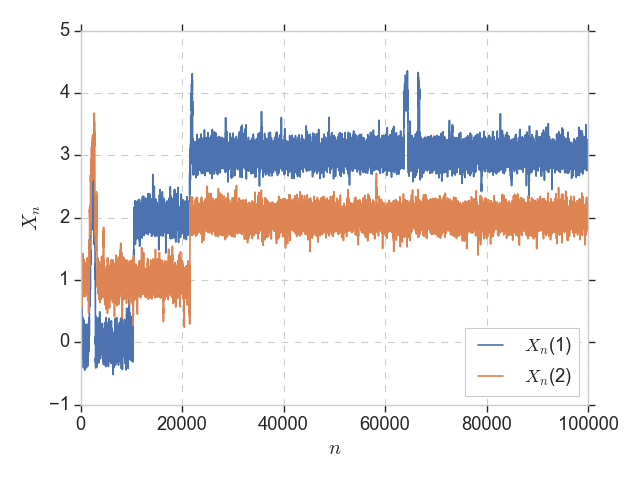}}
	\\	
	\caption{Convergence of the iterates under SGD and SGLD}
	\label{fig:SGcomp}
\end{figure}


\subsection{Non-Convex Optimization Problems with $d\geq 2$}
In this section, we demonstrate the performance of (n)GDxLD on some classical non-convex optimization test problems.
In particular, we consider the Rastrigin function
\[
F(x)=10n + \sum_{i=1}^{d}\left(x_i^2-10\cos(2\pi x_i)\right)
\]
restricted to $x\in[-5,5]^d$.
We also consider the Griewank function
\[
F(x)=\sum_{i=1}^{d}\frac{x_i^2}{4000} - \prod_{i=1}^{d}\cos\left(\frac{x_i}{\sqrt{i}}\right) + 1
\]
restricted to $x\in[-5,5]^d$. For both functions, the global minimum is at the origin, i.e., $x^*=(0,\dots, 0)$, with $F(x^*)=0$.
We select these two functions because they are classic simple test problems for non-convex, multimodal optimization algorithms. 
Existing tools for optimizing these functions often involve metaheurstics, which lack rigorous complexity analysis \cite{gamperle2002parameter, esquivel2003use, cheng2014competitive}.  
Figure \ref{fig:nonconvex_F} provides an illustration of the two test functions when $d=2$.
We note that the Rastrigin function is especially challenging to optimize as it has many local minima, while
the Griewank function restricted to $[-5,5]^d$ has a relatively smooth landscape with only a few local minima.

\begin{figure}[tbp]
	\centering
	\subfloat[Rastrigin Function]{\includegraphics[width=0.5\textwidth]{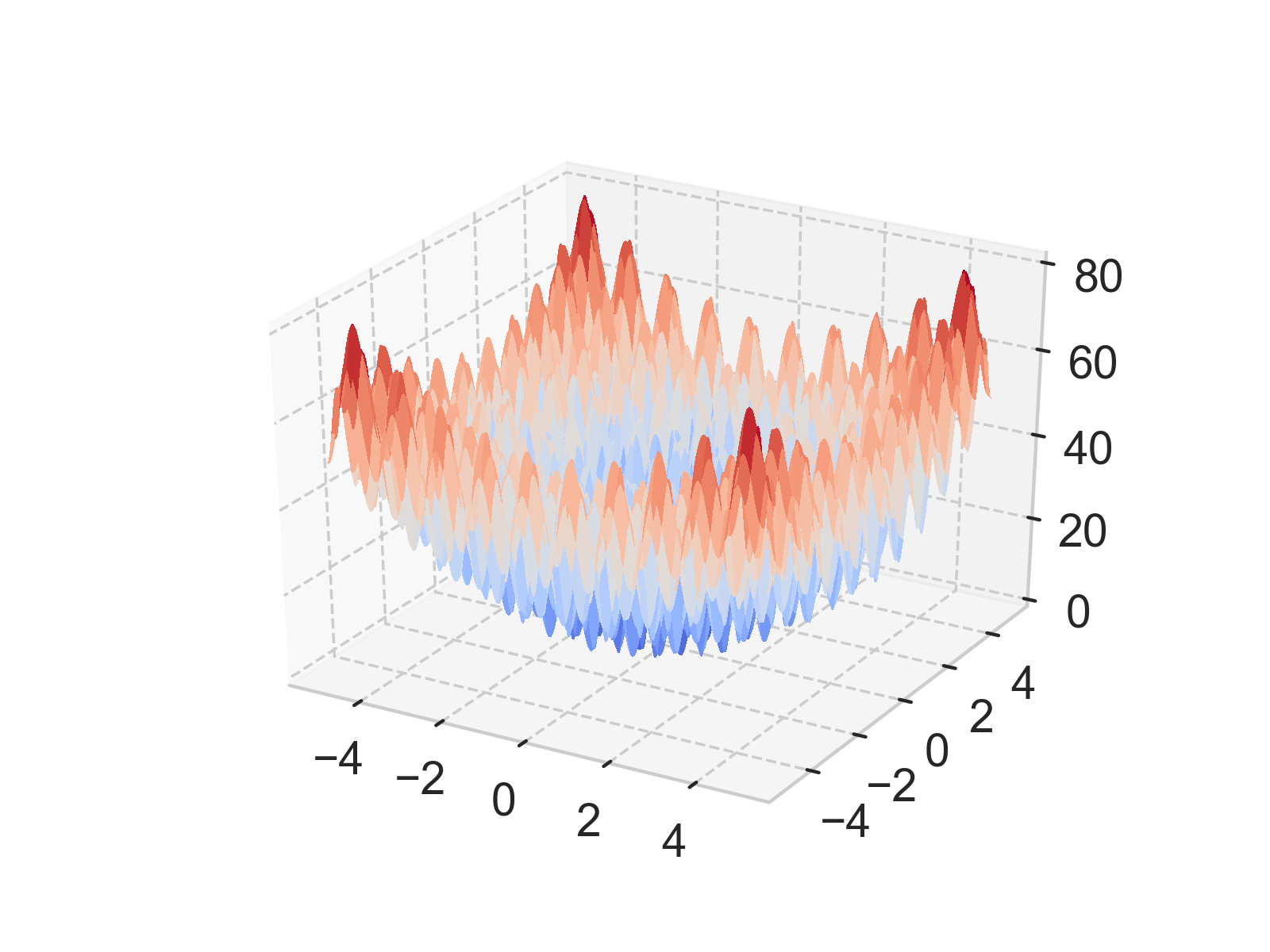}}
	\subfloat[Griewank Function]{\includegraphics[width=0.5\textwidth]{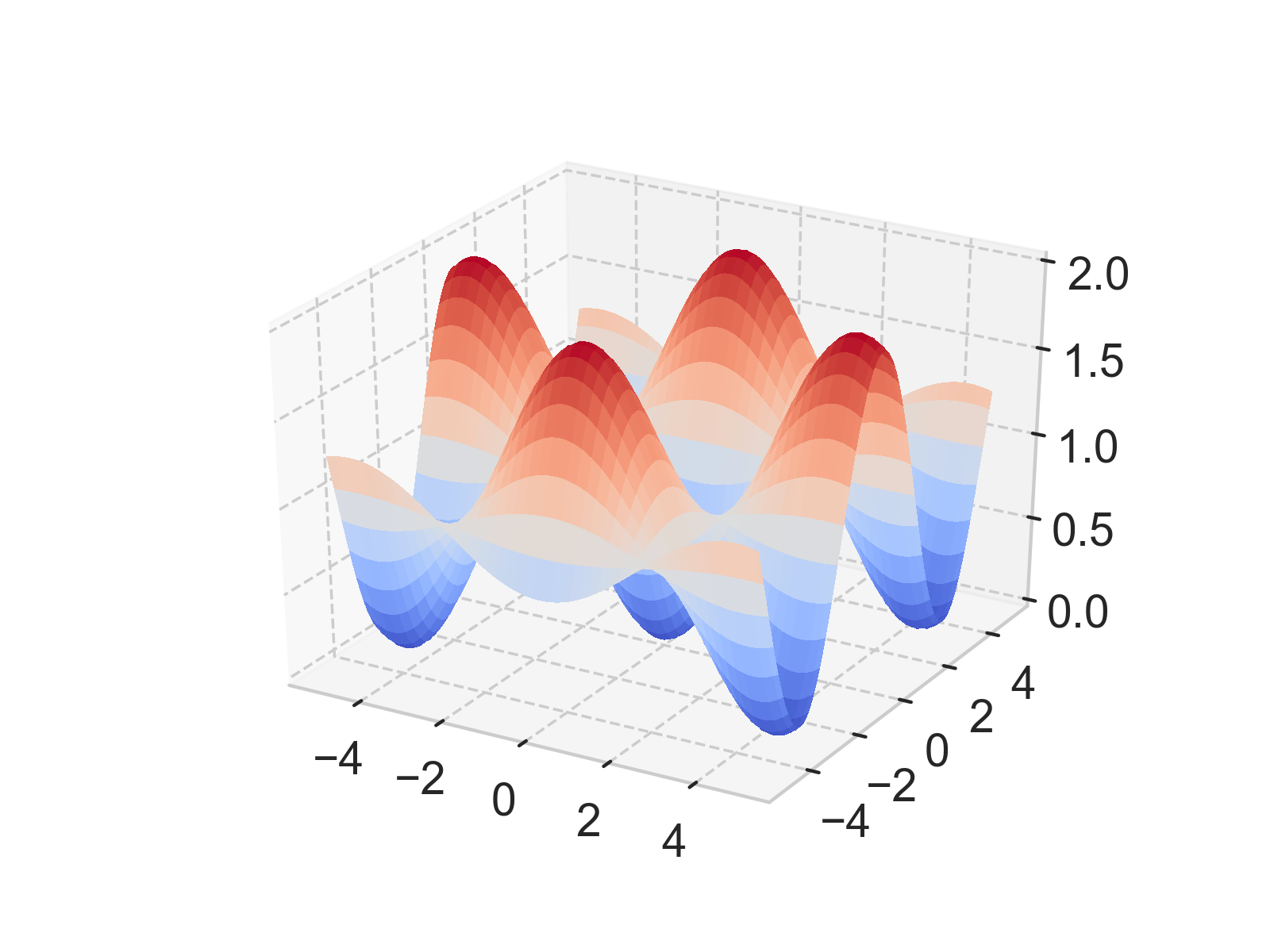}}
	\\	
	\caption{3-d plot of $F$ when $d=2$}
	\label{fig:nonconvex_F}
\end{figure}

For Rastrigin functions of different dimensions, we plot $F(X_n)$ under at different iterations under GDxLD in Figure \ref{fig:rastrigin}.
We observe that as $d$ increases, it takes longer to find the global minimum. For example, when $d=2$, 
it takes around $5\times 10^2$ iterations to find the global minim, while when $d=5$, it takes around $7\times 10^4$ iterations to find the 
the global minimum. When $d \geq 10$, we are not able to find the global minimum within $10^5$ iterations, but the function value reduces substantially. nGDxLD achieves very similar performances as GDxLD. To avoid repetition, we do not include the corresponding plots here.

\begin{figure}[tbp]
	\centering
	\subfloat[$d=2$]{\includegraphics[width=0.33\textwidth]{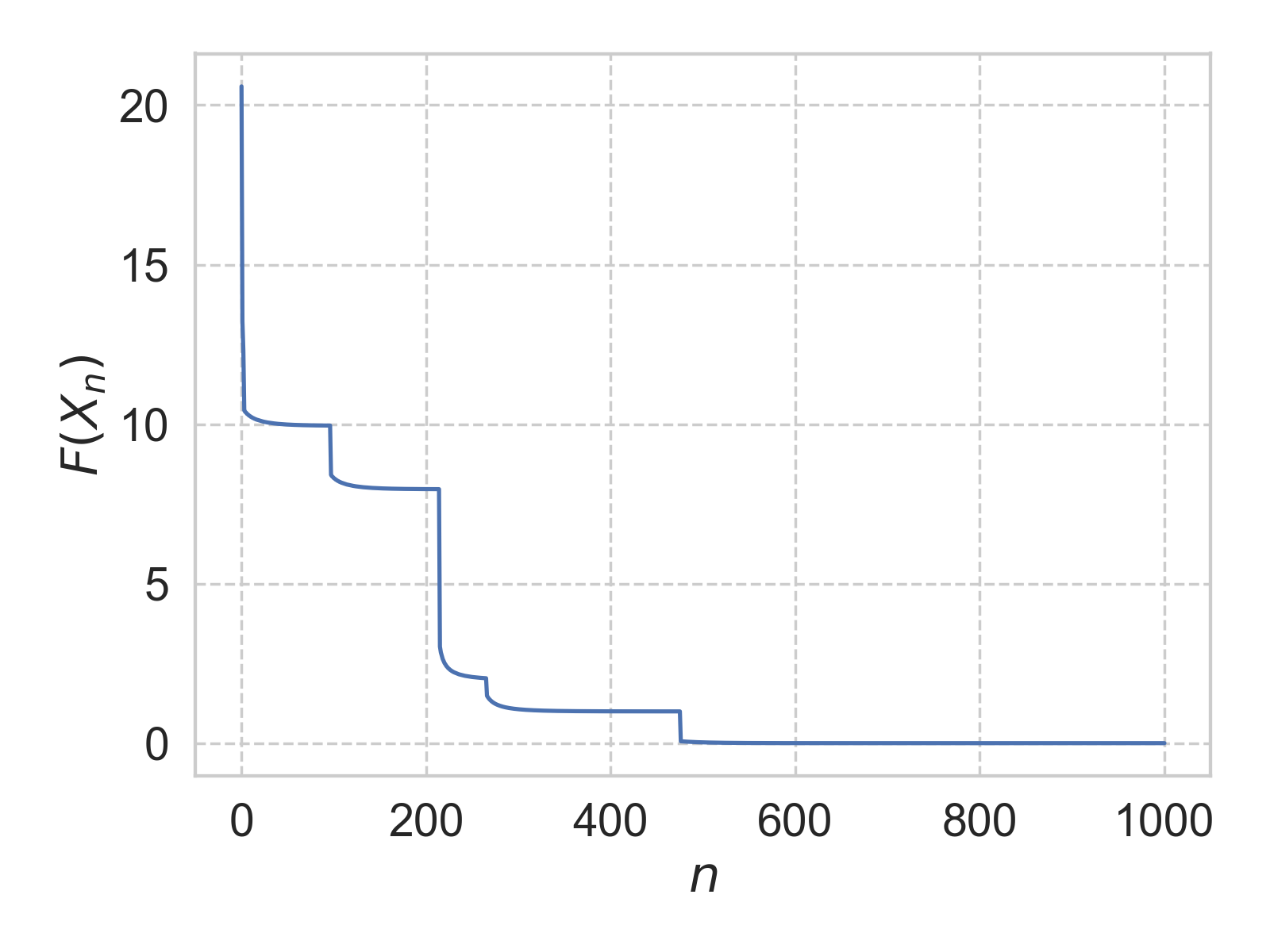}}
	\subfloat[$d=3$]{\includegraphics[width=0.33\textwidth]{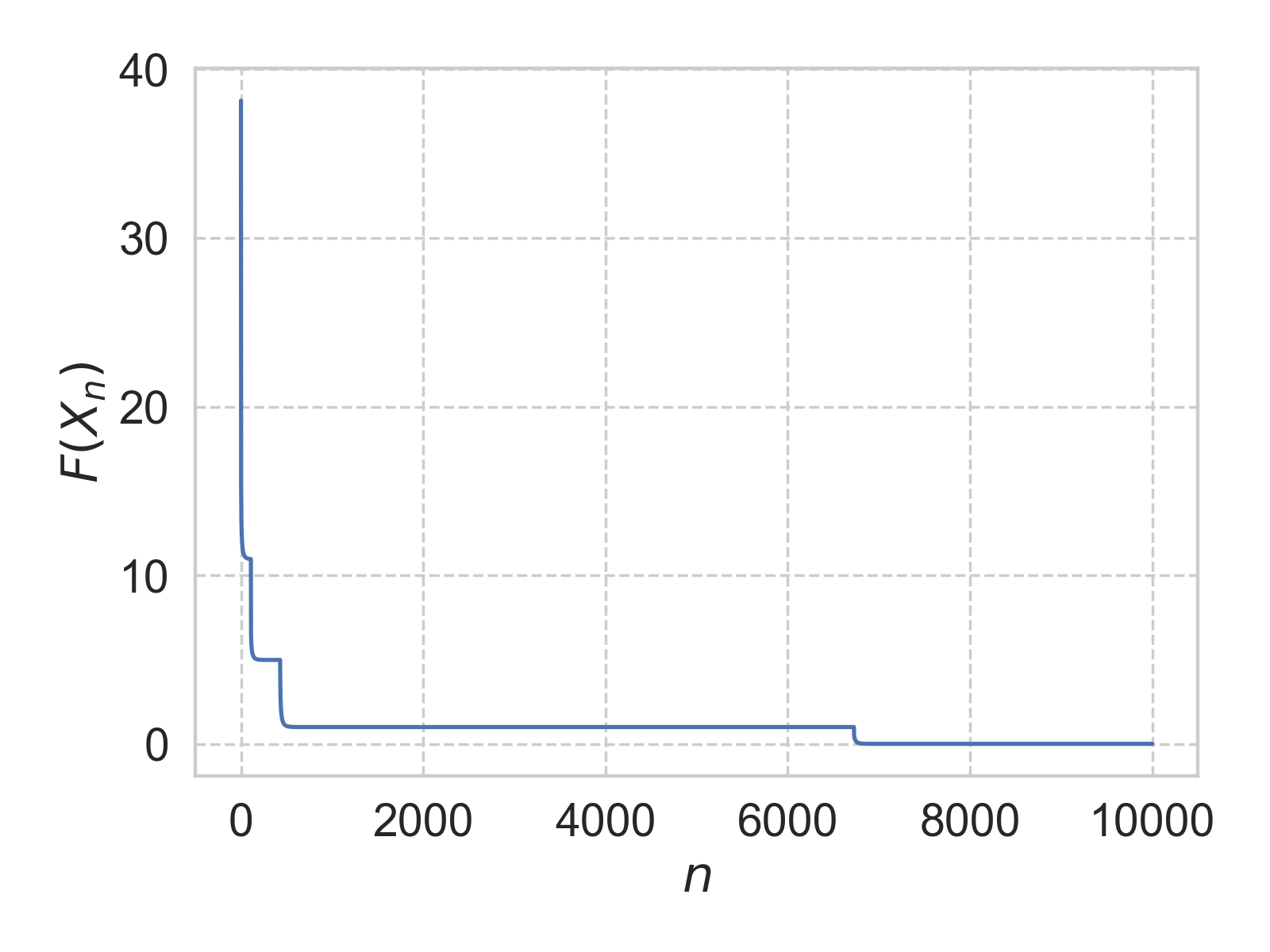}}
	\subfloat[$d=5$]{\includegraphics[width=0.33\textwidth]{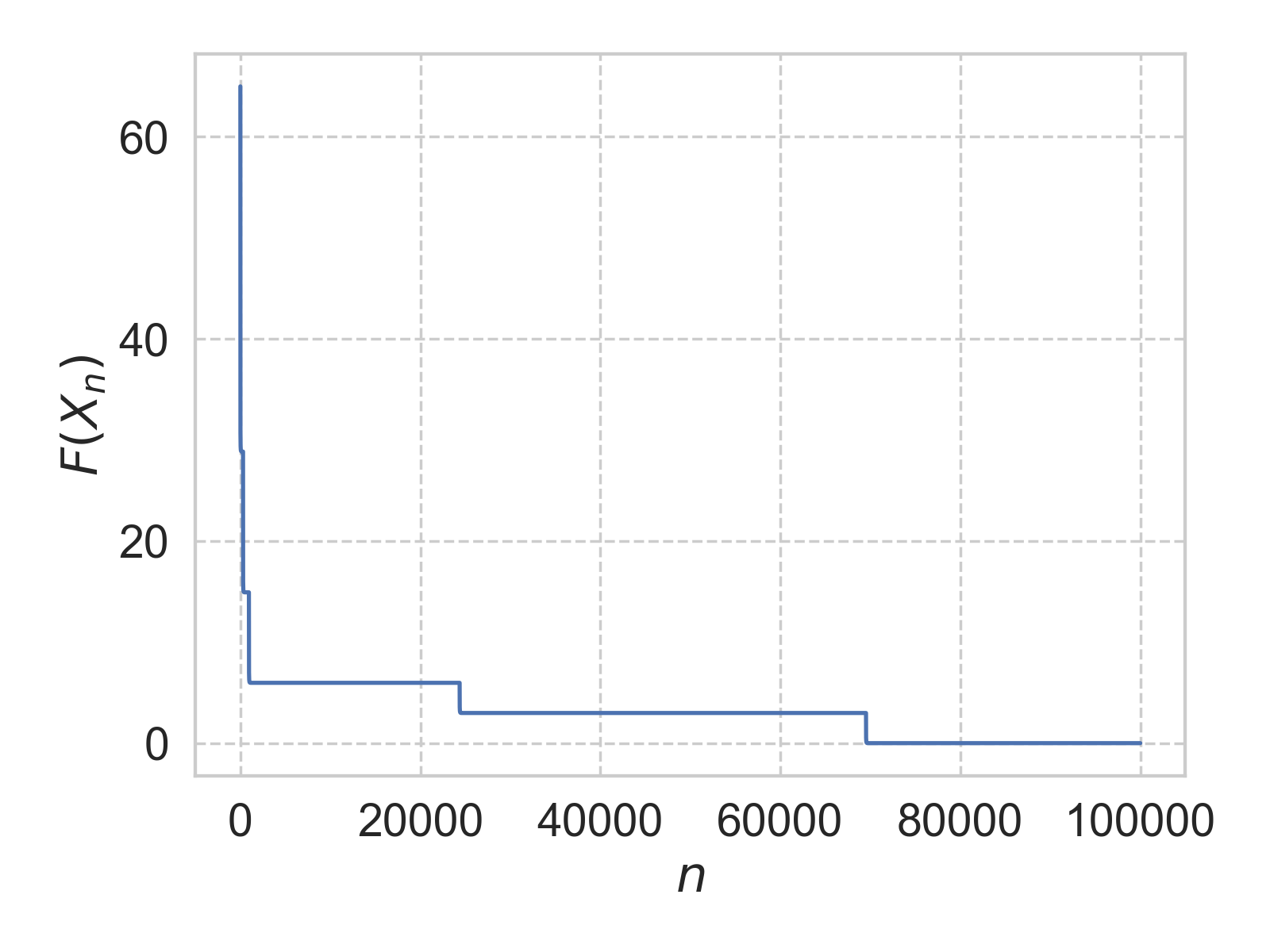}}
	\\
	\subfloat[$d=10$]{\includegraphics[width=0.33\textwidth]{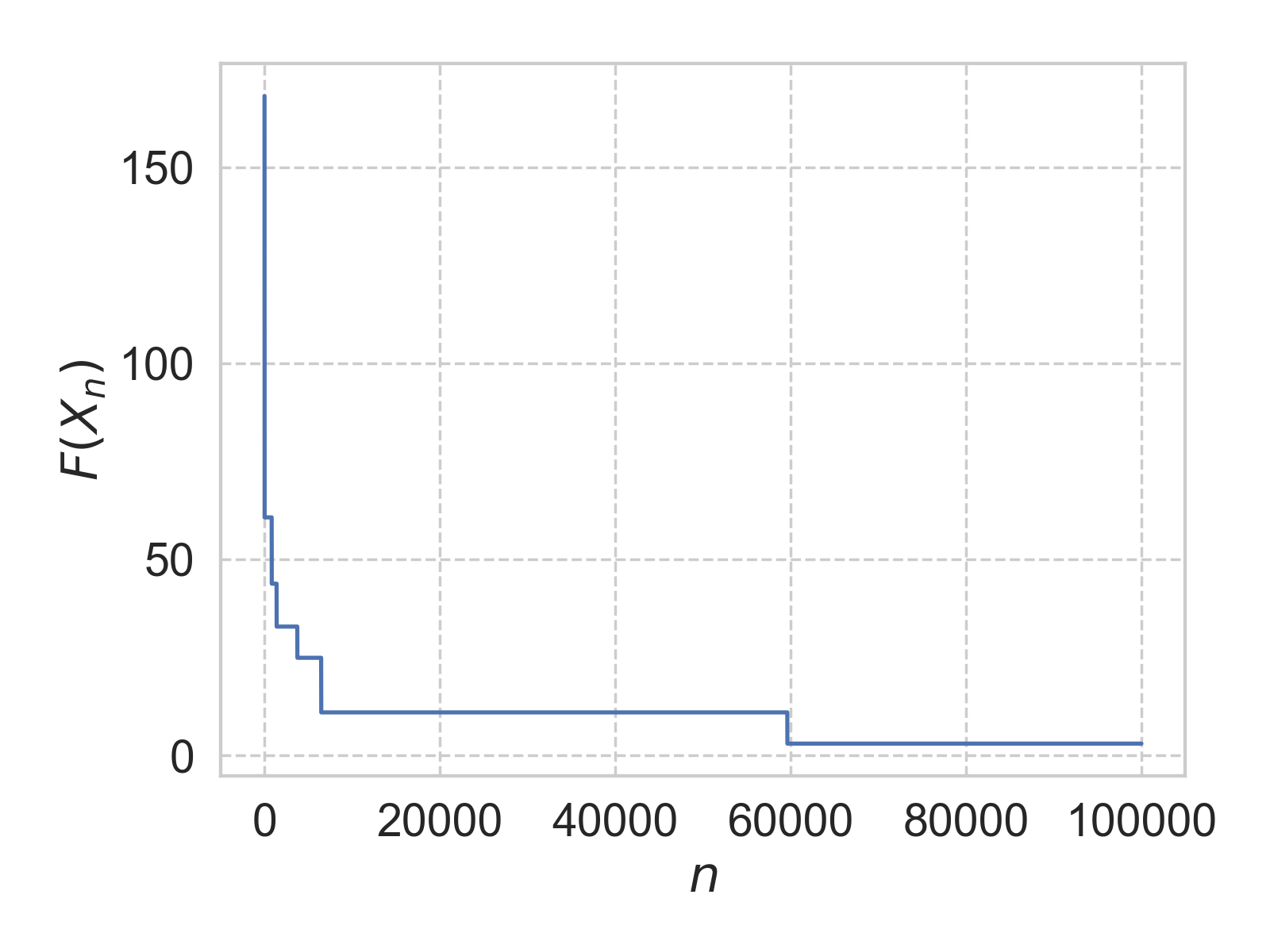}}
	\subfloat[$d=20$]{\includegraphics[width=0.33\textwidth]{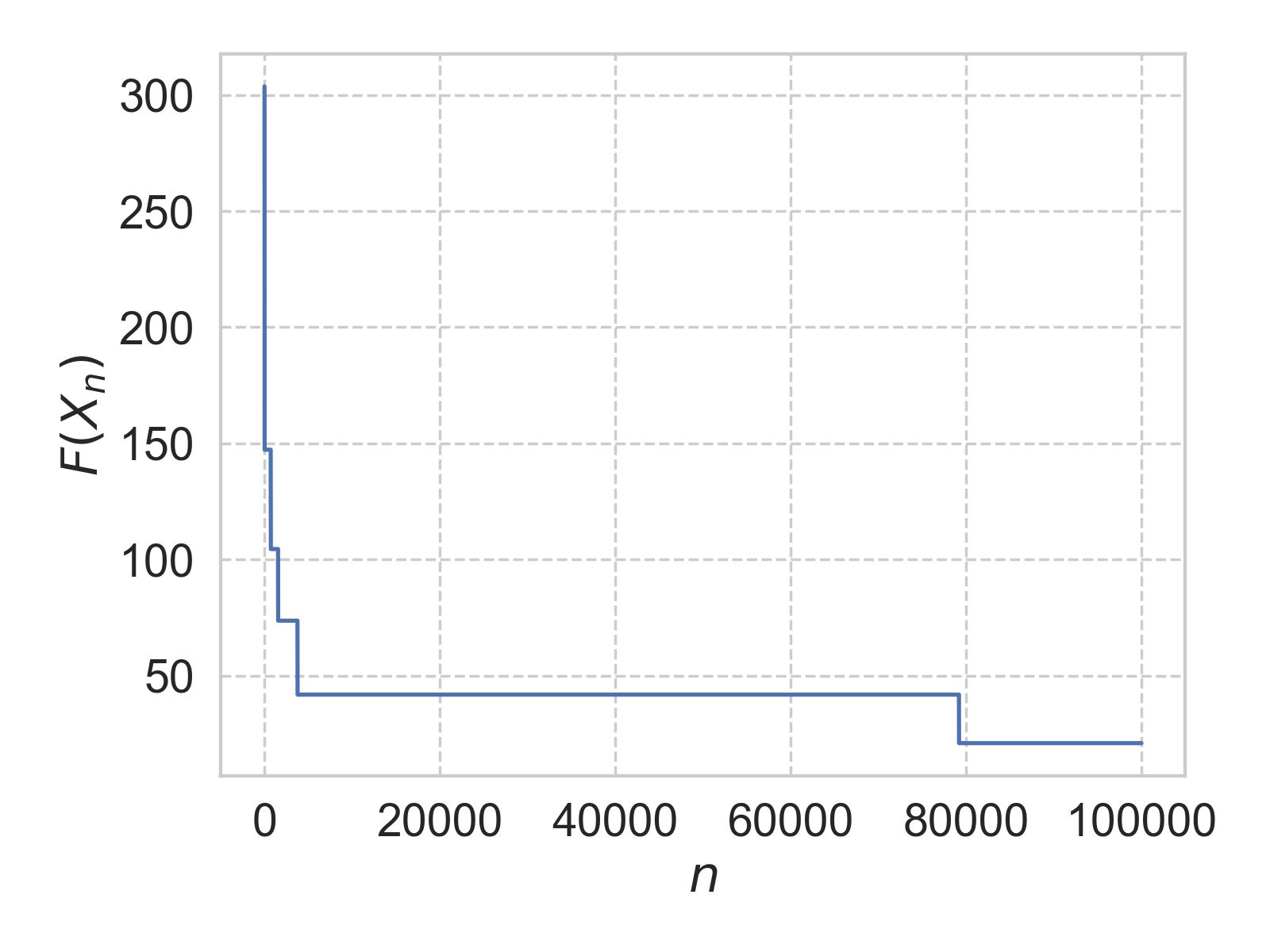}}
	\subfloat[$d=30$]{\includegraphics[width=0.33\textwidth]{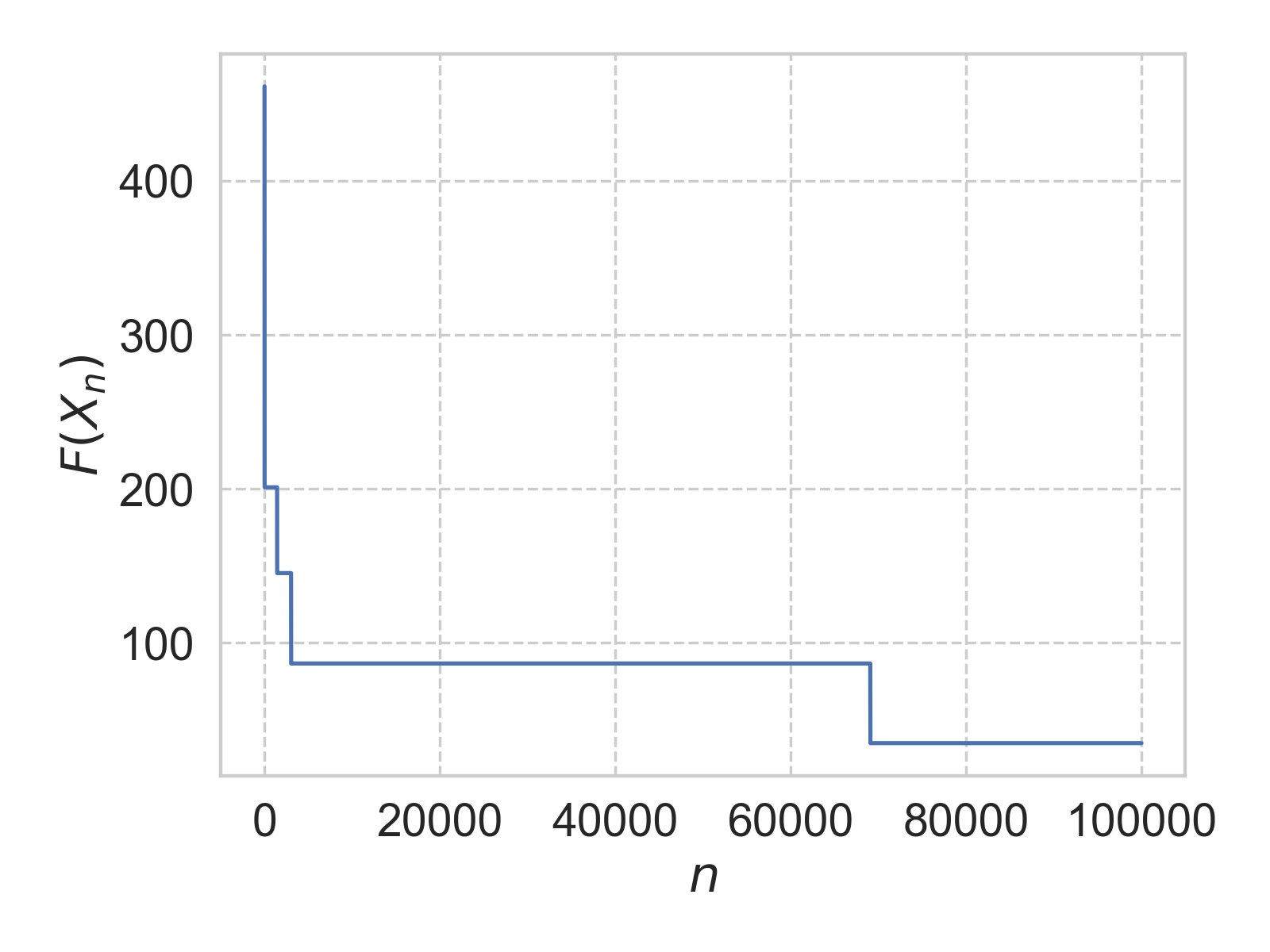}}	
	\caption{GDxLD for Rastrigin functions on $[-5,5]^d$. For (a), (b), and (c), $h=0.005$ and $\gamma=5$.
	For (d), (e), and (f), $h=0.001$ and $\gamma=5$. The initial values are generated uniformly on $[-5,5]^d$}
	\label{fig:rastrigin}
\end{figure}

For Griewank functions of different dimensions, we plot $F(X_n)$ at different iterations under GDxLD in Figure \ref{fig:rastrigin}.
We observe that for $d$ as large as 50, the algorithm is able to find the global minimum within $10^4$ iterations.

\begin{figure}[tbp]
	\centering
	\subfloat[$d=5$]{\includegraphics[width=0.33\textwidth]{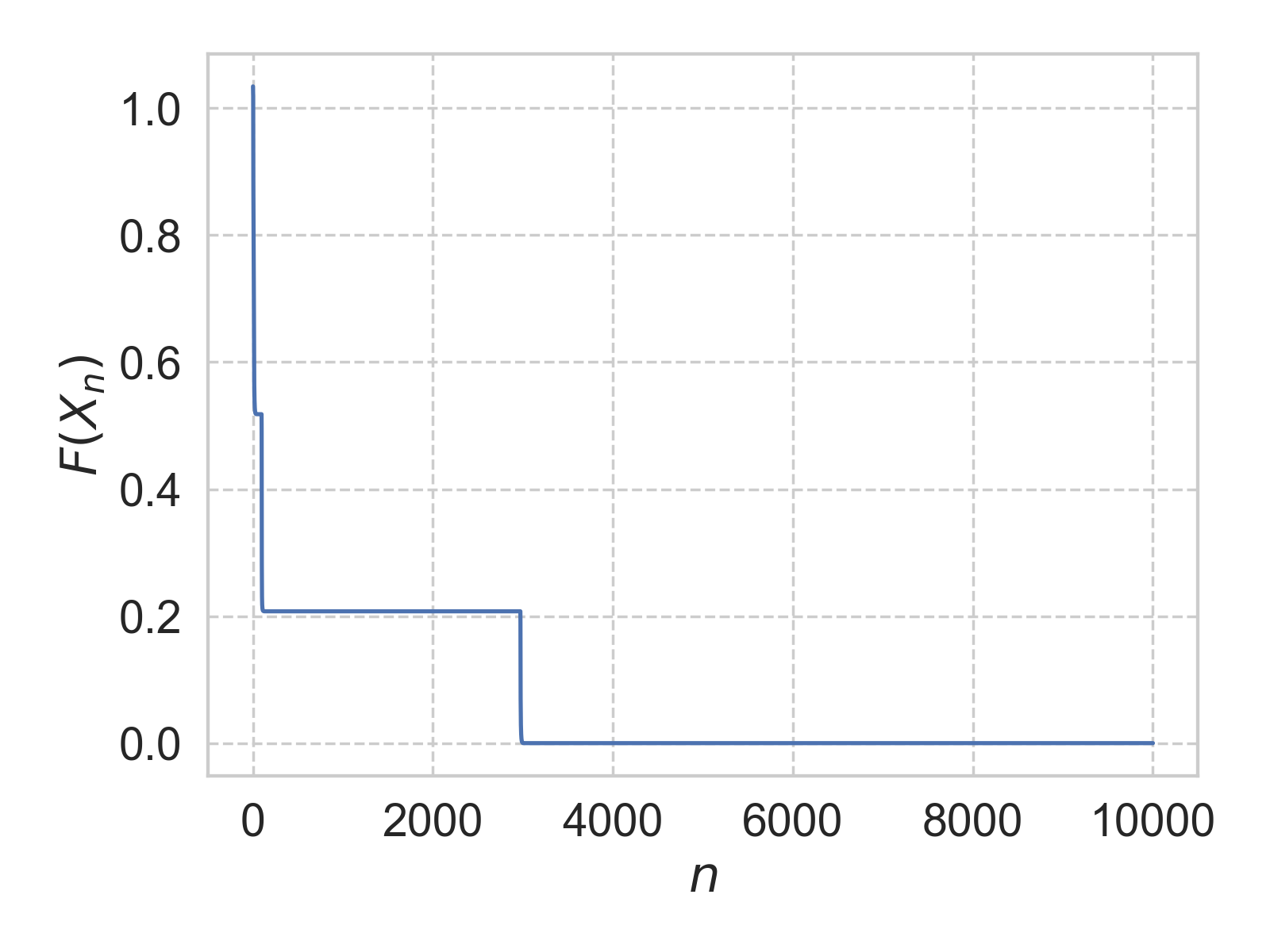}}
	\subfloat[$d=25$]{\includegraphics[width=0.33\textwidth]{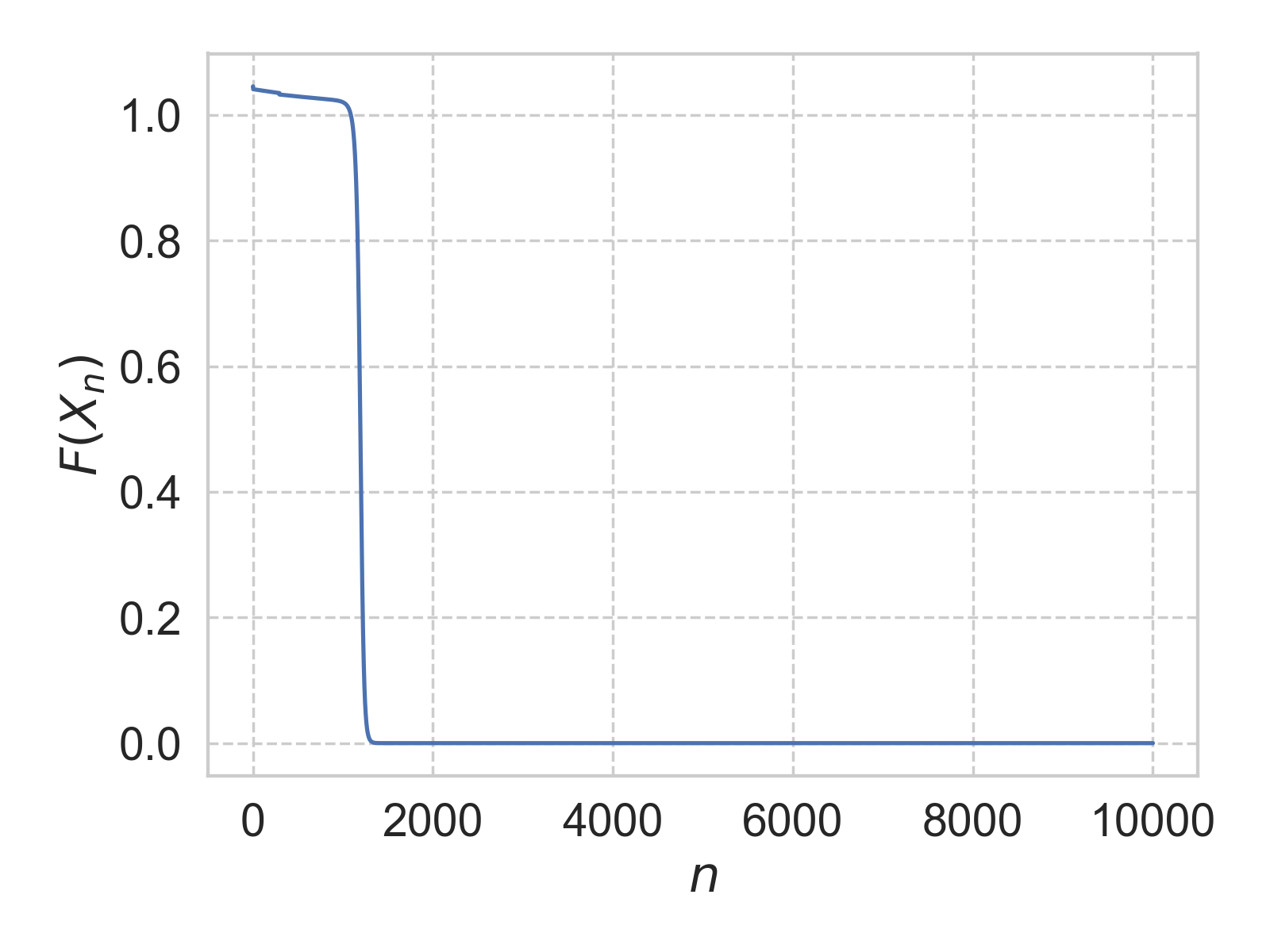}}
	\subfloat[$d=50$]{\includegraphics[width=0.33\textwidth]{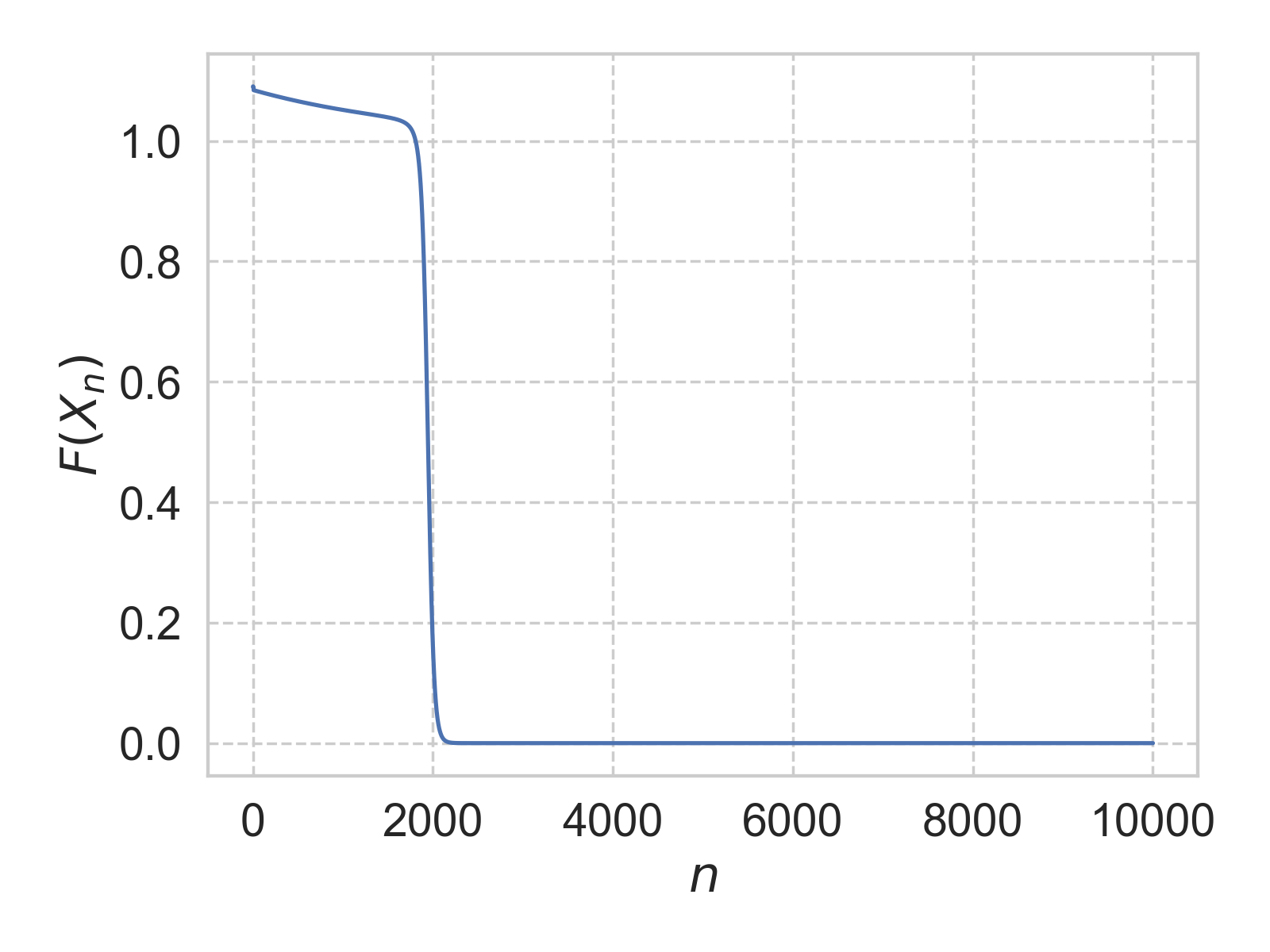}}
	\caption{GDxLD for Griewank functions on $[-5,5]^d$. $h=0.1$ and $\gamma=5$.
	For (d), (e), and (f), $h=0.001$ and $\gamma=5$. The initial values are generated uniformly on $[-5,5]^d$}
	\label{fig:rastrigin}
\end{figure}

\section{Conclusion}
GD is known to converge quickly for convex objective functions, but it is not designed for exploration and it can be trapped at local minima.
LD is better at exploring the state space. But in order for the stationary distribution of the LD to concentrate around the global minimum, it needs to run with a weak stochastic force, which in general slows down its convergence. This paper considers a novel exchange mechanism to exploit the expertise of both GD and LD. The proposed algorithm, (n)GDxLD, can converge to the global minimum linearly with high probability for non-convex objective functions, under the assumption that the objective function is strongly convex in a neighborhood  of the unique global minimum. Our algorithms can be generalized to online settings. To do so, we replace the exact gradient and function evaluation with their corresponding batch-average versions, and introduce an appropriate threshold for exchange. Lastly, we demonstrate the strength of our algorithms through numerical experiments.


\subsection*{Acknowledgement}
We thank anonymous reviewers for their valuable comments and suggestions. The research of Jing Dong is supported by National Science Foundation Grant DMS-1720433.
The research of Xin T. Tong is supported by the Singapore Ministry of Education Academic Research Funds Tier 1 grant R-146-000-292-114.



\appendix
\section{Detailed complexity analysis of (n)GDxLD}
\label{sec:offlineproof}
In this section, we provide the detailed proof of Theorem \ref{thm:strongconvex}.
The proof uses a constructive  stochastic control argument, under which we can ``drive" the iterates into the desired neighborhood. We start by providing an overview of the construction, which can be of interests to the analysis of other sampling-based numerical methods. We first note once $X_n\in B_0$, by convexity, $X_n$ will converge to $x^*$ with properly chosen step size $h$ (see details in Lemma \ref{lem:conv} and \ref{lem:strongconv}). It thus remains to show that $X_n$ will be in $B_0$ with high probability for $n$ large enough. This task involves two key steps.\\

\noindent{\bf Step 1.} We construct a proper exponential-type Lyapunov function $V$  with corresponding parameters $C>0$ and $0<\rho<1$ (Lemma \ref{lem:decay}).
In particular, if  $\|Y_n-x^*\|>C$ and $\|X_n-x^*\|>C$,
\[\E[V(Y_{n+1}^{\prime})|Y_n]\leq \rho V(Y_n) \mbox{ and } \E[V(X_{n+1}^{\prime})|X_n]\leq \rho V(X_n)\]
Utilizing this Lyapunov function, we can show for $Y_n$, the $k$-th, $k\geq 1$, visit time to the set $\{x:\|x-x^*\|\leq C\}$ 
has a finite moment generating function in a neighborhood of the origin (Lemma \ref{lem:stoppingtimes}). This implies that $Y_n$ visits the set 
$\{x:\|x-x^*\|\leq C\}$ quite ``often" (i.e., the inter-visit time has a sub-exponential distribution). \\

\noindent{\bf Step 2.} We then show that during each visit to $\{x:\|x-x^*\|\leq C\}$, there is positive probability that $Y_n$ will also visit $B_0$ (Lemma \ref{lem:smallset}). This essentially creates a sequence of geometric trials whenever $Y_n\in \{x:\|x-x^*\|\leq C\}$. 
Note that once $Y_n\in B_0$, $X_k\in B_0$ for any $k\geq n$ due to the exchange mechanism.\\

\begin{remark}\label{rem:dimension}
The positive probability of visiting $B_0$ in Step 2 can decay exponentially with the dimension $d$. Therefore, the complexity estimates in Theorem \ref{thm:strongconvex} and likewise Theorem \ref{thm:SGDstrongconvex} can grow exponentially in dimension as well. This is not due to the techniques we are using, as the estimates in \cite{raginsky2017non} and \cite{xu2017global} depend on a quantity called ``spectral gap", which can scale exponentially with the dimension as well. 
\end{remark}
To facilitate subsequent discussion, we introduce a few more notations. We will use the filtration
\[
\mathcal{F}_n=\sigma\{X_k, Y_k, Z_{k-1}, k\leq n \},
\]
to denote the information up to iteration $n$.  We use $\Prob_n$ to denote the conditional probability, conditioned on $\mathcal{F}_n$, and $\E_n$ to denote the corresponding conditional expectation.  Note that these notations generalize to stopping times.

To keep our derivation concise, we assume $x^*=0$ and $F(x^*)=0$. This does not sacrifice any generality, since we can always shift the function and consider 
\[
F_c(x)=F(x-x^*)-F(x^*).
\]
It is easy to check if $F$ satisfy the assumptions introduced in Section \ref{sec:main}, $F_c$ also satisfy the assumptions with slightly different constants that depend on $x^*$.

\subsection{Recurrence of the small set} \label{app:small}
In this section, we provide details about Step 1 and 2 in the proof outline. We start from checking  Lemma \ref{lem:coercive}.
\begin{proof}[Proof of Lemma \ref{lem:coercive}]
For Claim 1), note that 
\[
-\langle \nabla F_\lambda(x),x\rangle=-2\lambda \|x\|^2-\langle \nabla F(x), x\rangle.
\]
By Assumption \ref{aspt:Lip}, 
\begin{align*}
|\langle \nabla F(x), x\rangle |&\leq \|\nabla F(x)\|\|x\| \mbox{ by  Cauchy-Schwarz inequality}\\
&\leq (\|\nabla F(x)-\nabla F(0)\|+\|\nabla F(0)\|)\|x\|\\
&\leq (L\|x\|+\|\nabla F(0)\|)\|x\| \mbox{ by Assumption \ref{aspt:Lip}}\\
&= L\|x\|^2+\|\nabla F(0)\|\|x\|. 
\end{align*}
Then, applying Young's inequality, we have
\begin{align*}
-\langle \nabla F_\lambda(x),x\rangle&\leq -2\lambda \|x\|^2+L\|x\|^2+\|\nabla F(0)\|\|x\|\\
&\leq -2\lambda \|x\|^2+2L\|x\|^2+\frac{\|\nabla F(0)\|^2}{L}\\
&\leq -\lambda\|x\|^2+\frac{L^2}{\lambda}+\frac{\|\nabla F(0)\|^2}{L},
\end{align*}
which is of form \eqref{eq:coer}.  

For Claim 2), note that 
\begin{align*}
&-\langle \nabla F(x), x-x^*\rangle\\
\leq& -\lambda_0 \|x\|^2+M_0+\|\nabla F(x)\|\|x^*\| \mbox{ by \eqref{eq:coer}}\\
\leq& -\lambda_0 \|x\|^2+M_0+L\|x-x^*\|\|x^*\| \mbox{ by Assumption \ref{aspt:Lip}}\\
\leq& - \lambda_0 \|x-x^*\|^2+\lambda_0\|x^*\|^2+(2\lambda_0+L)\|x-x^*\|\|x^*\|+M_0\\
\leq& - \frac12\lambda_0 \|x-x^*\|^2+\lambda_0\|x^*\|^2+\frac{2\lambda_0+L}{2\lambda_0}\|x^*\|^2+M_0 \mbox{ by Young's inequality.}
\end{align*}
Setting $M_1=\lambda_0\|x^*\|^2+\frac{2\lambda_0+L}{2\lambda_0}\|x^*\|^2+M_0$, then,
\[
-\langle \nabla F(x), x-x^*\rangle\leq - \frac12\lambda_0 \|x-x^*\|^2+M_1. 
\]
Using Young's inequality again, we have,
\[
\frac{1}{\lambda_0}\|\nabla F(x)\|^2+\frac1{4}\lambda_0\|x^*-x\|^2\geq \langle \nabla F(x), x-x^*\rangle\geq \frac12\lambda_0 \|x-x^*\|^2-M_1.
\]
Thus,
\[
\|\nabla F(x)\|^2\geq \frac{1}{4}\lambda_0^2 \|x-x^*\|^2-\lambda_0 M_1. 
\]

\end{proof}

Our first result provides a proper construction of the Lyapunov function $V$.
It also establishes that $F(X_n)$ is monotonically decreasing.
 
\begin{lemma}
\label{lem:decay}
For (n)GDxLD, under Assumption \ref{aspt:Lip}, if $Lh\leq 1/2$,
\begin{enumerate}
\item The value of $F(X_n)$ keeps decreasing:
\[
F(X_{n+1})\leq F(X'_{n+1})\leq F(X_n)-\frac12\|\nabla F(X_n)\|^2h\leq F(X_n).
\]
\item Assume also Assumption \ref{aspt:coercive}, for $\eta\leq (8\gamma)^{-1}$, $V(x):=\exp (\eta F(x))$ satisfies the following:
\begin{align*}
\E_n[F(Y'_{n+1})]&\leq F(Y_n)-\frac12\|\nabla F(Y_n)\|^2h +2\gamma Lh d\\
&\leq (1-\frac12\lambda_c h) F(Y_n)+2C_Vh. 
\end{align*}
\[
\E_n [V(Y'_{n+1})]\leq \exp\left(-\frac14\eta h \lambda_c F(Y_n)+\eta hC_V\right) V(Y_n),
\]
\[
\E_n [V(X_{n+1})]\leq \E_n [V(X'_{n+1})]\leq \exp\left(-\frac14\eta h \lambda_c F(X_n)+\eta hC_V\right) V(X_n),
\]
where $C_V=M_c/4+4\gamma Ld$.
\end{enumerate}
\end{lemma}
\begin{proof}
For claim 1), note that by Rolle's theorem, there exits $x_n$ on the line segment between $X_n$ and $X'_{n+1}$, such that
\begin{align*}
F(X'_{n+1})=&F(X_n)-\nabla F(x_n)^T\nabla F(X_n) h\\
=&F(X_n)-\nabla F(X_n)^T\nabla F(X_n) h-(\nabla F(x_n)-\nabla F(X_n))^T\nabla F(X_n) h\\
\leq& F(X_n)-\|\nabla F(X_n)\|^2h+h \|\nabla F(x_n)-\nabla F(X_n)\|\|\nabla F(X_n)\| \\
&\mbox{ by Cauchy-Schwarz inequality}\\
\leq& F(X_n)-\|\nabla F(X_n)\|^2h+L h^2 \|\nabla F(X_n)\|^2 \mbox{ by Assumption \ref{aspt:Lip}}\\
\leq& F(X_n)-\frac12\|\nabla F(X_n)\|^2h \mbox{ as $Lh<\frac12$.}
\end{align*}
Claim 1) then follows, as $F(X'_{n+1})\leq F(X_{n+1})$. 

Next, we turn to claim 2). We start by establishing a bound for $F(Y'_{n+1})$. 
Let $\Delta Y_n=Y'_{n+1}-Y_n=-\nabla F(Y_n)h+\sqrt{2\gamma h}Z_n$.
Note that again by Rolle's theorem, there exits $y_n$ on the line segment between $Y_n$ and $Y'_{n+1}$, such that
\begin{align*}
 F(Y'_{n+1})=&F(Y_n)+\nabla F(y_n)^T\Delta Y_n\\
=&F(Y_n)+\nabla F(Y_n)^T\Delta Y_n-(\nabla F(y_n)+\nabla F(Y_n))^T\Delta Y_n\\
\leq&  F(Y_n)+\nabla F(Y_n)^T\Delta Y_n+L\|\Delta Y_n\|^2\\
&\mbox{ by Cauchy-Schwarz inequality and Assumption \ref{aspt:Lip}}\\
\leq& F(Y_n)-\nabla F(Y_n)^T\nabla F(Y_n)h +\sqrt{2\gamma h}\nabla F(Y_n)^TZ_n\\
&+L\nabla F(Y_n)^T\nabla F(Y_n)h^2-2\sqrt{2\gamma h} hL\nabla F(Y_n)^TZ_n+2L\gamma hZ_n^TZ_n\\
\leq& F(Y_n)-\frac12\|\nabla F(Y_n)\|^2h +\beta\sqrt{2\gamma h}\nabla F(Y_n)^TZ_n+2\gamma Lh \|Z_n\|^2 \mbox{ as $Lh<\frac12$},
\end{align*}
with $\beta=1-2Lh\in (0,1)$. Taking conditional expectation and using Assumption \ref{aspt:coercive}, we have our first estimate. 

Recall that $V(y)=\exp(\eta F(y))$. Then, 
%
\[\begin{split}
\E_n[V(Y'_{n+1})]\leq& V(Y_n)\exp\left(- \frac{\eta h}{2}\|\nabla F(Y_n)\|^2\right)\E_n\left[\exp\left(\beta\eta\sqrt{2\gamma h}\nabla F(Y_n)^TZ_n+2\eta\gamma Lh \|Z_n\|^2\right)\right]\\
= &V(Y_n)\exp\left(- \frac{\eta h}{2}\|\nabla F(Y_n)\|^2\right)(1-4\eta\gamma Lh)^{-d/2}\exp\left(\frac{\beta^2\eta^2 \gamma h}{1-4\eta\gamma Lh}\|\nabla F(Y_n)\|^2\right)\\
&\mbox{ as $Z_n\sim \mathcal{N}(0, I_d)$ and $4\eta\gamma Lh<1/4$ } \\
\leq &V(Y_n)\exp\left(- \frac{\eta h}{2}\|\nabla F(Y_n)\|^2\right) \exp\left(\frac{\eta h}{4}\|\nabla F(Y_n)\|^2+8\eta\gamma Lh\frac{d}{2}\right)\\
&\mbox{ as $\eta\gamma <1/8$ and $\beta<1$}\\
\leq & V(Y_n)\exp\left(- \frac{\eta h}{4}\|\nabla F(Y_n)\|^2+4\eta\gamma Lhd\right)\\
\leq & V(Y_n)\exp\left(- \frac{\eta h}{4}\lambda_cF(Y_n)+ \frac{\eta h}{4}M_c+4\eta\gamma Lhd\right) \mbox{ by Assumption \ref{aspt:coercive}}.
\end{split}\]
Similarly, from the derivation of claim 1), we have
\[F(X'_{n+1})\leq F(X_n)-\frac12\|\nabla F(X_n)\|^2h.\]
Then,
\[\begin{split}
\E_n[V(X_{n+1}^{\prime})]&\leq V(X_n)\exp\left(-\frac{\eta h}{2}\|\nabla F(X_n)\|^2\right)\\
&\leq V(X_n)\exp\left(- \frac{\eta h}{4}\lambda_cF(X_n)+ \frac{\eta h}{4}M_c\right).
\end{split}\]
\end{proof}

In the following, we set $R_V=8\lambda_c^{-1}C_V$ and define a sequence of stopping times:
\[
\tau_0=\min\{n\geq 0: F(X_n)\leq R_V\},
\]
and for $k=1,2,\dots$, 
\[
\tau_k=\min\{n>\tau_{k-1}, F(Y_n)\leq R_V\}.
\]

Utilizing the Lyapunov function $V$, our second result establishes bounds for the moment generating function of $\tau_k$'s, $k\geq 0$.
\begin{lemma}
\label{lem:stoppingtimes}
For (n)GDxLD, under Assumptions \ref{aspt:Lip} and \ref{aspt:coercive}, if $Lh\leq 1/2$ and $\eta\leq (8\gamma)^{-1}$,
then for any $K\geq 0$, the stopping time $\tau_K$ satisfies
\[
\E [\exp(h\eta C_V\tau_K)]\leq \exp(2Kh\eta C_V+K\eta R_V)(V(X_0)+V(Y_0)). 
\]
\end{lemma}
\begin{proof}
Note that for $n<\tau_0$, $F(Y_n)\geq F(X_n)>R_V=8\lambda_c^{-1}C_V$. Then, by Lemma \ref{lem:decay},
\[
\E_n[V(X_{n+1})+V(Y_{n+1})]\leq \exp(-h\eta C_V  )(V(X_{n})+V(Y_{n})).
\]
This implies 
\[
\left(V(X_{\tau_0\wedge n})+V(Y_{\tau_0\wedge n})\right)\exp(h\eta C_V   ( \tau_0 \wedge n))
\]
is a supermartingale. 
As $V(x)\geq 1$, we have, by sending $n\rightarrow\infty$,
\[
\E[\exp(h\eta C_V \tau_0)]\leq V(X_0)+V(Y_0). 
\]

By Lemma \ref{lem:decay}, $F(X_{n+1})\leq F(X_{n+1}^{\prime})<R_V$ for $n\geq \tau_0$.
Therefore, for $k\geq 0$, if $\tau_{k+1}>\tau_k+1$, $F(Y_n')>R_V$ and there is no jump for $X_n$ at step $n$ for $\tau_k<n<\tau_{k+1}$.

Given $\mathcal{F}_{\tau_k}$ (starting from $\tau_k$), for $n\leq \tau_{k+1}-1$, we have $F(Y_n)>R_V$, and by Lemma \ref{lem:decay},
\[\E_n[V(Y_{n+1})]\leq\exp(-h\eta C_V)V(Y_n).\]
This implies $V(Y_{\tau_{k+1}\wedge n})\exp(h\eta C_V(\tau_{k+1}\wedge n))$ is a supermartingale.
Then, because $V(x)\geq 1$, by sending $n\rightarrow\infty$, we have
\[\E_{\tau_k+1}[\exp(h\eta C_V(\tau_{k+1}-\tau_k-1))]\leq V(Y_{\tau_{k}+1}).\]
Next,
\[\begin{split}
\E_{\tau_k}[\exp(h\eta C_V(\tau_{k+1}-\tau_k-1))]
\leq& \E_{\tau_k}[V(Y_{\tau_{k}+1})]\\
\leq& \exp\left(-\frac{1}{4}h\eta F(Y_{\tau_k})+h\eta C_V\right)V(Y_{\tau_k})  \mbox{ by Lemma \ref{lem:decay}}\\
\leq& \exp(h\eta C_V + \eta R_V).
\end{split}\]
Now because
\[\E[\exp(h\eta C_V\tau_0)]\leq V(X_0)+V(Y_0),\]
we have
\[\begin{split}\E[\exp(h\eta C_V(\tau_K-K))]&=\E\left[\prod_{k=0}^{K-1}\exp(h\eta C_V(\tau_{k+1}-\tau_k-1))\right]\\
&\leq\left(\prod_{k=0}^{K-1}\exp(h\eta C_V + \eta R_V)\right)(V(X_0)+V(Y_0)).
\end{split}\]
Then,
\[\E[\exp(h\eta C_V\tau_K)]=\exp(2h\eta K C_V+ K\eta R_V)(V(X_0)+V(Y_0)).\]
\end{proof}

Let 
\[D=\max\{\|x-h\nabla F(x)\|: F(x)\leq R_V\}.\]
The last result of this subsection shows that if $F(Y_n)\leq R_V$, there is a positive probability that
$\|Y'_{n+1}\|\leq r$ for any $r>0$. In particular, this includes $r=r_l$.

\begin{lemma}
\label{lem:smallset}
For (n)GDxLD, if $F(Y_n)\leq R_V$, for any $r>0$, there exist an $\alpha(r,D)>0$, such that
\[
\Prob_n(\|Y'_{n+1}\|\leq r)>\alpha(r,D).
\]
In particular, a lower bound for $\alpha(r,D)$ is given by 
\[\alpha(r,D)\geq \frac{S_d r^d}{(4\gamma h\pi)^{\frac{d}{2}}}\exp\left(-\frac1{2\gamma h}(D^2+ r^2)\right)>0,\]
where $S_d$ is the volume of a $d$-dimensional unit-ball.
\end{lemma}
\begin{proof}
\begin{align*}
\Prob_n (\|Y'_{n+1}\|\leq r)&=\Prob_n(\|Y_n-h\nabla F(Y_n)+\sqrt{2\gamma h}Z_n\|\leq r)\\
&=\Prob_n \left(\|Z_{n}-Q_{n}\|\leq \frac{r}{\sqrt{2\gamma h}}\right),
\end{align*}
where
$Q_{n}:=-(Y_n-h\nabla F(Y_n))/\sqrt{2\gamma h}$.
Note that as $F(Y_n)\leq R_V$,
$\|Q_n\|\leq D/\sqrt{2\gamma h}$. Thus,
\begin{align*}
\Prob_n \left(\|Z_{n}-Q_{n}\|\leq \frac{r}{\sqrt{2\gamma h}}\right)&=\int_{\|z\|\leq r/\sqrt{2\gamma h}} \frac{1}{(2\pi)^{\frac{d}{2}}}\exp\left(-\frac12\|z+Q_n\|^2\right)dz\\
&\geq \int_{\|z\|\leq r/\sqrt{2\gamma h}} \frac{1}{(2\pi)^{\frac{d}{2}}}\exp\left(-\frac1{2\gamma h}(D^2+r^2)\right)dz\\
&\geq  \frac{S_d r^d}{(4\gamma h\pi)^{\frac{d}{2}}}\exp\left(-\frac1{2\gamma h}(D^2+r^2)\right).
\end{align*}
\end{proof}

\subsection{Convergence to global minimum}
In this subsection, we analyze the ``speed" of convergence for $\{X_{n+k}: k\geq 0\}$ to $x^*=0$ when $X_n\in B_0$.
Most of these results are classical. In particular, if we assume $B_0=R^d$, then these rate-of-convergence results can be found in \cite{Nesterov:2005}. For self-completeness, we provide the detailed arguments here adapted to our settings. First of all, we show Assumption \ref{aspt:convstrong} leads to Assumption \ref{aspt:conv}.

\begin{lemma}
\label{lem:3t4}
Under Assumption \ref{aspt:convstrong}, Assumption \ref{aspt:conv} holds with 
\[
r_u=r,\quad r_0=\frac12a,\quad r_l=\sqrt{\frac{a}{M}}. 
\]
Moreover for any $x,y\in B_0$, 
\begin{equation}
\label{eqn:strongconv}
F(y)-F(x)\geq \langle \nabla F(x), y-x\rangle+\frac{m}{2} \|y-x\|^2. 
\end{equation}
\end{lemma}
\begin{proof}
Without loss of generality, we assume $x^*=0$ and $F(0)=0$.
Let $a=\min_{x:\|x\|\geq r} F(x)$. Then $a> 0$ by our assumption. We choose $r_0=\frac12 a, r_u=r$, then   $B_0=\{x: F(x)\leq F(x^*)+r_0\}\subset \{\|x\|\leq r\}$. 

Next by Taylor expansion, we know $F$ is convex within $B_0$ and $\{|x|\leq r\}$. This also leads to $B_0$ being convex since it is a sublevel set. Moreover for any $x$ so that $\|x\|\leq r$, for some $x'$ on the line between $x$ and $0$,
\[
F(x)=F(0)+x^T\nabla F(0)+\frac12 x^T\nabla^2 F(x')x=\frac12 x^T\nabla^2 F(x')x.
\]
So $F(x)\leq \frac12 M\|x\|^2$. So if we let $r_l=\sqrt{\frac{a}{M}}$,
$\{\|x-x^*\|\leq r_l\}\subset B_0$. Finally, using Taylor's expansion leads to \eqref{eqn:strongconv}.
\end{proof}

\begin{lemma}
\label{lem:conv}
For (n)GDxLD, under Assumptions \ref{aspt:Lip} and \ref{aspt:conv}, and assuming $h\leq \min\{1/(2L), r_u^2/r_0\}$, if $X_n\in B_0$, 
 \[
 F(X_{n+k})\leq \frac{1}{1/r_0+ kh/r_u^2}
 \] 
 for all $k\geq 0$. 
 \end{lemma}
\begin{proof}
From Lemma \ref{lem:decay}, we have, if $F(X_n)\leq r_0$, i.e, $X_n\in B_0$, $F(X_{n+1})\leq r_0$.

We first note for any $k\geq 0$,
\begin{equation}\label{eq:conv_bound}
F(X_{n+k})\leq \langle\nabla F(X_{n+k}), X_{n+k}\rangle \leq r_u \|\nabla F(X_{n+k})\|,
\end{equation}
where the first inequality follows by convexity (Assumption \ref{aspt:conv}) and the second inequality follows by H\"older's inequality.

Next, by Lemma \ref{lem:decay},
\[ 
F(X_{n+k}) \leq F(X_{n+k-1})-\frac{1}{2}\|\nabla F(X_{n+k-1})\|^2h
\leq F(X_{n+k-1}) -\frac{h}{2r_u^2}F(X_{n+k-1})^2,
\]
where the last inequality follows from \eqref{eq:conv_bound}.
This implies 
\[
\frac{1}{F(X_{n+k})}\geq \frac{1}{F(X_{n+k-1}) -\frac{h}{2r_u^2}F(X_{n+k-1})^2}
=\frac{1}{F(X_{n+k-1})}+\frac{1}{2r_u^2/h - F(X_{n+k-1})}.
\]
Because $F(X_{n+k})\leq r_0 < r_u^2/h$ by assumption, 
\[
\frac{1}{F(X_{n+k})}\geq \frac{1}{F(X_{n+k-1})}+\frac{h}{r_u^2}.
\]
Then, by induction, we have
\[
\frac{1}{F(X_{n+k})}\geq \frac{1}{F(X_{n})}+\frac{kh}{r_u^2} \geq \frac{1}{r_0}+\frac{kh}{r_u^2}.
\]
\end{proof}

\begin{lemma}
\label{lem:strongconv}
For (n)GDxLD, under Assumptions \ref{aspt:Lip} and \ref{aspt:convstrong}, if $F(X_n)\leq r_0$, $Lh\leq 1/2$,
\[
F(X_{n+k})\leq (1-m h)^kF(X_n)
\] 
for all $k\geq 0$. 
\end{lemma}
\begin{proof}
We first note if $F(x)$ is strongly convex in $B_0$, for $x\in B_0$,
\[F(x^*)-F(x)\geq \langle\nabla F(x), x^*-x\rangle + \frac{m}{2} ||x-x^*||^2.\]
By rearranging the inequality, we have
\begin{equation} \label{eq:strong_conv_bound}
F(x)-F(x^*)\leq \langle\nabla F(x), x-x^*\rangle - \frac{m}{2} ||x-x^*||^2
\leq \frac{1}{2m} \|\nabla F(x)\|^2,
\end{equation}
where the last inequality follows from Young's inequality.

Next, from Lemma \ref{lem:decay}, we have
\[
F(X_{n+1})\leq F(X_{n})-\frac12 \|\nabla F(X_n)\|^2h \leq (1-mh)F(X_{n}),
\]
where the second inequality follows from \eqref{eq:strong_conv_bound}.
Note that by Lemma \ref{lem:decay}, $F(X_{n+k})\leq r_0$ for $k\geq 0$.
Thus, by induction, we have
\[
F(X_{n+k})\leq (1-mh)^k F(X_{n}).
\]
\end{proof}

\begin{remark}
\label{rem:convex}
The proof of Lemma \ref{lem:strongconv} deals with $F(X_n)$ and $\nabla F(X_n)$ directly. It is thus easily generalizable to the online setting as the noise is additive (see Lemma \ref{lem:strongconvSGD}).  In contrast, the proof for Lemma \ref{lem:conv} requires investigating $(F(X_n))^{-1}$. Its generalization to the online setting can be much more complicated, as the stochastic noise can make the inverse singular. 
\end{remark}

\subsection{Proof of Theorem \ref{thm:strongconvex}}
We are now ready to prove the main theorem.

Note from Lemma \ref{lem:conv} that if $X_n \in B_0$,  
for any 
\[
k\geq\left(\frac{1}{\epsilon}-\frac{1}{r_0}\right)\frac{r_u^2}{h},
\]
$F(X_{n+k})\leq \epsilon$. We set 
\[
k(\epsilon)=(1/\epsilon-1/r_0)r_u^2/h = O(\epsilon^{-1}).
\]

Next, we study ``how long" it takes for $X_n$ to reach the set $B_0$.
From Lemma \ref{lem:smallset}, every time $Y_n\in \{x: F(x)\leq R_V\}$,
\[
\Prob_n(\|X_{n+1}\|\leq r_l)\geq \Prob_n(\|Y_{n+1}^{\prime}\|\leq r_l)\geq \alpha(r_l, D)>0.
\]
Then,
\[
\Prob(F(X_{\tau_K+1})>r_0)=\E\left[\prod_{k=1}^{l}\Prob_{\tau_k}(F(X_{\tau_k+1})>r_0)\right]\leq (1-\alpha(r_l,D))^K.
\]
Thus, if 
\[
K>\frac{\log(\delta/2)}{\log(1-\alpha(r_l,D))},
\]
$\Prob(F(X_{\tau_K+1})>r_0)<\delta/2$. We set 
\[
K(\delta)=\log(\delta/2)/\log(1-\alpha(r_l,D))=O(-\log \delta).
\]

Lastly, we establish a bound for $\tau_K$.
From Lemma \ref{lem:stoppingtimes}, we have, by Markov inequality,
\[
\Prob(\tau_K>T)\leq \frac{\E[\exp(h\eta C_V \tau_K)]}{\exp(h\eta C_V T)}
\leq \frac{\exp(2Kh\eta C_V + K\eta R_V)(V(X_0)+V(Y_0))}{\exp(h\eta C_V T)}.
\]
Thus, if
\[
T>-\frac{\log(\delta/2)}{h\eta C_V}+2K+\frac{KR_V}{hC_V}+\frac{\log(V(X_0)+V(Y_0))}{h\eta C_V}
\]
$\Prob(\tau_K>T)<\delta/2$. We set
\[
T(\delta)=-\frac{\log (\delta/2)}{h\eta C_V}+2K(\delta)+\frac{K(\delta)R_V}{hC_V}+\frac{\log(V(X_0)+V(Y_0))}{h\eta C_V}
=O(-\log \delta).
\]

Above all, we have, for any $N\geq T(\delta)+k(\epsilon)$,
\[\begin{split}
\Prob(F(X_N)>\epsilon)&=\Prob(F(X_N)>\epsilon, \tau_K>T(\delta))+\Prob(F(X_N)>\epsilon, \tau_K<T(\delta))\\
&\leq \Prob(\tau_K>T(\delta)) + \Prob(F(X_{\tau_K+1})>r_0)\leq \delta.
\end{split}\]

When Assumption \ref{aspt:convstrong} holds, $F$ is strongly convex in $B_0$, from Lemma \ref{lem:strongconv}, if $X_n \in B_0$, 
for any 
\[
k\geq \frac{\log(\epsilon)-\log(r_0)}{\log(1-mh)},
\]
$F(X_{n+k})\leq \epsilon$. In this case, we can set
\begin{equation}\label{eq:GD_convergence}
k(\epsilon)=\left(\log(\epsilon)-\log(r_0)\right)/\log(1-mh) = O(-\log \epsilon).
\end{equation}

Note that $T$ scales with $K$ and $1/\alpha(r_l,D)$. Lemma \ref{lem:smallset} shows that $1/\alpha(r_l,D)$ depends exponentially on $d$ and $\frac{1}{\gamma h}$.  

\subsection{Proof of Theorem \ref{thm:LSI}} \label{app:thm3}
The proof of Theorem \ref{thm:LSI} follows similar lines of arguments as the proof of Theorem \ref{thm:strongconvex}.
In particular, it relies on a geometric trial argument. The key difference is that the success probability of the geometric trial is bounded using the pre-constant of LSI  rather than the small set argument in Lemma \ref{lem:smallset}.

Let 
\[\begin{split}
\tilde Y_{n+1} &= \tilde Y_n - \nabla F(\tilde Y_n)h +\sqrt{2\gamma h}Z_n\\
&=\tilde Y_n - \frac{1}{\gamma} \nabla F(\tilde Y_n) \gamma h + \sqrt{2\gamma h}Z_n
\end{split}\]
This is to be differentiated with $Y_n$ in Algorithm \ref{alg:GDxLD}, which can swap position with $X_n$. Note that $\tilde Y_n$ is also known as the unadjusted Langevin algorithm (ULA). We denote $\tilde \mu_n$ as the distribution of $\tilde Y_{n}$. Recall that $\pi_{\gamma}(x)=\tfrac{1}{U_{\gamma}} \exp(-\tfrac1\gamma F(x))$. We first prove an auxiliary bound for $V(x)=\exp (\eta F(x))$ that will be useful in our subsequent analysis.

\begin{lemma}
\label{lem:exponential}
Under Assumptions \ref{aspt:Lip} and \ref{aspt:coercive}, for $\eta\leq (8\gamma)^{-1}$, $h\leq 1/(2L)$,  
\[
\E_n [V(\tilde Y_{n+m})]\leq V(\tilde Y_n)+4\exp(\eta R_V),
\]
where $R_V=8C_V/\lambda_c$.
\end{lemma}
\begin{proof}
From the second claim in Lemma \ref{lem:decay}, we have
\[
\E_n[V(\tilde Y_{n+1})]\leq \exp\left(-\frac14\eta h \lambda_c F(\tilde Y_n)+\eta hC_V\right) \exp(\eta F(\tilde Y_n)).
\]
When $F(\tilde Y_n)< R_V$, 
\begin{align*}
\E_n[V(\tilde Y_{n+1})]&\leq \exp\left(\eta hC_V\right) \exp(\eta F(\tilde Y_n))\\
&\leq \left(1-\frac12\eta hC_V\right) \exp(\eta F(\tilde Y_n))+\left(\exp\left(\eta hC_V\right)-1+\frac{1}{2}\eta hC_V\right)\exp(\eta R_V)\\
&\leq \left(1-\frac12\eta hC_V\right) \exp(\eta F(\tilde Y_n))+2\eta hC_V
\exp(\eta R_V)\\
\end{align*}
When $F(\tilde Y_n)\geq R_V$, we have 
\[
\E_n[V(\tilde Y_{n+1})]\leq \exp\left(-\eta hC_V\right) \exp(\eta F(\tilde Y_n))\leq \left(1-\frac12 \eta hC_V\right)\exp(\eta F(\tilde Y_n))
\]
In both cases, 
\[
\E_n[V(\tilde Y_{n+1})]\leq \left(1-\frac12\eta hC_V\right)V(\tilde Y_n)+2\eta hC_V\exp(\eta R_V),
\]
which implies that
\[
\E_n[V(\tilde Y_{n+m})]\leq V(\tilde Y_n)+4\exp(\eta R_V).
\]
\end{proof}

\begin{lemma} \label{lm:LSI}
Under Assumptions \ref{aspt:Lip} and \ref{aspt:LSI},  
\[
\KL(\tilde\mu_n\|\pi_{\gamma})\leq  \exp(-\beta  (n-1) h\gamma) \left(\frac1\gamma F(\tilde Y_0)+\tilde C_d\right)+\frac{8 hdL^2}{\beta\gamma}
\]
where $\tilde M_d:=2Lh d-\tfrac{d}{2}(\log (4\pi h)+1)+\log (U_\gamma)$.
\end{lemma}
\begin{proof}
Given $\tilde Y_0$, 
\[
\tilde Y_1\sim \mathcal{N}(m(\tilde Y_0),2\gamma h I_d ),\mbox{ where } m(\tilde Y_0)=\tilde Y_0-\nabla F(\tilde Y_0)h.
\]
Then, given $\tilde Y_0$,
\[\begin{split}
\KL(\tilde \mu_1\|\pi_{\gamma})&=\int \log\left(\frac{\tilde\mu_1(y)}{\pi_{\gamma}(y)}\right)\tilde \mu_1(y)dy\\
&=\int \left(-\frac{d}{2}\log(4\pi\gamma h) - \frac{\|y-m(\tilde Y_0)\|^2}{4\gamma h} + \log U_{\gamma} + \frac{F(y)}{\gamma}\right)\tilde\mu_1(y)dy\\
&=-\frac{d}{2}\log(4\pi\gamma h)-\frac{d}{2} + \log U_{\gamma} + \frac{1}{\gamma}\E[F(\tilde Y_1)|\tilde Y_0]\\
&\leq -\frac{d}{2}\log(4\pi\gamma h)-\frac{d}{2} + \log U_{\gamma} + \frac{1}{\gamma}(F(\tilde Y_0)+2\gamma Lhd) \mbox{ by Lemma \ref{lem:decay} claim 2.}
\end{split}\] 
For 
\[\tilde M_d=-\frac{d}{2}(\log(4\pi\gamma h)+1)+ \log U_{\gamma} + 2Lhd,\]
we have
\[\KL(\tilde \mu_1\|\pi_{\gamma}) \leq \frac{1}{\gamma} F(\tilde Y_0) + \tilde M_d  \]
Next, from Theorem 2 in \cite{vempala2019rapid} with $L=L/\gamma$ and $\epsilon=\gamma h$, we have  
\[\begin{split}
\KL(\tilde \mu_n\|\pi_{\gamma})&\leq \exp(-\beta   (n-1) h\gamma) \KL(\tilde\mu_1\|\pi_{\gamma})+\frac{8 hdL^2}{\beta \gamma}\\
&\leq \exp(-\beta  (n-1) h\gamma) \left(\frac1\gamma F(\tilde Y_0)+\tilde M_d\right)+\frac{8  hdL^2}{\beta\gamma}. 
\end{split}\]
\end{proof}

Based on Lemma \ref{lm:LSI}, let
\[n_0=\frac{2}{\beta\gamma} h^{-1}\log(1/h)+1.\]
Note that for $n\geq n_0$ and $h$ small enough,
\begin{equation}\label{eq:LSI_bound}
\begin{split}
\KL(\tilde\mu_n\|\pi_{\gamma})&\leq  h^2\left(\frac1\gamma F(\tilde Y_0)+\tilde M_d\right)+\frac{8 hdL^2}{\beta\gamma}\\
&\leq h F(\tilde Y_0) + \frac{9 hdL^2}{\beta \gamma}.
\end{split}
\end{equation}

We next draw connection between the bound \eqref{eq:LSI_bound} and the hitting time of $Y_n$ to $B_0$.
For nGDxLD, let 
\[\phi=\min_n \{Y_n\in B_0 \}.\] 
With a slight abuse of notation, for GDxLD, let 
\[\phi=\min_n \{X_{n}=Y'_n \text{ or } Y_n\in B_0 \}.\]

\begin{lemma} \label{lm:LSI_hit}
For nGDxLD or GDxLD, under Assumptions \ref{aspt:Lip}, \ref{aspt:coercive}, and \ref{aspt:LSI}, 
\[
\Prob(\phi\leq n_0 )\geq \pi_\gamma(B_0) - \sqrt{h}F(\tilde Y_0) - 3 \frac{\sqrt{hd} L}{\sqrt{\beta\gamma}}.
\]
\end{lemma}
\begin{proof}
For nGDxLD, $Y_n=\tilde Y_n$. 
For GDxLD,  before replica exchange, we also have $Y_n=\tilde Y_n$. Therefore,
$\Prob(\phi > n)\leq   \Prob(Z_n\notin B_0)$,
which further implies that 
\[\Prob(\phi\leq n) \geq \Prob(Z_n\in B_0).\]

By Pinsker's inequality, the KL divergence provides an upper bound for the total variation distance. Let $d_{tv}(\mu, \pi)$ denote the total variation distance between $\mu$ and $\pi$. Then, for $n\geq n_0$,
\[d_{tv}(\tilde \mu_n, \pi_{\gamma}) \leq \sqrt{\frac{1}{2} \KL(\tilde \mu_n\|\pi_{\gamma})}
\leq \sqrt{\frac{1}{2}hF(\tilde Y_0) + \frac{1}{2}\frac{ 9hdL^2}{\beta\gamma}}
\leq \sqrt{h}F(\tilde Y_0) + 3 \frac{\sqrt{hd} L}{\sqrt{\beta\gamma}},\]
where the second inequality follows from \eqref{eq:LSI_bound}.
Thus, 
\[
\Prob(\tilde Y_{n_0}\in B_0) \geq \pi_{\gamma}(B_0) - d_{tv}(\tilde \mu_n, \pi_{\gamma})
\geq \pi_{\gamma}(B_0) - \sqrt{h}F(\tilde Y_0) - 3 \frac{\sqrt{hd} L}{\sqrt{\beta\gamma}}.
\]
\end{proof}

\begin{lemma} \label{lm:LSI_hit2}
For (n)GDxLD, fix any $K$ and $T\geq (n_0+1)K$,
\[\begin{split}
\Prob_0(\phi> T)\leq&\exp(2(n_0+1)Kh\eta C_V+(n_0+1)K\eta R_V-h\eta C_V T)(V(X_0)+V(Y_0))\\
&+\left(1-\pi_{\gamma}(B_0) + \sqrt{h}R_V + 3 \frac{\sqrt{ hd} L}{\sqrt{\beta\gamma}}\right)^K.
\end{split}\]
\end{lemma}
\begin{proof}
Recall the sequence of stopping times
$\tau_j=\inf\{n>\tau_{j-1}: F(Y_n)\leq R_V\}$.
From Lemma \ref{lem:stoppingtimes}, we have
\[
\E [\exp(h\eta C_V\tau_k)]\leq \exp(2kh\eta C_V+k\eta R_V)(V(X_0)+V(Y_0)). 
\]
We next define a new sequence of stopping times:
\[
\psi_0=\inf\left\{n: F(Y_n)\leq R_V\right\}, ~~~\psi'_0=\psi_0+n_0,
\]
and for $k=1, \dots$,
\[
\psi_k=\inf\left\{n\geq \psi'_{k-1}+1, F(Y_n)\leq R_V\right\}, ~~~
\psi'_k=\psi_k+n_0+1.
\]
Note that $\psi_i$ always coincide with one of $\tau_j$'s, and as $\tau_{j}-\tau_{j-1}\geq 1$, 
\[\psi_{k}\leq \tau_{(n_0+1)k}.\]
Thus,
\begin{align*}
\Prob_0(\phi\geq T)\leq & \Prob_0\left(\tau_{(n_0+1)K}\geq T\right)+\Prob_0\left(Y_{\psi'_k}\notin B_0,\forall k\leq K\right)\\
\leq& \exp(2(n_0+1)Kh\eta C_V+(n_0+1)K\eta R_V-h\eta C_V T)(V(X_0)+V(Y_0))\\
&+\left(1-\pi_{\gamma}(B_0) + \sqrt{h}R_V + 3 \frac{\sqrt{hd} L}{\sqrt{\beta\gamma}}\right)^K \mbox{ by Lemma \ref{lm:LSI_hit}.}
\end{align*}
\end{proof}

\noindent{\bf We are now ready to prove Theorem \ref{thm:LSI}.}
Set
\[h<\left(\frac{\pi_{\gamma}(B_0)}{2R_V+6\sqrt{ d}L/\sqrt{\beta\gamma}}\right)^2.\] 
Then,
\[1-\pi_{\gamma}(B_0) + \sqrt{h}R_V + 3 \frac{\sqrt{ hd} L}{\sqrt{\beta\gamma}}<1-\frac{\pi_{\gamma}(B_0)}{2}.\]

{\bf For nGDxLD}, based on Lemma \ref{lm:LSI_hit2}, let
\[K(\delta)=\frac{\log(\delta/2)}{\log(1-\pi_{\gamma}(B_0)/2)}.\]
For $K\geq K(\delta)$, 
\[\left(1-\pi_{\gamma}(B_0) + \sqrt{h}R_V + 3 \frac{\sqrt{ hd} L}{\sqrt{\gamma\beta}}\right)^K\leq \frac{\delta}{2}.\]
In addition, let
\[
T(\beta, \delta)=-\frac{\log (\delta/2)}{h\eta C_V}+2(n_0+1)K(\delta)+\frac{(n_0+1)K(\delta)R_V}{hC_V}+\frac{\log(V(X_0)+V(Y_0))}{h\eta C_V}.
\]
Note that $T(\beta,\delta)=O(K(\delta)n_0)=O(\log(1/\delta)/\beta)$.  
For $T\geq T(\delta)$, $\Prob(\tau_K>T)\leq \delta/2$. 
Above all,
$\Prob(\phi>T(\delta))\leq \delta$.

Lastly, note that once $Y_n$ is in $B_0$, $X_n$ will be moved there if it is not in $B_0$ already.
After $X_n$ is in $B_0$, it takes at most $k(\epsilon)=\left(\log(\epsilon)-\log(r_0)\right)/\log(1-mh)$ iterates 
to achieve the desired accuracy.
Therefore, we can set 
\[N(\beta,\epsilon,\delta)=T(\beta,\delta)+k(\epsilon).\] 

{\bf For GDxLD,} let $J=\lceil R_v/t_0\rceil$. We also define
\[
\phi_0=\min\{n\geq 0: F(X_n)\leq R_V\}
\]
and for $k=1,2, \dots$,
\[
\phi_{k}=\min\{n\geq 0: X_{\phi_{k-1}+n}=Y^{\prime}_{\phi_{k-1}+n} \mbox{ or } Y_{\phi_{k-1}+n}\in B_0\}.
\]
We next make a few observations. First, for $n\geq \phi_0$, $F(X_n)\leq R_V$. Second, $F(Y_{\phi_k})\leq R_V$ for $k\geq 1$.
Lastly, the value of $F(X_n)$ will decrease by $t_0$ every time swapping takes place. Thus, after $\phi_0$, there can be at most $J$ swapping events. The last observation implies that before time $\phi_{J}$, $Y_n$ must visit $B_0$ at least once. Let 
\[\Xi=\min\{n\geq 0: Y_n\in B_0 \mbox{ or } F(X_n)\leq \epsilon\}.\]
Then,
\[\begin{split}
&\Prob_{0}(\Xi > (J+1)T)\\
\leq& \Prob_{0}\left(\sum_{k=0}^{J}\phi_k>T\right)\\
\leq& \E_{\phi_0}\left[\sum_{k=0}^J\Prob_{\phi_{k-1}}(\phi_k>T)\right] \mbox{ where $\phi_{-1}\equiv 0$}\\
\leq&(J+1) \exp(2(n_0+1)Kh\eta C_V+(n_0+1)K\eta R_V-h\eta C_V T)(V(X_0)+V(Y_0)+2R_V)\\
&+(J+1)\left(1-\pi_{\gamma}(B_0) + \sqrt{h}R_V + 3 \frac{\sqrt{ hd} L}{\sqrt{\beta \gamma}}\right)^K\\
&\mbox{by Lemma \ref{lm:LSI_hit2} and the fact that $V(X_{\phi_k})\leq R_V$, $V(Y_{\phi_k})\leq R_V$.}
\end{split}\]
Let 
\[\tilde K(\delta)=\frac{\log(\delta)-\log(4(J+1))}{\log(1-\pi_{\gamma}(B_0)/2)}\]
and
\[\begin{split}
T(\beta, \delta)=&-\frac{\log (\delta)-\log(4(J+1))}{h\eta C_V}+2(n_0+1)K(\delta)+\frac{(n_0+1)K(\delta)R_V}{hC_V}\\
&+\frac{\log(V(X_0)+V(Y_0)+2R_V)}{h\eta C_V}.
\end{split}\]
For $T>T(\beta,\delta)$, $\Prob(\Xi > (J+1)T)\leq \delta/2$. 
In this case, we can set 
\[
N(\beta,\epsilon,\delta)=(J+1)T(\beta,\delta)+k(\epsilon).
\]

\section{Detailed complexity analysis of (n)SGDxSGLD}
\label{sec:onlineproof}
In this section, we provide the proof of Theorem \ref{thm:SGDstrongconvex}. The proof follows a similar construction as the proof of Theorem \ref{thm:strongconvex}.
However, the stochasticity of $\hat F(x)$ and $\nabla \hat F(x)$ substantially complicates the analysis. 

To facilitate subsequent discussions, we start by introducing some additional notations. 
We denote 
\[
\zeta_n^X=\hat F_n(X_n)- F(X_n) \mbox{ and } \zeta_n^Y=\hat F_n(Y_n)- F(Y_n).
\]
Similarly, we denote 
\[
\xi_n^X=\nabla \hat F_n(X_n) - \nabla F(X_{n}) \mbox{ and } \xi_n^Y=\nabla \hat F_n(Y_n) - \nabla F(Y_{n}).
\]
We also define
\[
\mathcal{F}_n=\sigma\{X_k, Y_k, Z_{k-1}, k\leq n \},
\]
and
\[
\mathcal{G}_n=\mathcal{F}_n \vee \sigma\{X'_{n+1}, Y'_{n+1}\}.
\]
We use $\tilde \Prob_n$ to denote the conditional probability, conditioned on $\mathcal{G}_n$, and $\tilde \E_n$ to denote the corresponding conditional expectation.  

Following the proof of Theorem \ref{thm:strongconvex}, the proof is divided into two steps. We first establish the positive recurrence of some small sets centered around the global minimum. We then establish convergence to the global minimum conditional on being in the properly defined small set.
Without loss of generality, we again assume $x^*=0$ and $F(x^*)=0$.

\subsection{Recurrence of the small set}
Our first result establishes some bounds for the decay rate of $F(X_n)$.
 \begin{lemma}
 \label{lem:SGDdecay}
For (n)SGDxSGLD, under Assumptions \ref{aspt:Lip} and \ref{aspt:noise}, if $Lh\leq 1/2$,
\begin{enumerate}
\item The value of $F(X_n)$ keeps decreasing on average:
 \[
 \E_n[ F(X_{n+1}^{\prime})]\leq F(X_n)-\frac12\|\nabla F(X_n)\|^2h+dLh^2\theta.
\]
If $\theta\leq -\frac{t_0^2}{\log(2Lh^2t_0)}$,
\[
\E_n[F(X_{n+1})]\leq F(X_n)-\frac12\|\nabla F( X_n)\|^2h+2d Lh^2 \theta.
\]

\item Assume also Assumption \ref{aspt:coercive}, for $\hat \eta\leq \min\{(16\gamma)^{-1},(8h\theta)^{-1}\}$, $\hat V(x):=\exp (\hat \eta F(x))$ satisfies the following:
\[
\E_n[F(Y_{n+1})]\leq (1-\frac12 \lambda_ch)F(Y_n)-\frac12\|\nabla F(Y_n)\|^2h +2C_Vh. 
\]
\[
\E_n[\hat V(Y'_{n+1})]\leq \exp\left(-\frac14\hat\eta h \lambda_c F(Y_n)+\hat\eta h\hat C_V\right) \hat V(Y_n),
\]
\[
 \E_n[\hat V(X'_{n+1})]\leq \exp\left(-\frac14\hat\eta h \lambda_c F(X_n)+\hat \eta h\hat C_V\right) \hat V( X_n),
\]
where $\hat C_V=M_c/4+(8\gamma Ld+4Lh\theta d)$.
\item If $\hat{\eta}<\min\left\{(8h\theta)^{-1}, 2t_0/\theta\right\}$, for $\theta\leq -\frac{t_0^2}{\log(\exp(\hat \eta h\theta d)-1)}$,
\[
 \E_n[\hat V(X_{n+1})]\leq \hat V(X_n) \exp\left(\tfrac14 d\right) .
\]
\end{enumerate}
\end{lemma}
\begin{proof}
For Claim 1), by Rolle's theorem, there exits  $x_n$ on the line segment between $X_n$ and $X'_{n+1}$, such that
\begin{align*}
\E_n[F(X'_{n+1})]=&F(X_n)-\E_n\left[\nabla F(x_n)^T(\nabla F(X_n)+\xiX_n)\right]h\\
=&F(X_n)-\E_n\left[\nabla F(X_n)^T(\nabla F(X_n)+\xiX_n)\right]h\\
&+\E_n\left[(\nabla F(x_n)-\nabla F(X_n))^T(\nabla F(X_n)+\xiX_n)\right]h\\
\leq& F(X_n)-\|\nabla F(X_n)\|^2h\\
&+h \left(\E_n\left[\|\nabla F(x_n)-\nabla F(X_n)\|^2\right]\right)^{1/2}\left(\E_n\left[\|\nabla F(X_n)+\xiX_n\|^2 \right]\right)^{1/2}\\
&\mbox{ by H\"older inequality}\\
\leq& F(X_n)-\|\nabla F(X_n)\|^2h+L h^2\E_n\left[ \|\nabla F(X_n)+\xiX_n\|^2\right] \mbox{ by Assumption \ref{aspt:Lip}}\\
\leq& F(X_n)-\frac12\|\nabla F(X_n)\|^2h+d Lh^2\theta \\
&\mbox{ as $Lh<\frac12$ and $\E_n[\|\xi_n^X\|^2]\leq d\theta$ by Assumption \ref{aspt:noise}.}
\end{align*}
For $X_{n+1}$, we first note that when $\|X'_{n+1}\|>\hat M_V$ or $\|Y'_{n+1}\|>\hat M_V$, $F(X_{n+1})=F(X'_{n+1})$.
When $\|X'_{n+1}\|\leq \hat M_V$ and $\|Y'_{n+1}\|\leq \hat M_V$,
we note that if $F(Y'_{n+1}) \leq F(X'_{n+1})$, then $F(X_{n+1})\leq F(X'_{n+1})$.
If $F(Y'_{n+1}) > F(X'_{n+1})$, we may ``accidentally" move $X_{n+1}$ to $Y'_{n+1}$ due to the estimation errors.
In particular,
\begin{align*}
\tilde \E_{n}[F(X_{n+1})]=&F(X'_{n+1})\tilde \Prob_{n}(F_n(X'_{n+1})\leq F_n(Y'_{n+1})+t_0)\\
&+F(Y'_{n+1})\tilde \Prob_{n}(F_n(X'_{n+1})>F_n(Y'_{n+1})+t_0).
\end{align*}
This implies 
\[\begin{split}
&\tilde \E_{n} [F(X_{n+1})]-F(X'_{n+1})\\
=&(F(Y'_{n+1})-F(X'_{n+1}))\tilde \Prob_{n}(\hat F_n(X'_{n+1})>\hat F_n(Y'_{n+1})+t_0)\\
=&(F( Y'_{n+1})-F(X'_{n+1}))\tilde \Prob_{n}(\zeta_{n+1}^X - \zeta_{n+1}^Y>F(Y'_{n+1})-F( X'_{n+1})+t_0)\\
\leq & (F(Y'_{n+1})-F( X'_{n+1}))\exp(-(F(Y'_{n+1})-F(X'_{n+1})+t_0)^2/\theta) \mbox{ by Assumption \ref{aspt:noise}}\\
\leq & \frac{\theta}{2t_0}\exp\left(-\frac{t_0^2}{\theta}\right) \mbox{ as $F(Y'_{n+1}) > F(X'_{n+1})$ and $e^{-x}\leq 1/x$ for $x>0$}\\
\leq & Lh^2\theta \mbox{ as $\theta\leq -\frac{t_0^2}{\log(2Lh^2t_0)}$.}
\end{split}\]
Thus,
$\E_n[ F(X_{n+1})]\leq F(X_n)-\frac12\|\nabla F(X_n)\|^2h+2d Lh^2 \theta$.

For Claim 2), we start by establishing a bound for $F(Y'_{n+1})$. 
Let 
\[
W_n=-\sqrt{h}\xiY_n+\sqrt{2\gamma }Z_n \mbox{ and } \Delta Y_n=Y'_{n+1}-Y_n=-\nabla F(Y_n)h+\sqrt{h}W_n.
\]
By Rolle's theorem, there exits $y_n$ on the line segment between $Y_n$ and $Y'_{n+1}$, such that
\begin{align} \label{eq:bd1}
 F(Y'_{n+1})=&F(Y_n)+\nabla F(y_n)^T\Delta Y_n\nonumber\\
=&F(Y_n)+\nabla F(Y_n)^T\Delta Y_n+(\nabla F(y_n)-\nabla F(Y_n))^T\Delta Y_n\nonumber\\
\leq&  F(Y_n)+\nabla F(Y_n)^T\Delta Y_n+L\|\Delta Y_n\|^2\nonumber\\
&\mbox{ by Cauchy-Schwarz inequality and Assumption \ref{aspt:Lip}}\nonumber\\
\leq& F(Y_n)-\nabla F(Y_n)^T\nabla F( Y_n)h +\sqrt{h}\nabla F( Y_n)^TW_n\nonumber\\
&+L\nabla F(Y_n)^T\nabla F(Y_n)h^2-2\sqrt{h} hL\nabla F( Y_n)^TW_n+LhW_n^TW_n\nonumber\\
\leq& F(Y_n)-\frac12\|\nabla F(Y_n)\|^2h + \beta\sqrt{ h}\nabla F( Y_n)^TW_n+Lh \|W_n\|^2 \mbox{ as $Lh<\frac12$},
\end{align}
where $\beta=1-2hL\in (0,1)$. Taking conditional expectation yields the first estimate. \\
Next, note that for any $0<b<\min\{1/(8\gamma), 1/(4 h\theta)\}$, we have
\begin{align} \label{eq:bd2}
&\E\left[\exp(a^TW_n+b\|W_n\|^2)\right] \nonumber\\
\leq& \E\left[\exp\left(-\sqrt{h}a^T\xiY_n+\sqrt{2\gamma} a^TZ_n+4\gamma b\|Z_n\|^2+2bh\|\xiY_n\|^2\right)\right]\nonumber\\
&\mbox{ by Young's inequality $\|W_n\|^2\leq 4\gamma \|Z_n\|^2+2h \|\xiY_n\|^2$}\nonumber\\
\leq& (1-8\gamma b)^{-d/2}\exp\left(\frac{\gamma\|a\|^2}{1-8\gamma b}\right)(1-4bh\theta)^{-d/2}\exp\left( \frac{h\|a\|^2\theta}{2(1-4bh\theta)}\right)\\
&\mbox{ by Assumption \ref{aspt:noise} and the fact that $Z_n\sim \mathcal{N}(0,I)$.}\nonumber
\end{align}

Then,
\begin{align*}
&\E_n[\hat V(Y'_{n+1})]\\
\leq& \hat V(Y_n)\exp\left(- \frac{\hat \eta h}{2}\|\nabla F(Y_n)\|^2\right)\E_n\left[\exp\left(\hat \eta\beta\sqrt{ h}\nabla F(Y_n)^TW_n+\hat \eta Lh \|W_n\|^2\right)\right]\mbox{ by \eqref{eq:bd1}}\\
= &\hat V(Y_n)\exp\left(- \frac{\hat \eta h}{2}\|\nabla F(Y_n)\|^2\right)(1-8\gamma \hat \eta L h)^{-d/2}\exp\left(\frac{\gamma\hat \eta^2\beta^2h\|\nabla F(Y_n)\|^2}{1-8\gamma\hat  \eta L h}\right)\\
&(1-4\hat \eta Lh^2\theta)^{-d/2}\exp\left( \frac{h\theta\hat \eta^2\beta^2h\|\nabla F(Y_n)\|^2}{2(1-4\hat \eta Lh^2\theta)}\right) \\
&\mbox{ by \eqref{eq:bd2} as $\gamma\hat \eta<1/16$, $\hat \eta h\theta<1/8$ and $Lh<1/2$}\\
\leq &\hat V(Y_n)\exp\left(- \frac{\hat \eta h}{2}\|\nabla F(Y_n)\|^2\right) \exp\left(\frac{\hat \eta h}{4}\|\nabla F(Y_n)\|^2+16\hat \eta\gamma Lh\frac{d}{2}+8\hat \eta Lh^2\theta\frac{d}{2}\right)\\
&\mbox{ as $8\gamma\hat \eta Lh<1/4$, $4\hat \eta L h^2\theta<1/4$ and $\beta<1$}\\
\leq &\hat  V(Y_n)\exp\left(- \frac{\hat \eta h}{4}\|\nabla F(Y_n)\|^2+\hat \eta h(8\gamma Ld+4Lh\theta d)\right)\\
\leq &\hat  V(Y_n)\exp\left(- \frac{\hat \eta h}{4}\lambda_cF(Y_n)+ \frac{\hat \eta h}{4}M_c+\hat \eta h(8\gamma Ld+4Lh\theta d)\right) \mbox{ by Assumption \ref{aspt:coercive}}.
\end{align*}
The upper bound for $\E_n[\hat V(X'_{n+1})]$ can be obtained in a similar way.

Lastly, for Claim 3), we first note following the same argument as \eqref{eq:bd1}, we have
\[
F(X'_{n+1})\leq F(X_n)-\frac12\|\nabla F(X_n)\|^2h -\beta h\nabla F( X_n)^T\xi_n^X+ Lh^2 \|\xi_n^X\|^2. 
\]
Then,
\[\begin{split}
\E_n[\hat V(X'_{n+1})] \leq& \hat V(X_n)\exp\left(-\frac{1}{2}\|\nabla F(X_n)\|^2\hat{\eta} h\right)\\
&(1-2\hat \eta Lh^2\theta)^{-d/2}\exp\left(\frac{\hat\eta^2 \beta^2 h^2\theta \|\nabla F(X_n)\|^2}{2(1-2\hat \eta L h^2\theta)}\right)\\
\leq& \hat V(X_n) \exp\left(4 \hat \eta Lh^2\theta\tfrac{d}{2}\right)\mbox{ as $\hat \eta h \theta<1/8$ and $hL<1/2$}\\
\leq& \hat V(X_n) \exp(\hat \eta h\theta d) \mbox{ as $hL<1/2$.}
\end{split}\]
Next, note that when $\|X'_{n+1}\|>\hat M_V$ or $\|Y'_{n+1}\|>\hat M_V$, $\hat V(X_{n+1})\leq \hat V(X'_{n+1})$.
When $\max\{\|X'_{n+1}\|,\|Y'_{n+1}\|\}\leq \hat M_V$,
if $F(Y_{n+1}^{\prime})\leq F(X_{n+1}^{\prime})$, $\hat V(X_{n+1})\leq \hat V(X_{n+1}^{\prime})$.
If $F(Y_{n+1}^{\prime})>F(X_{n+1}^{\prime})$,  we have

\[\begin{split}
&\tilde \E_n[\hat V(X_{n+1})]\\
=& \hat V(X'_{n+1})\Prob\left(\hat F_n(X'_{n+1})\leq \hat F_n(Y'_{n+1})+t_0\right)
+ \hat V(Y'_{n+1})\Prob\left(\hat F_n(X'_{n+1})> \hat F_n(Y'_{n+1})+t_0\right)\\
\leq& \hat V(X'_{n+1})\left(1+\exp(\hat\eta(F(Y'_{n+1})-F(X'_{n+1})))\exp\left(-\frac{1}{\theta}\left(F(Y'_{n+1})-F(X'_{n+1})+t_0\right)^2\right)\right)\\
& \mbox{ by Assumption \ref{aspt:noise}}\\
\leq& \hat V(X'_{n+1})\left(1+\exp\left(-\tfrac{t_0^2}{\theta}\right)\right) \mbox{ as $\hat \eta<2t_0/\theta$}\\
\leq& \hat V(X'_{n+1})\exp(\hat\eta h\theta d) \mbox{ as $1+\exp(-t_0^2/\theta)<\exp(\hat\eta h\theta d)$}.
\end{split}\]
Thus,
\[
\E_n[\hat V(X_{n+1})] \leq \E_n[\hat V(X'_{n+1})\exp(\hat\eta h\theta d)] \leq \hat V(X_n) \exp(2\hat \eta h\theta d).
\]
\end{proof}

Recall that $\hat R_V=8\lambda_c^{-1}\hat C_V$.
We define a sequence of stopping times:
\[
\hat \tau_0=\min\left\{n\geq 0: F(Y_n)\leq \hat R_V\right\}, 
\]
and for $k=1,2,\dots$,
\[
\hat \tau_k=\min\left\{n>\hat \tau_{k-1}:  F(Y_n)\leq \hat R_V\right\}.
\]
Utilizing the Lyapunov function $\hat V$, our second result establishes bounds for the moment generating function 
of the stopping times.

\begin{lemma}
\label{lem:stoppingtimes2}
For (n)SGDxSGLD, under Assumptions \ref{aspt:Lip}, \ref{aspt:coercive} and \ref{aspt:noise}, if $Lh\leq 1/2$
and $\hat \eta\leq \min\{(16\gamma)^{-1},(8h\theta)^{-1}\}$,
for any $K\geq 0$, the stopping time $\hat \tau_K$ satisfies
\[
\E[\exp(h\hat \eta \hat C_V\hat \tau_K)]\leq \exp(2Kh\hat \eta \hat C_V+K\hat \eta \hat R_V)\hat V(Y_0).
\]
\end{lemma}
The proof of Lemma \ref{lem:stoppingtimes2} follows exactly the same lines of arguments as the proof of Lemma \ref{lem:stoppingtimes}.
We thus omit it here.

Let 
\[\hat D=\max\{\|x-h\nabla F(x)\|: F(x)\leq \hat R_V\}.\]
Following the similar lines of arguments as Lemma \ref{lem:smallset}, we have the following result.

\begin{lemma}
\label{lem:smallset2}
For (n)SGDxSGLD, under Assumption \ref{aspt:noise},
if $F(Y_n)\leq \hat R_V$, then, for any $r>0$, there exists $\hat \alpha(r,\hat D)>0$, such that
\[
\Prob_n(\|Y'_{n+1}\|\leq r)>\hat \alpha(r,\hat D).
\]
In particular,
\[\hat \alpha(r,\hat D)\geq \frac{S_d r^d}{(8\tau h\pi)^{\frac{d}{2}}}\left(1-\frac{4h^2\theta}{r^2}\right)\exp\left(-\frac1{2\tau h}(\hat D^2+ r^2)\right)>0.\]
\end{lemma}
\begin{proof}
\begin{align*}
\Prob_n (\|Y'_{n+1}\|\leq r)&=\Prob_n(\|Y_n-h\nabla F(Y_n)-h\xiY_n+\sqrt{2\gamma h}Z_n\|\leq r)\\
&\geq \Prob_n \left(\|Z_{n}-Q_{n}\|\leq \frac{r}{2\sqrt{2\tau h}}\right) \Prob_n\left(h\|\xiY_n\|\leq \frac{r}{2}\right),
\end{align*}
where
$Q_{n}=-(Y_n-h\nabla F(Y_n))/\sqrt{2\tau h}$.
Note that as $F(Y_n)\leq \hat R_V$,
$\|Q_n\|\leq \hat D/\sqrt{2\tau h}$. Thus,
\begin{align*}
\Prob_n \left(\|Z_{n}-Q_{n}\|\leq \frac{r}{\sqrt{2\tau h}}\right)&=\int_{\|z\|\leq \frac{r}{2\sqrt{2\tau h}}} \frac{1}{(2\pi)^{\frac{d}{2}}}\exp\left(-\frac12\|z+Q_n\|^2\right)dz\\
&\geq \int_{\|z\|\leq \frac{r}{2\sqrt{2\tau h}}} \frac{1}{(2\pi)^{\frac{d}{2}}}\exp\left(-\frac1{2\tau h}(\hat D^2+r^2)\right)dz\\
&\geq  \frac{S_d r^d}{(8\tau h\pi)^{\frac{d}{2}}}\exp\left(-\frac1{2\tau h}(\hat D^2+r^2)\right).
\end{align*}
Lastly, by Markov inequality, 
\[
\Prob_n\left(h\|\xiY_n\|\leq \frac{r}{2}\right)\geq 1-\frac{\E[4h^2 \|\xiY_n\|^2]}{r^2}\geq 1-\frac{4h^2\theta d}{r^2}.
\]
\end{proof}

Our next result shows that if $F(Y'_{n+1})\leq \frac14 r_0$, there is a positive probability that $F(X_{n+1})\leq \tfrac12 r_0$.
\begin{lemma}
\label{lem:SGDexprob}
For (n)SGDxSGLD, under Assumption \ref{aspt:noise}, if $t_0\leq \frac18r_0$, $\theta \leq \tfrac{r_0^2}{64\log 2}$, $F(Y'_{n+1})\leq \tfrac14 r_0$, and $\|X'_{n+1}\|\leq \hat M_V$, then 
\[
\tilde \Prob_{n}\left(F(X_{n+1})\leq \frac12 r_0\right)\geq \frac12.
\]
\end{lemma}
\begin{proof}
Note that if $F(X'_{n+1})\leq \frac12 r_0$, $F(X_{n+1})$ is guaranteed to be less than $\frac12r_0$. 
If $F(X'_{n+1})> \frac12 r_0$, the probability of exchange is
\begin{align*}
&\tilde \Prob_n\left(\hat F_n(X'_{n+1})\geq \hat F_n(Y'_{n+1})+t_0\right)\\
=&\tilde \Prob_n\left(F( X'_{n+1})-F( Y'_{n+1})+\xiX_n-\xiY_n\geq t_0\right)\\
\geq& \tilde \Prob_n\left(\xiX_n-\xiY_n\geq \tfrac18r_0\right)
\mbox{ as $F(X'_{n+1})-F(Y'_{n+1})>\tfrac14r_0$ and $t_0\leq \tfrac18 r_0$}\\
\geq& 1-\exp\left(-\frac{r_0^2}{64\theta}\right) \mbox{ under Assumption \ref{aspt:noise}}\\
\geq& \frac{1}{2} \mbox{ as $\theta \leq \tfrac{r_0^2}{64\log 2}$.}
\end{align*}
If exchange takes place, $X_{n+1}=Y'_{n+1}$ and $F(X_{n+1})\leq \frac14r_0$. 
\end{proof}

\subsection{Convergence to global minimum}
In this subsection, we analyze the ``speed" of convergence $\{X_{n+k}: k\geq 0\}$ to $x^*$ when $F(X_{n})\leq \frac12r_0$.
Let 
\[
\kappa_n=\inf\{k>0: F(X_{n+k})>r_0\}.
\]
\begin{lemma}
\label{lem:badprob}
For (n)SGDxSGLD, under Assumption \ref{aspt:Lip} and \ref{aspt:noise}, and assuming $Lh\leq 1/2$ and $\hat\eta<(8h\theta)^{-1}$,
if $F(X_n)\leq \frac12 r_0$, then for any fixed $k>0$, 
\[
\Prob_n(\kappa_n> k)\geq 1- \exp\left(-\frac{r_0}{16h\theta}+\frac{1}{4}dk\right).
\]
\end{lemma}
\begin{proof}
From Lemma \ref{lem:SGDdecay}, the following is a supermartingale
\[
\hat V(X_{n+k})\exp\left(-\frac14 d k\right)1_{\kappa_n\geq k}.
\]
In particular,
\[
\E_{n+k}\left[\hat V(X_{n+k+1})]\exp\left(-\frac14d (k+1)\right)1_{\kappa_n\geq k+1}\right]
\leq \hat V(X_{n+k})\exp\left(-\frac14 d k\right)1_{\kappa_n\geq k}.
\]
Therefore,
\[
\E_n[\hat V(X_{n+(\kappa_n\wedge k)})\exp(-\tfrac14 d(\kappa_n\wedge k))]\leq \hat V(X_n) \leq \exp\left(\frac12\hat \eta r_0\right).
\]
We also note
\begin{align*}
\E_n\left[\hat V(X_{n+(\kappa_n\wedge k)})]\exp\left(-\frac14 d(\kappa_n\wedge k)\right)\right]
\geq \E_n\left[\exp(\hat \eta r_0)\exp\left(-\frac14 d k\right)1_{\kappa_n \leq k}\right].
\end{align*}
Then,
\[
\Prob_n(\kappa_n\leq k)\leq \exp\left(-\frac12\hat \eta r_0+\frac14d  k\right)
\leq \exp\left(-\frac{r_0}{16h\theta}+\frac{1}{4}dk\right),
\]
since $8\hat\eta h \theta<1$.
\end{proof}

\begin{lemma}
\label{lem:strongconvSGD}
For (n)SGDxSGLD, under Assumptions \ref{aspt:Lip} and \ref{aspt:noise}, and assuming 
$F$ is strongly convex in $B_0$ and $h\leq \min\{1/(2L),1/m\}$, if $F(X_n)\leq \frac12 r_0$,
\[
\E_n[F(X_{n+k})1_{\kappa_n>k}]\leq (1-m h)^kF(X_n)+\frac{d\theta}{m},
\] 
for all $k\geq 0$. 
\end{lemma}
\begin{proof}
We first note from Lemma \ref{lem:SGDdecay}, if $F(X_n)\leq r_0$,
\begin{align*}
\E_{n+j}[F(X_{n+j+1})]&\leq F(X_{n+j})-\frac12 \|\nabla F(X_{n+j})\|^2h +2dLh^2\theta\nonumber\\
&\leq (1-mh)F(X_{n+j})+2dLh^2\theta,
\end{align*}
where the second inequality follows from \eqref{eq:strong_conv_bound} as $F(x)$ is strongly convex in $B_0$. 

Next, we note 
\begin{align*}
\E_n[F(X_{n+(\kappa_n\wedge k)})]&\leq \E_n\left[ (1-mh)^{(\kappa_n\wedge k)-1} F(X_n)+\sum_{j=1}^{\kappa_n\wedge k} (1-mh)^{(\kappa_n\wedge k)-j}dLh^2\theta\right]\\
&\leq F(X_n)+\frac{2dLh\theta}{m} \leq F(X_n)+\frac{d\theta}{m} \mbox{ as $Lh<1/2$}.
\end{align*}
Because $\E_n[F(X_{n+(\kappa_n\wedge k)})]>\E_n[F(X_{n+k})1_{\kappa_n> k}]$,
\[
\E_n[F(X_{n+k})1_{\kappa_n>k}]\leq (1-mh)^kF( X_n)+\frac{d\theta}{m}. 
\]
\end{proof}

\subsection{Proof of Theorem \ref{thm:SGDstrongconvex}} \label{proof:thm4}

For any fixed accuracy $\epsilon>0$ and $\delta$, we set
\[
K(\delta)=\frac{\log(\delta/3)}{\log(1-\hat \alpha(r_0/4,\hat D))}=O(-\log(\delta)),
\]
\[
k(\epsilon,\delta)=\frac{\log(2\delta\epsilon/(9r_0))}{\log(1-mh)}=O(-\log(\delta)-\log(\epsilon)),
\]
and
\[
N(\epsilon,\delta)=k(\epsilon,\delta)+\frac{2K(\delta)h\hat \eta v C_V + K(\delta)\hat \eta \hat  R_V+\log \hat V(Y_0) - \log(\delta/3)}{h\hat\eta \hat C_V}=O(-\log(\delta)-\log(\epsilon)).
\]
For any fixed $N>N(\epsilon,\delta)$, we set
\[\begin{split}
&\theta(N,\epsilon,\delta)\\
\leq&\min\Bigg\{\frac{\delta\epsilon m}{9d},\frac{r_0}{16h(dN/4 - \log(\delta/9))}, -\frac{t_0^2}{\log(2Lh^2t_0)},
-\frac{t_0^2}{\log(\exp(d/8)-1)}, \frac{r_0^2}{64\log 2}\Bigg\}\\
=& O(\min\{N^{-1},\epsilon\delta\}).
\end{split}\]

Now for fixed $N>N(\epsilon,\delta)$ and $\theta\leq \theta(N,\epsilon,\delta)$,
we first note if $F(X_n)\leq \frac12r_0$ for $n\leq N-k(\epsilon,\delta)$, 
\[\begin{split}
\Prob_n(F(X_{N})>\epsilon)=&\Prob_n(F( X_{N})>\epsilon, \kappa_n>N-n)+\Prob_n(F(X_{n+k})>\epsilon, \kappa_n\leq N-n)\\
\leq& \Prob_n(F( X_{N})1_{\kappa_n>N-n}>\epsilon)+\Prob_n(\kappa_n\leq N-n)\\
\leq& \frac{1}{\epsilon}\left((1-mh)^{N-n}F(X_n)+\frac{d\theta}{m}\right) + \exp\left(-\frac{r_0}{16h\theta}+\frac{1}{4}dN\right) \\
&\mbox{ by Markov inequality, Lemma \ref{lem:strongconvSGD}, and Lemma \ref{lem:badprob}}\\
\leq& \frac{1}{\epsilon}\left((1-mh)^{k(\epsilon,\delta)}\frac{r_0}{2}+\frac{d\theta}{m}\right) + \exp\left(-\frac{r_0}{16h\theta}+\frac{1}{4}dN\right)\\
\leq& \frac{1}{3}\delta \mbox{ by our choice of $\theta(N,\epsilon,\delta)$ and $k(\epsilon,\delta)$}.
\end{split}\]

Next, we study how long it takes $X_n$ to visit the set $\{x: F(x)\leq r_0/2\}$. In particular, we denote $T=\inf\{n: F(X_n)\leq r_0/2\}$.
From Lemma \ref{lem:smallset2} and \ref{lem:SGDexprob}, every time $Y_n\in \{x: F(x)\leq \hat R_V\}$,
\[
\Prob_n\left(F(X_{n+1})\leq \frac12r_0\right)\geq \frac12\Prob_n\left(\|Y_{n+1}^{\prime}\|\leq \frac14 r_0\right)\geq \frac12\hat \alpha(r_0/4, \hat D)>0.
\]
Then, 
\[\begin{split}
\Prob(F(X_{\hat \tau_k+1}) \mbox{ for $k=1,\dots,K(\delta)$})&=\E\left[\prod_{k=1}^{K(\delta)}\Prob_{\hat \tau_k}\left(F(X_{\hat \tau_k+1})>\frac12r_0\right)\right]\\
&\leq \left(1-\frac12\hat \alpha(r_0/4,\hat D)\right)^{K(\delta)}<\frac{\delta}{3}.
\end{split}\]
From Lemma \ref{lem:stoppingtimes2}, by Markov inequality,
\[\begin{split}
&\Prob(\hat \tau_{K(\delta)}>N-k(\epsilon,\delta))\\
&\leq \frac{\E[\exp(h\hat \eta \hat C_V \hat \tau_{K(\delta)})]}{\exp(h\hat \eta \hat C_V (N-k(\epsilon,\delta)))}\\
&\leq \frac{\exp(2K(\delta)h\hat \eta v C_V + K(\delta)\hat \eta \hat  R_V)V(Y_0)}{\exp(h\hat \eta \hat  C_V (N-k(\epsilon,\delta)))}\\
&\leq \frac{\delta}{3} \mbox{ since } N-k(\epsilon,\delta)>\frac{2K(\delta)h\hat \eta v C_V + K(\delta)\hat \eta \hat  R_V+\log V(Y_0) - \log\delta +\log 3}{h\hat\eta \hat C_V}.
\end{split}\]
Then,
\[\begin{split}
\Prob(T\leq N-k(\epsilon,\delta))
&\geq \Prob\left(\hat \tau_{K(\delta)}\leq N-k(\epsilon,\delta) \mbox{ and } F(\hat X_{\hat \tau_k+1}) \mbox{ for some $k=1,\dots,K(\delta)$}\right)\\
&\geq \left(1-\frac{\delta}{3}\right) + \left(1-\frac{\delta}{3}\right)-1 = 1-\frac{2\delta}{3}.
\end{split}\]
Lastly,
\[\begin{split}
\Prob(F(X_N)\leq \epsilon)&\geq \Prob(F(X_N)\leq \epsilon, T\leq N-k(\epsilon,\delta))\\
&\geq \E[\Prob_T(F(X_N)\leq \epsilon)|T\leq N-k(\epsilon,\delta)]\Prob(T\leq N-k(\epsilon,\delta))\\
&\geq \left(1-\frac{2\delta}{3}\right) \left(1-\frac{\delta}{3}\right) \geq 1-\delta.
\end{split}\]
%


\subsection{Proof of Theorem \ref{thm:LSIonline}} \label{app:thm5}
Let 
\[
\hat Y_{n+1}=\hat Y_n - \nabla F(\hat Y_n)h-\xi^Y_n h +\sqrt{2\gamma h}Z_n.
\]
This is to be differentiated with $Y_n$ in Algorithm \ref{alg:SGDxSGLD}, which can swap position with $X_n$. 
We denote $\hat \mu_n$ as the distribution of $\hat Y_{n}$. 

\begin{lemma} \label{lm:LSI_on}
Under Assumptions \ref{aspt:Lip}, \ref{aspt:LSI}, and \ref{aspt:noise},  for $h<\frac{\beta}{4\sqrt{2}L^2}$,
\[\KL(\hat \mu_n|\pi_{\gamma})\leq  \exp(-\tfrac{1}{2}\beta \gamma h(n-1))\left(\frac{1}{\gamma}F(\hat Y_0) + \hat M_d\right)
+\frac{6 hL^2d+\theta d/\gamma}{\beta \gamma},\]
where $\hat M_d=2 Lh d-\tfrac{d}{2}(\log (4\pi h)+1)+\log (U_\gamma)$.
\end{lemma}
\begin{proof}
For a given realization of $\xi_0^Y$ and $\hat Y_0$, 
\[
\hat Y_1\sim \mathcal{N}(m(\hat Y_0),2\gamma h I_d ),\mbox{ where } m(\hat Y_0)=\hat Y_0-\nabla F(\hat Y_0)h - \xi_0^Yh.
\]
Then, given $\tilde Y_0$,
\begin{equation}\label{eq:mu1_bd}
\begin{split}
\KL(\hat \mu_1\|\pi_{\gamma})&=\int \log\left(\frac{\hat \mu_1(y)}{\pi_{\gamma}(y)}\right)\hat\mu_1(y)dy\\
&=\int \left(-\frac{d}{2}\log(4\pi\gamma h) - \frac{\|y-m(\hat Y_0)\|^2}{4\gamma h} + \log U_{\gamma} + \frac{F(y)}{\gamma}\right)\hat\mu_1(y)dy\\
&=-\frac{d}{2}(\log(4\pi\gamma h)+1) + \log U_{\gamma} + \frac{1}{\gamma}\E[F(\hat Y_1)|\hat Y_0]\\
&\leq -\frac{d}{2}(\log(4\pi\gamma h)+1) + \log U_{\gamma} + \frac{1}{\gamma}(F(\hat Y_0)+2Lh\|\xi_0^Y\|^2+2\gamma Lhd)\\
&=\frac{1}{\gamma}F(\hat Y_0)+\hat M_d + \frac{2Lh}{\gamma}\|\xi_0^Y\|^2,
\end{split}\end{equation}
where the last inequality follows from \eqref{eq:bd1}.

Next, following the proof of Lemma 3 in \cite{vempala2019rapid}, for a given $\xi_1^Y$, with slight abuse of notation, consider the modified diffusion process:
\[d Y_t = -\nabla F(\hat Y_1) dt - \xi_1^Y dt + \sqrt{2\gamma}dB_t,\quad Y_1=\hat{Y}_t\]
for $t\geq 1$. Note that $\hat Y_2$ follows the same law as $Y_{1+h}$. 
Let $\mu_t$ denote the distribution of $Y_t$. Then
\[\begin{split}
&\frac{d}{dt}\KL(\mu_t\|\pi_{\gamma})\\
=&-\gamma \int \left\|\nabla\log\frac{\mu_t(x)}{\pi_{\gamma}(x)}\right\|^2\mu_t(x) dx 
+ \E\left[\left\langle\nabla F(Y_t) - \xi_1^Y-\nabla F(\hat Y_1), \nabla \log\frac{\mu_t(Y_t)}{\pi_{\gamma}(Y_t)} \right\rangle\right]\\
\leq&-\frac{3\gamma}{4}\int \left\|\nabla\log\frac{\mu_t(x)}{\pi_{\gamma}(x)}\right\|^2\mu_t(x) dx 
+ \frac1\gamma\E\left[\left\|\nabla F(Y_t) - \xi_1^Y-\nabla F(\hat Y_1)\right\|^2\right]\\
\leq&-\frac{3\gamma}{4}\int \left\|\nabla\log\frac{\mu_t(x)}{\pi_{\gamma}(x)}\right\|^2\mu_t(x) dx
+\frac2\gamma L^2\E[\|Y_t - \hat Y_1\|^2]+\frac 2\gamma\|\xi_1^Y\|^2.
\end{split}\]
Because $Y_t \overset{d}{=} \hat Y_1 - (\nabla F(\hat Y_1) + \xi_1^Y)(t-1) + \sqrt{2\gamma (t-1)}Z_1$, 
where $Z_1$ is a standard $d$-dimensional Gaussian random vector,
\[\begin{split}
\E[\|Y_t - \hat Y_1\|^2]&=2(t-1)^2\E[\|\nabla F(\hat Y_1)\|^2] + 2(t-1)^2\|\xi_1^Y\|^2 + 2\gamma (t-1)d\\
&\leq\frac{8(t-1)^2L^2}{\beta}\KL(\hat\mu_1\|\pi_{\gamma})+4(t-1)^2dL+2(t-1)^2\|\xi_1^Y\|^2+2\gamma (t-1)d.
\end{split}\]
In addition, because,
\[\int \left\|\nabla\log\frac{\mu_t(x)}{\pi_{\gamma}(x)}\right\|^2\mu_t(x) dx \geq 2\beta \KL(\mu_t(x)\|\pi_{\gamma}),\]
we have
\[\begin{split}
&\frac{d}{dt}\KL(\mu_t\|\pi_{\gamma})\\
\leq& -\frac{3\beta\gamma}{2}\KL(\mu_t\|\pi_{\gamma}) \\
&+ 2\frac{L^2}{\gamma}\left(\frac{8(t-1)^2L^2}{\beta}\KL(\hat\mu_1\|\pi_{\gamma})+4(t-1)^2dL+2(t-1)^2\|\xi_1^Y\|^2+2 \gamma(t-1)d \right)
+ \frac2\gamma\|\xi_1^Y\|^2\\
\leq& -\frac{3\beta\gamma}{2}\KL(\mu_t\|\pi_{\gamma})
+\frac{16(t-1)^2L^4}{\beta\gamma}\KL(\mu_1\|\pi_{\gamma})
+\frac{1}{\gamma}(4(t-1)^2 L^2+2)\|\xi_1^Y\|^2\\
& + 8(t-1)^2d\frac{L^3}{\gamma}+4 (t-1)d L^2,
\end{split}\]
which further implies that 
\[\KL(\mu_{1+h}\|\pi_{\gamma}) \leq
e^{-\tfrac{3}{2}\beta \gamma h}\KL(\mu_1\|\pi_{\gamma})
+\frac{16h^3L^4}{\beta\gamma}\KL(\mu_1\|\pi_{\gamma})
+\frac{4h^3L^2+2h}{\gamma}\|\xi_1^Y\|^2 + \frac{8h^3dL^3}{\gamma}+4 h^2d L^2\]
For $h<\frac{\beta\gamma}{4\sqrt{2}L^2}$, we have
\[
\KL(\mu_{1+h}\|\pi_{\gamma}) \leq 
\left(1-\beta h\gamma\right)\KL(\mu_1\|\pi_{\gamma}) 
+ \frac{3h}{\gamma}\|\xi_1^Y\|^2+ 6 h^2L^2d.
\]
The above analysis implies that given $\xi_k^Y$'s,
\[\begin{split}
\KL(\hat \mu_{n}\|\pi_{\gamma}) &\leq 
\left(1-\beta h\gamma\right)\KL(\hat \mu_{n-1}\|\pi_{\gamma}) 
+ \frac{3h}{\gamma}\|\xi_1^Y\|^2+ 6 h^2L^2d\\
&\leq \exp(-\tfrac{1}{2}\beta h(n-1)\gamma)\KL(\hat \mu_{1}\|\pi_{\gamma}) 
+  \frac{6 hL^2d}{\beta\gamma} + \frac{3h}{\gamma}\sum_{k=1}^{n-1}(1-h\beta\gamma)^{n-k}\|\xi_k^Y\|^2. 
\end{split}\]
Plug in the bound for $\KL(\hat \mu_1\|\pi_{\gamma})$ in \eqref{eq:mu1_bd} and taking the expectation with respect to $\xi_k^Y$'s, we have
\[\begin{split}
\KL(\hat \mu_{n}\|\pi_{\gamma}) 
&\leq \exp(-\tfrac{1}{2}\beta h(n-1)\gamma)\left(\frac{1}{\gamma}F(\hat Y_0) + \hat M_d\right)
+\frac{6hL^2d}{\beta\gamma}+\frac{1}{\gamma}h\sum_{k=0}^{n-1}(1-h\beta\gamma)^{n-k}\theta d\\
&\leq \exp(-\tfrac{1}{2}\beta h(n-1)\gamma)\left(\frac{1}{\gamma}F(\hat Y_0) + \hat M_d\right)
+\frac{6hL^2d+\theta d/\gamma}{\beta\gamma}.
\end{split}\]
\end{proof}

Based on Lemma \ref{lm:LSI_on}, let
\[n_0=\frac{4}{\beta\gamma} h^{-1}\log(1/h)+1.\]
For $n\geq n_0$ and $h$ small enough,
\begin{equation}\label{eq:LSI_bound_on}
\begin{split}
\KL(\hat\mu_n\|\pi_{\gamma})&\leq  h^2\left(\frac1\gamma F(\hat Y_0)+\hat M_d\right)+\frac{6hL^2d+\theta d/\gamma }{\beta\gamma}\\
&\leq hF(\hat Y_0) + \frac{8hL^2d+\theta d/\gamma }{\beta\gamma}
\end{split}
\end{equation}

Recall that
\[\hat B_0=\{x: F(x)\leq r_0/4\}.\]
We next draw connection between the bound \eqref{eq:LSI_bound_on} and the hitting time of $Y_n$ to $\hat B_0$.
For nSGDxSGLD, let $\hat \phi=\min_n \{Y_n\in \hat B_0 \}$. With a slight abuse of notation, for SGDxSGLD, let $\hat \phi=\min_n \{X_{n}=Y'_n \text{ or } Y_n\in \hat B_0 \}$.

\begin{lemma} \label{lm:LSI_hit_on}
For (n)SGDxSGLD, under Assumptions \ref{aspt:Lip}, \ref{aspt:coercive}, \ref{aspt:LSI}, and \ref{aspt:noise}, 
\[
\Prob(\hat \phi\leq n_0 )\geq \pi_{\gamma}(\hat B_0) - \sqrt{h} F(\hat Y_0) - 2\sqrt{h}\frac{\sqrt{ d}L}{\sqrt{\beta\gamma}}-\frac{\sqrt{\theta d}}{\sqrt{\beta}\gamma}.
\]
\end{lemma}
\begin{proof}
For both nSGDxSGLD and SGDxSGLD, 
\[\Prob(\hat \phi\leq n) \geq  \Prob(\hat Y_n\in \hat{B}_0).\]
By Pinsker's inequality,
\[
d_{tv}(\hat \mu_n, \pi_{\gamma})\leq \sqrt{\frac{1}{2}\KL(\hat\mu_n\|\pi_{\gamma})}
\leq \sqrt{h} F(\hat Y_0) + 2\sqrt{h}\frac{\sqrt{ d}L}{\sqrt{\beta\gamma}}+\frac{\sqrt{\theta d}}{\sqrt{\beta}\gamma}
\]
where the last inequality follows from \eqref{eq:LSI_bound_on}.
Then
\[
\Prob(\hat Y_{n_0}\in \hat B_0)\leq \pi_{\gamma}(\hat B_0) -\sqrt{h} F(\hat Y_0) - 2\sqrt{h}\frac{\sqrt{ d}L}{\sqrt{\beta\gamma}}-\frac{\sqrt{\theta d}}{\sqrt{\beta}\gamma}
\]
\end{proof}

\begin{lemma} \label{lm:LSI_hit2_on}
For (n)GDxLD, fix any $K$ and $T\geq ( n_0+1)K$,
\[\begin{split}
\Prob_0(\hat \phi> T)\leq&\exp(2( n_0+1)Kh\hat \eta \hat C_V+( n_0+1)K\hat \eta \hat R_V-h\hat \eta \hat C_V T)\hat V(Y_0)\\
&+\left(1-\pi_{\gamma}(\hat B_0) + \sqrt{h} \hat R_V + 2\sqrt{h}\frac{\sqrt{d}L}{\sqrt{\beta\gamma}} + \frac{\sqrt{\theta d}}{\sqrt{\beta}\gamma}\right)^K.
\end{split}\]
\end{lemma}
\begin{proof}
Recall the sequence of stoping times
$\hat \tau_j=\inf\{n>\tau_{j-1}: F(Y_n)\leq \hat R_V\}$.
Applying Lemma \ref{lem:stoppingtimes2}, we have
\[
\E[\exp(h\hat \eta \hat C_V\hat \tau_k)]\leq \exp(2kh\hat \eta \hat C_V+k\hat \eta \hat R_V)\hat V(Y_0).
\]

We next define a new sequence of stopping times:
\[
\hat \psi_0=\inf\left\{n: F(Y_n)\leq \hat R_V\right\}, ~~~\hat\psi'_0=\hat\psi_0+\hat n_0,
\]
and for $k=1, \dots$,
\[
\hat\psi_k=\inf\left\{n\geq \hat\psi'_{k-1}+1, F(Y_n)\leq \hat R_V\right\}, ~~~
\hat\psi'_k=\hat\psi_k+\hat n_0+1.
\]
Note that $\hat\psi_i$ always coincide with one of $\hat\tau_j$'s, and as $\hat\tau_{j}-\hat\tau_{j-1}\geq 1$, 
\[\hat\psi_{k}\leq \hat\tau_{( n_0+1)k}.\]
Thus,
\begin{align*}
\Prob_0(\hat\phi\geq T)\leq & \Prob_0\left(\hat \tau_{(n_0+1)K}\geq T\right)+\Prob_0\left(Y_{\psi'_k}\notin B_0,\forall k\leq K\right)\\
\leq& \exp(2( n_0+1)Kh\hat \eta \hat C_V+( n_0+1)K\hat \eta \hat R_V-h\hat \eta \hat C_V T)\hat V(Y_0)\\
&+\left(1-\pi_{\gamma}(\hat B_0) + \sqrt{h} \hat R_V + 2\sqrt{h}\frac{\sqrt{d}L}{\sqrt{\beta\gamma}} + \frac{\sqrt{\theta d}}{\sqrt{\beta}\gamma}\right)^K \mbox{ by Lemma \ref{lm:LSI_hit_on}.}
\end{align*}
\end{proof}

\noindent{\bf We are now ready to prove Theorem \ref{thm:LSIonline}.}
For a fixed $N$, let $A_N$ denote the event
\[
\left\{\xi_n^X-\xi_n^Y>\frac{1}{2} t_0 \mbox{ for some $n\leq N$}\right\}
\]
Note that
\[
\Prob(A_N)\leq \sum_{n=1}^N \Prob\left(\zeta^X_n-\zeta^Y_n>\frac12 t_0\right)\leq N\exp\left(-\frac{t_0^2}{\theta}\right).
\]

Next, set
\[
h < \left(\frac{\pi_{\gamma}(\hat B_0)}{4\hat{R}_V+8\sqrt{d} L/\sqrt{\beta\gamma}}\right)^2
\mbox{ and }
\theta< \frac{\pi_{\gamma}(\hat B_0)^2\beta\gamma^2}{16 d}.
\]
Then, 
\[1-\pi_{\gamma}(\hat B_0) + \sqrt{h} \hat R_V + 2\sqrt{h}\frac{\sqrt{ d}L}{\sqrt{\beta\gamma}} + \frac{\sqrt{\theta d}}{\sqrt{\beta}\gamma}
\leq 1-\pi_{\gamma}(\hat B_0)/2\]

Following the proof of Theorem \ref{thm:SGDstrongconvex}, we also set
\[k(\epsilon,\delta)=\frac{\log(2\delta\epsilon/(9r_0))}{\log(1-mh)}\]
and
\[\theta(N,\delta)=\min\left\{\frac{1}{t_0^2}\log\left(\frac{\delta}{3N}\right), \frac{\delta\epsilon m}{9d}, \frac{r_0}{16h(dN/4 - \log(\delta/9))}, \frac{\pi_{\gamma}(\hat B_0)^2\beta\gamma^2}{16 d}\right\}.\]
In what follows we assume $\theta\leq \theta(N,\delta)$.

For nSGDxSGLD, define $J=1$,
\[
\hat \phi_{0}=\min \{n\geq 0: F(Y_{n})\in \hat B_0\},
\]
and for $k\geq 1$,
\[
\hat \phi_{k}=\min \{n>0: F(Y_{\phi_k+n})\in \hat B_0\}.
\]
For SGDxSGLD, with a slight abuse of notation, define $J=\hat R_v/t_0$,
\[
\hat \phi_0=\min \{n\geq 0:F(X_n)\leq \hat R_V\},
\]
and for $k\geq 1$,
\[
\hat \phi_{k}=\min \{n: X_{\phi_{k-1}+n}=Y'_{\phi_{k-1}+n} \text{ or } Y_{\phi_{k-1}+n}\in \hat B_0\}.
\]

We next note that
\begin{equation}
\label{eq:LSI_decomp}
\begin{split}
\Prob(F(X_N)>\epsilon)\leq& \Prob(A_N)+\Prob(F(X_N)>\epsilon, \phi_J<N, A_N^c)\\
&+\Prob(\phi_J>N, A_N^c)
\end{split}\end{equation}
We shall look at the three terms on the left hand side of \eqref{eq:LSI_decomp} one by one.
First, with our choice of $\theta$, $\Prob(A_N)\leq\delta/3$.

Second, suppose there exist $n<N-k(\epsilon,\delta)$ such that $F(X_n)<r_0/2$.
Let $\kappa_n=\inf\{k>0: F(X_n)>r_0\}$. Then,
\[\begin{split}
\Prob_n(F(X_{N})>\epsilon)=&\Prob_n(F( X_{N})>\epsilon, \kappa_n>N-n)+\Prob_n(F(X_{N})>\epsilon, \kappa_n\leq N-n)\\
\leq& \Prob_n(F( X_{N})1_{\kappa_n>N-n}>\epsilon)+\Prob_n(\kappa_n\leq N-n)\\
\leq& \frac{1}{\epsilon}\left((1-mh)^{N-n}F(X_n)+\frac{d\theta}{m}\right) + \exp\left(-\frac{r_0}{16h\theta}+\frac{1}{4}dN\right) \\
&\mbox{ by Markov inequality, Lemma \ref{lem:strongconvSGD}, and Lemma \ref{lem:badprob}}\\
\leq& \frac{1}{\epsilon}\left((1-mh)^{k(\epsilon,\delta)}\frac{r_0}{2}+\frac{d\theta}{m}\right) + \exp\left(-\frac{r_0}{16h\theta}+\frac{1}{4}dN\right) \\
\leq& \frac{1}{3}\delta ~~\mbox{ by our choice of $\theta$ and $k(\epsilon,\delta)$}.
\end{split}\]
In addition, from Lemma \ref{lem:SGDexprob},
\[
\Prob\left(F(X_{n})\leq \frac12r_0\Big |Y_n^{\prime}\in \hat B_0\right)\geq \frac12
\]
Thus every time $Y_n^{\prime}\in \hat B_0$, there is a chance larger than or equal to $\frac12$ for $X_n$ to visit $\{x: F(x)\leq r_0/2\}$.

Lastly, we study how long it takes $Y_n$ to visit $\hat B_0$.
Conditional on $A_N^c$ and $\phi_J<N$, there exists at least one $\phi_k$, $k\leq J$, such that
$Y_{\phi_k}\in \hat B_0$.  
In addition, at all $\phi_k$, $1\leq k\leq J$, $F(Y_{\phi_k})\leq \hat R_V+t_0$. 
Let $\phi_{-1}\equiv 0$. Then, 
for $N>(J+1)T$
\[\begin{split}
&\Prob_0(\phi_J\geq (J+1)T|A^c_N)\\
\leq& \E_0\left[\sum_{k=0}^J\Prob_{\phi_{k-1}}\left(\phi_{k+1}>T|A^c_N\right)\right]\\
\leq& (J+1) \exp\left(2( n_0+1)Kh\hat \eta \hat C_V+( n_0+1)K\hat \eta \hat R_V-h\hat \eta \hat C_V T\right)(\hat V(Y_0)+\hat{R}_V+t_0)\\
&+(J+1)\left(1-\pi_{\gamma}(\hat B_0) + \sqrt{h} \hat R_V + 2\sqrt{h}\frac{\sqrt{d}L}{\sqrt{\beta\gamma}} + \frac{\sqrt{\theta d}}{\sqrt{\beta}\gamma}\right)^K
\mbox{ by Lemma \ref{lm:LSI_hit2_on}}\\
\leq&(J+1) \exp\left(2( n_0+1)Kh\hat \eta \hat C_V+( n_0+1)K\hat \eta \hat R_V-h\hat \eta \hat C_V T\right)(\hat V(Y_0)+\hat{R}_V+t_0)\\
&+(J+1)\left(1-\pi_{\gamma}(\hat B_0)/2\right)^K \mbox{ by our choice of $h$ and $\theta$.}
\end{split}\]

Set 
\[K(\delta)=\frac{\log(\delta)-\log(6(J+1))}{\log(1-\pi_{\gamma}(\hat B_0)/2)}\]
and
\[\begin{split}
T(\beta, \delta)=&\frac{2(n_0+1)K(\delta)h\hat \eta \hat C_V + (n_0+1)K(\delta)\hat \eta \hat  R_V}{h\hat\eta \hat C_V}\\
&+\frac{\log \left(\hat V(Y_0)+\hat{R}_V+t_0\right) - \log(\delta)+\log(6(J+1))}{h\hat\eta \hat C_V}.
\end{split}\]
We have for $K>K(\delta)$ and $T>T(\delta,\beta)$, $\Prob(\phi_J\geq T,A^c_N)<\delta/3$.

Above all, if we set 
\[N(\beta,\epsilon,\delta)=T(\beta, \delta)+k(\delta,\epsilon)=O(\log(1/\delta)/\beta)+O(1/\epsilon),\]
we have for $N\geq N(\beta,\epsilon,\delta)$,
$\Prob(F(X_N)>\epsilon)\leq \delta$.

\bibliographystyle{plain}
\bibliography{GD_SGDbib.bib}

\end{document}